\pdfoutput=1
\RequirePackage[l2tabu, orthodox]{nag}

\documentclass[reqno]{amsart}

\usepackage[utf8]{inputenc}
\usepackage[english]{babel}
\usepackage{microtype} 

\usepackage{amsmath,amssymb,amsthm,mathrsfs,latexsym,mathtools,
mathdots,enumerate,bm,tikz,url,booktabs,comment,xparse,environ}
\usepackage{cases} 
\usepackage{placeins} 
\usepackage[export]{adjustbox} 
\usepackage[centertableaux]{ytableau}

\usepackage{hyperref}

\def\lineThickness{.7}
\usetikzlibrary{decorations.pathmorphing,shapes}

\usepackage{todonotes}
\presetkeys%
    {todonotes}%
    {inline,backgroundcolor=yellow}{}

\definecolor{lightblue}{rgb}{0.8,0.8,1.0}

\NewDocumentEnvironment{mymathenvironment}{m m O{}}{
    \begin{#1}[#3]\label{#1:#2}
}{
    \end{#1}
}

\NewEnviron{eq*}{%
\begin{equation*}\begin{split}
  \BODY
\end{split}\end{equation*}
}
\NewEnviron{eq}[1]{%
\begin{equation}\label{Equation:#1}\begin{split}
  \BODY
\end{split}\end{equation}
}

\NewDocumentEnvironment{thm}{m}{\begin{mymathenvironment}{Theorem}{#1}}{\end{mymathenvironment}}
\NewDocumentEnvironment{cor}{m}{\begin{mymathenvironment}{Corollary}{#1}}{\end{mymathenvironment}}
\NewDocumentEnvironment{conj}{m}{\begin{mymathenvironment}{Conjecture}{#1}}{\end{mymathenvironment}}
\NewDocumentEnvironment{definition}{m}{\begin{mymathenvironment}{Definition}{#1}}{\end{mymathenvironment}}
\NewDocumentEnvironment{mydef}{m}{\begin{mymathenvironment}{Definition}{#1}}{\end{mymathenvironment}}
\NewDocumentEnvironment{exa}{m}{\begin{mymathenvironment}{Example}{#1}}{\end{mymathenvironment}}
\NewDocumentEnvironment{exercise}{m}{\begin{mymathenvironment}{Exercise}{#1}}{\end{mymathenvironment}}
\NewDocumentEnvironment{lem}{m}{\begin{mymathenvironment}{Lemma}{#1}}{\end{mymathenvironment}}
\NewDocumentEnvironment{prob}{m}{\begin{mymathenvironment}{Problem}{#1}}{\end{mymathenvironment}}
\NewDocumentEnvironment{prop}{m}{\begin{mymathenvironment}{Proposition}{#1}}{\end{mymathenvironment}}
\NewDocumentEnvironment{rem}{m}{\begin{mymathenvironment}{Remark}{#1}}{\end{mymathenvironment}}

\newcommand{\myi}[1]{\item\label{Item:#1}}

\renewcommand{\qedsymbol}{$\blacksquare$}

\NewDocumentCommand\myfig{O{t} m m m}{
\begin{figure}[#1]
\centering
#4
\caption{#2}
\label{Figure:#3}
\end{figure}
}

\NewDocumentCommand\refx{mm}{#2~\ref{#2:#1}}
\NewDocumentCommand\refxi{mmm}{#3~\ref{#3:#1}~\refi{#1:#2}}

\NewDocumentCommand\refc{m}{\refx{#1}{Corollary}}

\NewDocumentCommand\refe{m}{\refx{#1}{Example}}
\NewDocumentCommand\reff{m}{\refx{#1}{Figure}}
\NewDocumentCommand\refi{m}{\eqref{Item:#1}}
\NewDocumentCommand\refl{m}{\refx{#1}{Lemma}}

\NewDocumentCommand\refp{m}{\refx{#1}{Proposition}}
\NewDocumentCommand\refpi{mm}{\refxi{#1}{#2}{Proposition}}
\NewDocumentCommand\refq{m}{\eqref{Equation:#1}}

\NewDocumentCommand\refs{m}{\refx{#1}{Section}}
\NewDocumentCommand\reft{m}{\refx{#1}{Theorem}}
\NewDocumentCommand\refti{mm}{\refxi{#1}{#2}{Theorem}}

\newcommand{\oeis}[1]{\href{http://oeis.org/#1}{#1}}

\newcommand{\abs}[1]{\left| #1 \right|}

\newcommand{\defin}[1]{\textcolor{blue}{\emph{#1}}}
\newcommand{\mathdefin}[1]{\textcolor{blue}{#1}}

\newcommand{\dn}{\downarrow}
\newcommand{\dns}{\downarrow}	
\newcommand{\dnw}{\Downarrow}	
\newcommand{\DyckPathAlgebra}{\mathbb{A}_q}

\newcommand{\setN}{\mathbb{N}}
\newcommand{\setNN}{\mathbb{N}^+}
\newcommand{\setQ}{\mathbb{Q}}
\newcommand{\setZ}{\mathbb{Z}}
\newcommand{\setC}{\mathbb{C}}

\newcommand{\xvec}{\mathbf{x}}
\newcommand{\yvec}{\mathbf{y}}
\newcommand{\xqvar}{}
\newcommand{\estep}{\texttt{e}}
\newcommand{\dstep}{\texttt{d}}
\newcommand{\nstep}{\texttt{n}}
\newcommand{\bs}{{\kern-.08em\texttt{.}\kern-.13em}} 
\newcommand{\phibar}{\phi'}
\newcommand{\edge}[2]{{#1#2}}
\newcommand{\dedge}[2]{{\overset{\to}{#1#2}}}

\newcommand{\rfun}[1]{{\rho_{#1}}}

\newcommand{\sfun}[2]{{s}}

\DeclareMathOperator{\rev}{rev}
\DeclareMathOperator{\trace}{trace}
\DeclareMathOperator{\area}{area}
\DeclareMathOperator{\bounce}{bounce}
\DeclareMathOperator{\dinv}{dinv}
\newcommand{\hrv}[2]{\hrvop(#1,#2)}
\DeclareMathOperator{\Flag}{Flag}
\DeclareMathOperator{\Frob}{Frob}
\DeclareMathOperator{\Hess}{Hess}
\DeclareMathOperator{\GL}{GL}
\DeclareMathOperator{\hrvop}{hrv} 	
\DeclareMathOperator{\wst}{wst}
\DeclareMathOperator{\length}{\ell}
\newcommand{\hrvp}{{\pi}} 
\newcommand{\hrvpp}{{\lambda}}

\newcommand{\symS}{\mathfrak{S}}
\newcommand{\field}{\setQ(q)}

\newcommand{\point}{\mathrm{pt}}
\newcommand{\elementaryE}{\mathrm{e}}
\newcommand{\completeH}{\mathrm{h}}
\newcommand{\powerSum}{\mathrm{p}}
\newcommand{\schurS}{\mathrm{s}}
\newcommand{\gessel}{\mathrm{F}}
\newcommand{\hallLittlewoodH}{\mathrm{H}}
\newcommand{\macdonaldH}{\tilde{\mathrm{H}}}
\newcommand{\LLT}{\mathrm{G}}
\newcommand{\LLTe}{\hat{\mathrm{G}}} 
\newcommand{\chrom}{\mathrm{X}}

\newcommand{\ORIENTS}{\mathcal{O}} 
\newcommand{\COL}{\mathcal{C}}     
\newcommand{\uig}{\Gamma}			
\newcommand{\DYCK}[1]{\mathcal{D}_{#1}}     
\newcommand{\SCH}[1]{\mathcal{S}_{#1}}     

\newcommand{\ASC}{\mathrm{ASC}}

\DeclareMathOperator{\asc}{asc}

\DeclareRobustCommand\beadMap[1]{\tikz[baseline]{%
\node[circle,draw,font=\small,%
inner sep=0pt,minimum size=1em,anchor=base]{\ensuremath{#1}};}}

\DeclareRobustCommand\diamondMap[1]{\tikz[baseline]{%
\node[star,star points=2,star point ratio=0.8,draw,fill=blue!15,font=\small,%
inner sep=0pt,minimum size=1em,anchor=base]{\ensuremath{#1}};}}

\DeclareRobustCommand\starMap[1]{\tikz[baseline]{%
\node[star,star points=6,star point ratio=0.8,draw,fill=green!15,font=\small,%
inner sep=0pt,minimum size=1em,anchor=base]{\ensuremath{#1}};}}

\title[A combinatorial \texorpdfstring{$e$}{e}-expansion of vertical-strip LLT polynomials]{%
A combinatorial expansion of vertical-strip LLT polynomials in the basis of elementary symmetric functions}

\author{Per Alexandersson}
\address{Dept. of Mathematics, Stockholm University, Sweden}
\email{per.w.alexandersson@gmail.com}

\author{Robin Sulzgruber}
\address{Dept. of Mathematics and Statistics, York University, Canada}
\email{robinsul@kth.se}

\date{September 7, 2020}

\tikzset{every picture/.append
  style={scale=1,
	baseline=(current bounding box.center),
	x=1em,
	y=1em,
	thinLine/.style={line width=\lineThickness pt},
	thickLine/.style={line width=2*\lineThickness pt,line join=round},
	snaking/.style={decorate,decoration={zigzag,segment length=1.0mm,amplitude=0.15mm}},
	fillGrey/.style={fill=black!30},
	entries/.style={xshift=-0.5em,yshift=-0.5em,font=\small},
	circled/.style={circle,draw,inner sep=0pt,minimum size=1em},
	mydot/.style={circle,fill=black,inner sep=0pt,minimum size=1mm}
	}
}

\tikzset{
	chromaticSmall/.pic = {
	code {
		\fill[color=gray!6] (-0.4,-1.2) rectangle (3.4,3.2);
		\draw[snaking] 
			(0,0)--(1,1)
			(2,2)--(3,3);
	}},
	chromSmall-ne/.pic = {
	code {
		\draw[thickLine] (1,1)--(1,2)--(2,2);
		\draw
			(1,1)node[mydot]{}
			(1,2)node[mydot]{}
			(2,2)node[mydot]{};
	}},
	chromSmall-en/.pic = {
	code {
		\draw[thickLine] (1,1)--(2,1)--(2,2);
		\draw
			(1,1)node[mydot]{}
			(2,1)node[mydot]{}
			(2,2)node[mydot]{};
	}},
	chromSmall-d/.pic = {
	code {
		\draw[thickLine] (1,1)--(2,2);
		\draw
			(1,1)node[mydot]{}
			(2,2)node[mydot]{};
	}}
	,
	chromatic/.pic = {
	code {
		\fill[color=gray!6] (-0.4,-1.2) rectangle (5.4,5.2);
		\draw[snaking] (1,2)--(3,4);
		\draw[thinLine] (3,0)--(5,2);
	}},
	chrom-swapShade/.pic = {
	code {
		\fill[color=green!20] (3,2) rectangle (3.8,4);
		\fill[color=green!20] (4.2,2) rectangle (5,4);
		\fill[color=blue!20] (1,1.2) rectangle (3,2);
		\fill[color=blue!20] (1,0) rectangle (3,0.8);
	}},
	chrom-nn/.pic = {
	code {
		\draw[dotted] (0,0)--(1,0);
		\draw[thickLine] (1,0)--(1,2);
		\draw
			(1,0)node[mydot]{}
			(1,1)node[mydot]{}
			(1,2)node[mydot]{};
	}},
	chrom-enn/.pic = {
	code {
		\draw[thickLine] (0,0)--(1,0)--(1,2);
		\draw
			(0,0)node[mydot]{}
			(1,0)node[mydot]{}
			(1,1)node[mydot]{}
			(1,2)node[mydot]{};
	}},
	chrom-nne/.pic = {
	code {
		\draw[thickLine] (0,0)--(0,2)--(1,2);
		\draw
			(0,0)node[mydot]{}
			(0,1)node[mydot]{}
			(0,2)node[mydot]{}
			(1,2)node[mydot]{};
	}},
	chrom-nen/.pic = {
	code {
		\draw[thickLine] (0,0)--(0,1)--(1,1)--(1,2);
		\draw
			(0,0)node[mydot]{}
			(0,1)node[mydot]{}
			(1,1)node[mydot]{}
			(1,2)node[mydot]{};
	}},
	chrom-nd/.pic = {
	code {
		\draw[thickLine] (0,0)--(0,1)--(1,2);
		\draw
			(0,0)node[mydot]{}
			(0,1)node[mydot]{}
			(1,2)node[mydot]{};
	}},
	chrom-dn/.pic = {
	code {
		\draw[thickLine] (0,0)--(1,1)--(1,2);
		\draw
			(0,0)node[mydot]{}
			(1,1)node[mydot]{}
			(1,2)node[mydot]{};
	}},
	chrom-nee/.pic = {
	code {
		\draw[thickLine] (3,4)--(3,5)--(5,5);
		\draw
			(3,4)node[mydot]{}
			(3,5)node[mydot]{}
			(4,5)node[mydot]{}
			(5,5)node[mydot]{};
	}},
	chrom-de/.pic = {
	code {
		\draw[thickLine] (3,4)--(4,5)--(5,5);
		\draw
			(3,4)node[mydot]{}
			(4,5)node[mydot]{}
			(5,5)node[mydot]{};
	}},
	chrom-ed/.pic = {
	code {
		\draw[thickLine] (3,4)--(4,4)--(5,5);
		\draw
			(3,4)node[mydot]{}
			(4,4)node[mydot]{}
			(5,5)node[mydot]{};
	}},
	chrom-ene/.pic = {
	code {
		\draw[thickLine] (3,4)--(4,4)--(4,5)--(5,5);
		\draw
			(3,4)node[mydot]{}
			(4,4)node[mydot]{}
			(4,5)node[mydot]{}
			(5,5)node[mydot]{};
	}},
	chrom-een/.pic = {
	code {
		\draw[thickLine] (3,4)--(5,4)--(5,5);
		\draw
			(3,4)node[mydot]{}
			(4,4)node[mydot]{}
			(5,4)node[mydot]{}
			(5,5)node[mydot]{};
	}},
	chrom-bounceLL/.pic = {
	code {
		\draw[red,thickLine,->] (4,5)--(4,1)--(0,1);
	}},
	chrom-bounceLS/.pic = {
	code {
		\draw[red,thickLine,->] (4,5)--(4,1)--(1,1);
	}},
	chrom-bounceSL/.pic = {
	code {
		\draw[red,thickLine,->] (4,4)--(4,1)--(0,1);
	}},
	chrom-bounceSS/.pic = {
	code {
		\draw[red,thickLine,->] (4,4)--(4,1)--(1,1);
	}}
}

\begin{document}

\begin{abstract}
We give a new characterization of the vertical-strip LLT polynomials $\LLT_P(\xvec;q)$ as the unique family of symmetric functions that satisfy certain combinatorial relations.
This characterization is then used to prove an explicit combinatorial expansion of vertical-strip LLT polynomials in terms of elementary symmetric functions.
Such formulas were conjectured independently by A.~Garsia et al.~and the first named author, 
and are governed by the combinatorics of orientations of unit-interval graphs.
The obtained expansion is manifestly positive if $q$ is replaced by $q+1$, thus recovering a recent result of M.~D'Adderio.
Our results are based on linear relations among LLT polynomials that arise in the work of D'Adderio, and of E.~Carlsson and A.~Mellit.
To some extent these relations are given new bijective proofs using colorings of unit-interval graphs.
As a bonus we obtain a new characterization of chromatic quasisymmetric functions of unit-interval graphs.
\end{abstract}

\maketitle

\setcounter{tocdepth}{2}
\tableofcontents

\section{Overview of main results}

LLT polynomials form a large class of symmetric functions that can be viewed as $q$-deformations of products of (skew) Schur functions.
They appear in many different contexts in representation theory and algebraic combinatorics.
In this paper we are concerned with a particular subclass of LLT polynomials, namely the vertical-strip LLT polynomials.
They contain information about the equivariant cohomology ring of 
regular semisimple Hessenberg varieties, and are integral to the study 
of diagonal harmonics and the proof of the Shuffle~Theorem.

Vertical-strip LLT polynomials can be indexed by certain lattice paths called Schr{\"o}der paths.
We view the vertical-strip LLT polynomial $\LLT_P(\xvec;q)$ associate to a Schr{\"o}der path $P$ of size $n$ as the generating function of certain colorings of a unit-interval graph $\Gamma_P$ on $n$ vertices.
The main result of this paper is an explicit expansion of $\LLT_P(\xvec;q)$ in the basis of 
elementary symmetric functions as a weighted sum over orientations of $\Gamma_P$.
This expansion is given by
\begin{eq}{mainResult}
\LLT_P(\xvec;q)
=
\sum_{\kappa \in \COL(P)} q^{\asc(\kappa)} x_{\kappa(1)} x_{\kappa(2)} \dotsm x_{\kappa(n)}
= \sum_{\theta \in \ORIENTS(P)} (q-1)^{\asc(\theta)}
\elementaryE_{\hrvpp(\theta)}(\xvec),
\end{eq}
where $\COL(P)$ is a set of colorings, $\ORIENTS(P)$ is a set of orientations,
and $\asc$ denotes certain combinatorial statistics on these objects.
The expression $\hrvpp(\theta)$ is an integer partition determined by the orientation $\theta$.
In particular, the above identity proves that $\LLT_P(\xvec;q+1)$ is $\elementaryE$-positive.
This resolves several open conjecture stated 
in \cite{AlexanderssonPanova2016,Alexandersson2019llt,GarsiaHaglundQiuRomero2019}.
Algebraically the $\elementaryE$-positivity phenomenon implies that the symmetric function in question is --- up to a twist by the sign representation --- the Frobenius character of a permutation representation of the symmetric group in which
each point stabilizer is a Young subgroup.
Our approach to vertical-strip LLT polynomials highlights their similarities with chromatic quasisymmetric functions.
The expansion in \refq{mainResult} is of particular interest because 
it serves as an analog of the Shareshian--Wachs conjecture 
regarding the $\elementaryE$-positivity of chromatic quasisymmetric functions.

\medskip

We now briefly outline the contents of this paper.
\refs{background} contains background on LLT polynomials, 
and all definitions necessary to state our main results.
In \reft{recursion} and \reft{LLT} we provide sets of relations 
which uniquely determine the vertical-strip LLT polynomials.
\reft{LLTe} states that the right-hand side in \refq{mainResult} satisfies the 
same set of relations, which implies equality in \refq{mainResult}.
\refs{bounce} treats linear relations among LLT polynomials and contains the proof of \reft{recursion}.
Moreover we give a new characterisation of the (smaller) class of unicellular LLT polynomials in \reft{dyckRecursion}.
\refs{colorings} contains bijective results on colorings and the proof of \reft{LLT}.
\refs{orientations} deals with the combinatorics of orientations, culminating in a proof of \reft{LLTe}.

In \refs{applications} we give an overview of related areas and open problems.
We only mention a few here. 
In \refc{schurExpansion} we state a new signed Schur expansion of vertical-strip LLT polynomials.
We obtain new expressions for quantities that appear in diagonal harmonics.
We consider $\nabla \elementaryE_n$ and $\nabla \omega \powerSum_n$,
which appear in the Shuffle~Theorem and the Square~Path~Theorem, respectively, and both of which can be expressed using vertical-strip LLT polynomials.
As a consequence they are both $\elementaryE$-positive after the substitution $q \to q+1$,
see \refc{shuffleCor} and \reft{squarePathsThmLLT}.
Finally we discuss chromatic quasisymmetric functions and the 
representation theory of regular semisimple Hessenberg varieties.
In particular \refc{chromaticRecursion} provides a new characterisation 
of chromatic quasisymmetric functions of unit-interval graphs.

\section{Background and main results}
\label{Section:background}

We first give additional history and references which lead up to this paper.
Then we go into more detail and give all necessary 
definitions needed in order to state the main results.

\subsection{Brief history of LLT polynomials}

Below is a very brief history of LLT polynomials --- there are
several more references which are not included with strong results.

\begin{itemize}
 \item[(1997)]  Motivated by the study of certain Fock space representations,
 and plethysm coefficients, A.~Lascoux, B.~Leclerc and J.~Y.~Thibon 
 introduce the LLT polynomials in \cite{Lascoux97ribbontableaux} 
 under the name \defin{ribbon Schur functions}.
 They are defined as a sum over semi-standard border-strip tableaux
 of straight shape.
 \item[(2000)] Three years later B.~Leclerc and J.~Y.~Thibon \cite{LeclercThibon2000}
 show that the LLT polynomials are Schur positive by means of representation theory.
 
 \item[(2000+)] M.~Bylund and M.~Haiman  use the so called \emph{Littlewood map} (see history in \cite[p.~92]{qtCatalanBook})
 which is a bijection between semi-standard border-strip tableaux
 and $k$-tuples of semi-standard Young tableaux of straight shape.
 They extend their model to include tuples of skew shapes,
 and they conjecture that LLT polynomials in this family are also Schur positive. 
 Note that the proof of Schur positivity from 2000 does not extend to the Bylund--Haiman model.
 We remark that a similar model was introduced independently 
 in \cite{SchillingShimozonoWhite2003}, but 
 they use a more complicated statistic for the $q$-weight.

 \item[(2005)] J.~Haglund, M.~Haiman and M.~Loehr \cite{HaglundHaimanLoehr2005}
 give a combinatorial formula for the \defin{modified Macdonald polynomials}.
 They show that these Macdonald polynomials can be expressed as positive linear combinations
 of LLT polynomials indexed by $k$-tuples of ribbon shapes (these are skew shapes).
 Hence a deeper understanding of LLT polynomials has consequences for the modified Macdonald polynomials.
 
 \item[(2006)] I.~Grojnowski and M.~Haiman post a preprint \cite{GrojnowskiHaiman2006},
 giving an argument for Schur positivity of LLT polynomials by using Kazhdan--Luztig theory
 and  geometric representation theory.
 
 \item[(2007+)] J.~Haglund~\cite[Ch.~6]{qtCatalanBook}
 gives an overview of the appearance of LLT polynomials in 
 the study of \defin{diagonal harmonics}.
 In particular, the \defin{path symmetric functions} ---
 which are LLT polynomials indexed by $k$-tuples of vertical strips ---
 appear in the \emph{Shuffle Conjecture}~\cite{HaglundHaimanLoehrRemmelUlyanov2005} and later in the 
 \emph{Compositional Shuffle Conjecture}~\cite{HaglundMorseZabrocki2012},
 \emph{Square Path Conjecture}~\cite{Sergel2017} and 
 \emph{Delta Conjecture}~\cite{HaglundRemmelWilson2018} .
 
 \item[(2016)] In his work on chromatic quasisymmetric 
 functions M.~Guay-Paquet~\cite{GuayPaquet2016x} uses a Hopf~algebra approach 
 to show that unicellular LLT polynomials are graded Frobenius series derived from the equivariant 
 cohomology rings of regular semisimple Hessenberg varieties.  This yields a second 
 proof that unicellular LLT polynomials are Schur positive, 
 but no combinatorial formula for the coefficients.
 
 \item[(2017)]  
 E.~Carlsson and A.~Mellit~\cite{CarlssonMellit2017}
 use the so called zeta map on the path symmetric functions to obtain a more convenient model for vertical-strip LLT polynomials.
 They introduce the \defin{Dyck path algebra} and prove the Compositional 
 Shuffle Conjecture --- henceforth Shuffle Theorem.
 
 E.~Carlsson and A.~Mellit briefly use a plethystic identity 
 originating in \cite{HaglundHaimanLoehr2005}. 
 In the language of Haglund et al.\ it essentially states that 
 there is a simple plethystic relationship between general 
 fillings and non-attacking fillings.
 In terms of vertex colorings of graphs this relates a sum over arbitrary colorings to a sum over proper colorings. 
 Consequently there is a plethystic relationship 
 between certain (unicellular) LLT polynomials and the chromatic quasisymmetric functions 
 of J.~Shareshian and M.~Wachs~\cite{ShareshianWachs2012,ShareshianWachs2016}. 
  
 \item[(2018)] The Carlsson--Mellit model is 
 used by G.~Panova and the first named author~\cite{AlexanderssonPanova2016}
 to emphasize the connection with the chromatic quasisymmetric functions.
 A conjecture regarding $\elementaryE$-positivity of certain LLT polynomials is stated in \cite[Conj.~4.5]{AlexanderssonPanova2016}.
 This conjecture serves as an analog of the $\elementaryE$-positivity conjecture of 
 Shareshian--Wachs, and thus the Stanley--Stembridge 
 conjecture for chromatic symmetric functions~\cite{Stanley1995,StanleyStembridge1993}.

 \item[(2019)] The first named author posts a preprint~\cite{Alexandersson2019llt} 
 that contains the explicit formula \refq{mainResult} for vertical-strip LLT polynomials as a conjecture.
 Shortly after, A.~Garsia, J.~Haglund, D.~Qiu and M.~Romero~\cite{GarsiaHaglundQiuRomero2019}
 independently publish an equivalent conjecture in the language of diagonal harmonics.

 \item[(2019)] M.~D'Adderio~\cite{DAdderio2020} proves $\elementaryE$-positivity 
 by using the Dyck path algebra and recursions developed by Carlsson--Mellit. 
 However, this does not prove the conjectured formula \refq{mainResult}.
\end{itemize}

\subsection{Symmetric functions}
Let $\mathdefin{\Lambda}$ denote the algebra of symmetric functions over $\field$.
For an introduction to symmetric functions the reader is referred to \cite{Macdonald1996,StanleyEC2}.
Elements of $\Lambda$ may be viewed as formal power series over $\field$ in infinitely many 
variables $\mathdefin{\xvec}\coloneqq (x_1,x_2,\dotsc)$ that have finite degree and are invariant 
under permutation of the variables.
Alternatively we may think of elements of $\Lambda$ as $\field$-linear 
combinations of the \defin{elementary symmetric functions} 
$\mathdefin{\elementaryE_{\lambda}}=\elementaryE_{\lambda_1}\dotsb\elementaryE_{\lambda_r}$ 
where $\lambda$ ranges over all integer partitions.

\subsection{Schr{\"o}der paths}
Let $\mathdefin{\texttt{n}}\coloneqq (0,1)$ be a \defin{north step}, 
$\mathdefin{\texttt{e}}\coloneqq (1,0)$ be an \defin{east step}, 
and $\mathdefin{\texttt{d}}\coloneqq (1,1)$ be a \defin{diagonal step}.
A \defin{Schr{\"o}der path} $P$ of \defin{size} $n$ 
is a lattice path from $(0,0)$ to $(n,n)$ using steps 
from $\{\texttt{n},\texttt{e},\texttt{d}\}$ that never goes below the main diagonal, and has no diagonal step on the main diagonal.
That is, every east step and every diagonal step of $P$ is preceded by more north steps than east steps.
Let $\mathdefin{\SCH{n}}$ denote the set of Schr{\"o}der paths of size $n$, 
and let $\mathdefin{\SCH{}}$ be the set of all Schr{\"o}der paths.
We describe Schr{\"o}der paths as words in $\{\nstep,\dstep,\estep\}^*$,
see \reff{schroederPaths} for examples.
A Schr{\"o}der path without diagonal steps is called a \defin{Dyck path}.

\myfig[ht]{
The set $\SCH{3}=\{\nstep\nstep\nstep\estep\estep\estep,\nstep\nstep\estep\nstep\estep\estep,\dotsc,\nstep\dstep\dstep\estep\}$ consisting of the eleven Schr{\"o}der paths of size $3$.
The first row consists of all Dyck paths.
The number of Schr{\"o}der paths of size $n$ is given by the sequence \oeis{A001003},
and starts with $1$, $3$, $11$, $45$, $197,\dotsc$.
}{schroederPaths}
{
\includegraphics[scale=1]{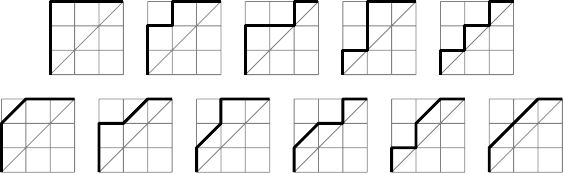}
}

\subsection{Bounce paths}

An important tool in our analysis is a type of bounce path 
that is similar to bounce paths that appear 
for example in \cite{GarsiaHaglund2000,AndrewsKrattenthalerOrsinaPapi2002}.

Given a Schr{\"o}der path $P\in\SCH{}$ and a point $(u_1,u_0)\in\setZ^2$ that lies on $P$ define the \defin{(partial reverse) bounce path} of $P$ at $(u_1,u_0)$ as follows:
Starting at $(u_1,u_0)$ move south until you reach the point $(u_1,u_1)$ on the main diagonal.
Now move west until you reach a 
point $(u_2,u_1)$ 
on $P$.
If this point lies between two diagonal steps of $P$ then continue south until the point $(u_2,u_2)$ on the diagonal.
Then move west until you reach a
point $(u_3,u_2)$ 
on $P$.
Continue in this fashion until the bounce path ends at a point incident to a north step or east step of $P$.
An example of a bounce path is shown in \reff{bounce_path_example}.
The points $(u_1,u_1),\dotsc,(u_k,u_k)$ where the bounce path touches the main diagonal are called \defin{bounce points}.
The coordinates of the peaks $(u_1,u_0),(u_2,u_1),\dotsc,(u_{k+1},u_k)$ of the bounce path are recorded in the \defin{bounce partition} $(u_0,u_1,\dotsc,u_{k+1})$ of $P$ at $(u_1,u_0)$.
There is a unique decomposition
\begin{eq*}
P=Us_1 \bs s_2Vs_3 \bs s_4 W
\end{eq*}
where $s_1,s_2,s_3,s_4\in\{\nstep,\dstep,\estep\}$ and $U,V,W\in\{\nstep,\dstep,\estep\}^{*}$ such that $Us_1$ is a path from $(0,0)$ to the endpoint $(u_{k+1},u_k)$ of the bounce path, and $s_4W$ is a path from the starting point $(u_1,u_0)$ of the bounce path to $(n,n)$.
We call this the \defin{bounce decomposition} of $P$ at $(u_1,u_0)$.
Note that we use a small dot to indicate the starting point and endpoint of the bounce path.
In this paper our focus is mainly on the special case where 
$s_1s_2 \in\{\nstep\nstep,\dstep\nstep\}$, $V\in\{\nstep,\dstep\}^*$ and $s_3s_4=\dstep\estep$.
The pieces that such bounce paths are made of are found in \reff{bounce_path}.

\myfig[ht]
{The bounce paths of two Schr{\"o}der paths $P$ (left) and $Q$ (right) at the point $(3,6)$.
The bounce decompositions are given by $P=U\nstep\bs\nstep V\dstep\bs\estep W$ where $U=\emptyset, V=\dstep\dstep\nstep$, and $W=\estep\estep$, and $Q=U'\nstep\bs\dstep V'\dstep\bs\estep W'$ where $U'=\nstep\dstep, V'=\nstep$, and $W'=\estep\estep$ respectively.
The left bounce path has two bounce points $(3,3)$ and $(1,1)$, and bounce partition $(6,3,1,0)$.}
{bounce_path_example}
{
\begin{tikzpicture}
\begin{scope}
\draw[gray,step=1em](0,0)grid(6,6)(0,0)--(6,6);
\draw[thickLine](0,0)--(0,2)--(2,4)--(2,5)--(3,6)--(6,6);
\draw[red,thickLine,->](3,6)--(3,3)--(1,3)--(1,1)--(0,1);
\end{scope}
\begin{scope}[xshift=10em]
\draw[gray,step=1em](0,0)grid(6,6)(0,0)--(6,6);
\draw[thickLine](0,0)--(0,1)--(1,2)--(1,3)--(2,4)--(2,5)--(3,6)--(6,6);
\draw[red,thickLine,->](3,6)--(3,3)--(1,3);
\end{scope}
\end{tikzpicture}
}

\myfig[ht]
{The start (left), middle (middle), and end (right) of the bounce paths we are mainly interested in.}{bounce_path}{
\begin{tikzpicture}
\begin{scope}
\begin{scope}
\draw[step=1em,gray](0,0)grid(2,3);
\draw[red,thickLine,->](1,3)--(1,1);
\draw[thickLine](0,2)--(1,3)--(2,3);
\draw
(0,2)node[mydot]{}
(1,3)node[mydot]{}
(2,3)node[mydot]{};
\end{scope}
%
\end{scope}
\begin{scope}[xshift=6em,yshift=-.5em]
\begin{scope}
\draw[step=1em,gray](0,0)grid(4,4)
(1,0)--(4,3);
\draw[red,thickLine,->](3,4)--(3,2)--(1,2);
\end{scope}
\begin{scope}[xshift=5em,yshift=0em]
\draw[step=1em,gray](0,0)grid(3,4);
\draw[red,thickLine,->](3,3)--(1,3)--(1,1);
\draw[thickLine](0,2)--(2,4);
\draw
(0,2)node[mydot]{}
(1,3)node[mydot]{}
(2,4)node[mydot]{};
\end{scope}
\end{scope}
\begin{scope}[xshift=18em,yshift=.5em]
\begin{scope}
\draw[step=1em,gray](0,0)grid(2,2);
\draw[thickLine](0,0)--(0,2);
\draw[red,thickLine,->](2,1)--(0,1);
\draw
(0,0)node[mydot]{}
(0,1)node[mydot]{}
(0,2)node[mydot]{};
\end{scope}
\begin{scope}[xshift=3em,yshift=0em]
\draw[step=1em,gray](0,0)grid(2,2);
\draw[thickLine](0,0)--(1,1)--(1,2);
\draw[red,thickLine,->](2,1)--(1,1);
\draw
(0,0)node[mydot]{}
(1,1)node[mydot]{}
(1,2)node[mydot]{};
\end{scope}
%
\end{scope}
\end{tikzpicture}
}

\subsection{The main recursion}

An essential step in M.~D'Adderio's proof of the $\elementaryE$-positivity of 
vertical-strip LLT polynomials is an inductive argument that relies on certain linear relations among 
them.
The starting point of our work is to simplify these relations, and to make them more explicit.
This leads to our first main result, namely, that the following initial conditions 
and relations determine a unique symmetric-function valued statistic on Schr{\"o}der paths.

\begin{thm}{recursion}
Let $A$ be a $K$-algebra that contains $\Lambda$.
Then there exists a unique function $F:\SCH{}\to A$, $P\mapsto F_P\xqvar$ that satisfies the following four conditions:
\begin{enumerate}[(i)]
\myi{recursion:e}
For all $k\in\setN$ the initial condition
$F_{\nstep\dstep^k\estep}\xqvar
=
\elementaryE_{k+1}
$
is satisfied.
\myi{recursion:mult}
The function $F$ is multiplicative, that is,
$F_{PQ}\xqvar
=
F_{P}\xqvar
F_{Q}\xqvar
$
for all $P,Q\in\SCH{}$.

\myi{recursion:unicellular}
For all $U,V\in\{\nstep,\dstep,\estep\}^*$ such that $U\dstep V\in\SCH{}$ we have
$
F_{U\nstep\estep V}\xqvar - F_{U\estep\nstep V}\xqvar
=
(q-1)F_{U\dstep V}\xqvar
$.

\myi{recursion:bounce}
Let $P\in\SCH{}$ be a Schr{\"o}der path, and let $(x,z)\in\setZ^2$ be a point on $P$ with $x+1<z$ such that the bounce path of $P$ at $(x,z)$ has \emph{only one} bounce point, and such that the bounce decomposition of $P$ at $(x,z)$ is given by 
$P=Us \bs tV\dstep \bs \estep W$
for some $V\in\{\nstep,\dstep\}^*$ and some $st\in\{\nstep\nstep,\dstep\nstep
\}$.
Then
\begin{eq*}
F_{P}\xqvar
=
\begin{cases}
qF_{U\nstep\nstep V\estep\dstep W}\xqvar
&\quad\text{if }st=\nstep\nstep,
\\
F_{U\nstep\dstep V\estep\dstep W}\xqvar
&\quad\text{if }st=\dstep\nstep
.
\end{cases}
\end{eq*}
\end{enumerate}
\end{thm}

Parts of this recursion are not new and have shown up in various forms and shapes in previous research.
Due to the plethystic relationship between
unicellular LLT polynomials and chromatic quasisymmetric function,
some of them have analogs on the chromatic quasisymmetric
function side.

\begin{rem}{unicellular}
We call the relation in \refti{recursion}{unicellular} the \defin{unicellular relation}.
It can be illustrated by
\begin{eq}{unicellular_diagram}
\begin{tikzpicture}[scale=0.5]
\draw pic {chromaticSmall};
\draw pic {chromSmall-ne};
\node at (3,-1.1) {$U{\nstep\estep}V$};
\end{tikzpicture}
\quad
-
\quad
\begin{tikzpicture}[scale=0.5]
\draw pic {chromaticSmall};
\draw pic {chromSmall-en};
\node at (3,-1.1) {$U{\estep\nstep}V$};
\end{tikzpicture}
\quad
=
\quad
(q-1)\times
\begin{tikzpicture}[scale=0.5]
\draw pic {chromaticSmall};
\draw pic {chromSmall-d};
\node at (3,-1.1) {$U{\dstep}V$};
\end{tikzpicture}
\end{eq}
and has appeared for example in 
\cite{AlexanderssonPanova2016,Alexandersson2019llt,DAdderio2020}.
It can be used to express the function $F_P$ indexed by a Schr{\"o}der path as a 
linear combination of such functions indexed by Dyck paths.
Equivalently it can be used to express vertical-strip LLT polynomials 
as linear combinations of unicellular LLT polynomials.
Moreover it is reflected in the definition of the operator $\varphi$ in terms of the operators $d_{+}$ and $d_{-}$ in the Dyck path algebra of E.~Carlsson and A.~Mellit~\cite{CarlssonMellit2017}. 
\end{rem}

\begin{rem}{bounce_relations}
We call relations of the type of \refti{recursion}{bounce} \defin{bounce relations}.
The two bounce relations in \refi{recursion:bounce}
can be illustrated by
\begin{align}
\begin{tikzpicture}[scale=0.5]
\draw pic {chromatic};
\draw pic {chrom-bounceLS};
\draw pic {chrom-nn};
\draw pic {chrom-de};
\node at (5,-1.1) {$U{\nstep\bs\nstep}V{\dstep\bs\estep}W$};
\end{tikzpicture}
\quad&=\quad
q\times
\begin{tikzpicture}[scale=0.5]
\draw pic {chromatic};
\draw pic {chrom-bounceSS};
\draw pic {chrom-nn};
\draw pic {chrom-ed};
\node at (5,-1.1) {$U{\nstep\bs\nstep}V{\estep\bs\dstep}W$};
\end{tikzpicture}
& \label{Equation:LLT_caseI_diagram}
\\
\begin{tikzpicture}[scale=0.5]
\draw pic {chromatic};
\draw pic {chrom-bounceLS};
\draw pic {chrom-dn};
\draw pic {chrom-de};
\node at (5,-1.1) {$U{\dstep\bs\nstep}V{\dstep\bs\estep}W$};
\end{tikzpicture}
\quad&=\quad
\begin{tikzpicture}[scale=0.5]
\draw pic {chromatic};
\draw pic {chrom-bounceSL};
\draw pic {chrom-nd};
\draw pic {chrom-ed};
\node at (5,-1.1) {$U{\nstep\bs\dstep}V{\estep\bs\dstep}W$};
\end{tikzpicture} \label{Equation:LLT_caseII_diagram}
\end{align}
Both of these relations are implicit in the work of M.~D'Adderio \cite[Lem.~5.2]{DAdderio2020} who derives them from commutation relations between operators in the Dyck path algebra.
The relation in~\refq{LLT_caseI_diagram} is equivalent to \refq{dyckBounceI} 
further down. To the best of the authors' knowledge, 
\refq{dyckBounceI} was first described in the context of chromatic quasisymmetric 
functions by M.~Guay-Paquet~\cite[Prop.~3.1]{GuayPaquet2013} 
where it is called the \defin{modular relation}.
It was independently found on the LLT side under the 
name \defin{local linear relation} by S.~J.~Lee \cite[Thm.~3.4]{Lee2018} who 
worked in the setting of \defin{abelian} Dyck paths.
Subsequently it appears in~\cite{Alexandersson2019llt,Miller2019}.
Moreover in \cite[Thm.~3.1]{HuhNamYoo2020} Lee's linear relation on LLT polynomials is translated to chromatic symmetric functions.
\end{rem}

\reft{recursion} is formulated in such a way that it only assumes the weakest set 
bounce relations that are necessary to imply uniqueness.
In \refs{bounce} we will see that these relations together with the 
unicellular relation are strong enough to imply more general bounce 
relations with arbitrarily many bounce points and bounce paths that may end in different patterns.
Moreover \reft{recursion} remains valid if we allow $V\in\{\nstep,\dstep,\estep\}^*$.

We prove the uniqueness statement in \reft{recursion} in \refs{bounce} along with other results of a similar flavour.
The existence statement is taken care of in Theorems~\ref{Theorem:LLT} and~\ref{Theorem:LLTe}.
There we show 
by combinatorial means that there exist \emph{two} statistics on Schr{\"o}der 
paths that satisfy the conditions of \reft{recursion}.
We then conclude that they must be equal.
It turns out that the unique function $F$ in \reft{recursion} assigns to 
each Schr{\"o}der path the corresponding vertical-strip LLT polynomial.

\subsection{Unit-interval graphs}
A graph $\Gamma=(V,E)$ with vertex set $V=[n]$ is a \defin{unit-interval graph} 
if $\edge{u}{w}\in E$ implies $\edge{u}{v},\edge{v}{w}\in E$ for all $u,v,w\in[n]$ with $u<v<w$.
Unit-interval graphs are the incomparability graphs of unit-interval orders.
They are in bijection with Dyck paths and of 
great interest in the context of chromatic (quasi)symmetric functions.
The seminal work on this topic is due to J.~Shareshian and M.~Wachs~\cite{ShareshianWachs2012,ShareshianWachs2016}.

A \defin{decorated} unit-interval graph $\Gamma=(V,E,S)$ is a pair of 
a unit-interval graph $\Gamma=(V,E)$ and a subset $S\subseteq E$ of \defin{strict edges}.
To a Schr{\"o}der path $P$ in $\SCH{n}$ we associate a decorated unit-interval graph $\mathdefin{\Gamma_P}$ with vertex set $[n]$ as follows:
For $u,v\in[n]$ with $u<v$ there is a non-strict edge $\edge{u}{v}$ in $\Gamma_P$ if and only if there is a cell in column $u$ and row $v$ below the path $P$.
Moreover if $(u,v)$ is the endpoint of a diagonal step of $P$ then $\edge{u}{v}$ is a strict edge in $\Gamma_P$.
See \reff{schroder_path_and_graph} for an example of this correspondence.

\myfig[h]
{(a) The Schr{\"o}der path $P=\texttt{nndnnenedeee}$.
(b) The unit-interval graph $\Gamma_P$.
(c) A sligthly simplified diagram (Dyck diagram) describing $P$.
We have that $\area(P)=12$, and it is the number 
of white squares in this diagram.
(d) The vertical strips (French notation!) corresponding to $P$. 
Thus, $\nu' = (2/1, 3/2,2/1,2,3/1)$, and $\LLT_P = \LLT_\nu$,
using the Bylund--Haiman convention for 
indexing LLT polynomials with tuples of skew shapes.
The labeling of the boxes indicate the correspondence 
with the vertices in $\Gamma_P$.
}{schroder_path_and_graph}{
\begin{tikzpicture}
\begin{scope}
\draw[step=1em,gray](0,0)grid(7,7);
\draw[thickLine](0,0)--(0,2)--(1,3)--(1,5)--(2,5)--(2,6)--(3,6)--(4,7)--(7,7);
\draw
(0,0)node[mydot]{}
(0,2)node[mydot]{}
(1,3)node[mydot]{}
(1,5)node[mydot]{}
(2,5)node[mydot]{}
(2,6)node[mydot]{}
(3,6)node[mydot]{}
(4,7)node[mydot]{}
(7,7)node[mydot]{};
\draw(1,1)node[entries]{$1$};
\draw(2,2)node[entries]{$2$};
\draw(3,3)node[entries]{$3$};
\draw(4,4)node[entries]{$4$};
\draw(5,5)node[entries]{$5$};
\draw(6,6)node[entries]{$6$};
\draw(7,7)node[entries]{$7$};
\draw
(1,2)node[entries]{$\scriptstyle{12}$}
(2,3)node[entries]{$\scriptstyle{23}$}
(2,4)node[entries]{$\scriptstyle{24}$}
(2,5)node[entries]{$\scriptstyle{25}$}
(3,4)node[entries]{$\scriptstyle{34}$}
(3,5)node[entries]{$\scriptstyle{35}$}
(3,6)node[entries]{$\scriptstyle{36}$}
(4,5)node[entries]{$\scriptstyle{45}$}
(4,6)node[entries]{$\scriptstyle{46}$}
(5,6)node[entries]{$\scriptstyle{56}$}
(5,7)node[entries]{$\scriptstyle{57}$}
(6,7)node[entries]{$\scriptstyle{67}$};
\end{scope}
\begin{scope}[xshift=8em]
\draw
(4,1)node[circled](v1){1}
(7,2)node[circled](v2){2}
(7,5)node[circled](v3){3}
(5,7)node[circled](v4){4}
(3,7)node[circled](v5){5}
(1,5)node[circled](v6){6}
(1,2)node[circled](v7){7};
\draw[thickLine] (v1)--(v2)--(v3)--(v4)--(v5)--(v6)--(v7);
\draw[thickLine] (v5)--(v2)--(v4)--(v6)--(v3)--(v5)--(v7);
\draw[thickLine,red,->] (v1)--(v3);
\draw[thickLine,red,->] (v4)--(v7);
\end{scope}
\end{tikzpicture}
\quad
\ytableausetup{boxsize=1.0em,aligntableaux=center}
\begin{ytableau}
*(lightgray)  &*(lightgray)&*(lightgray)&*(lightgray) \to&&& *(yellow) 7 \\
*(lightgray)  &*(lightgray) &&&& *(yellow) 6 \\
*(lightgray)  &&&& *(yellow) 5 \\
*(lightgray)  &&& *(yellow) 4 \\
*(lightgray) \to  &  & *(yellow) 3 \\
  & *(yellow)  2 \\
*(yellow)  1
\end{ytableau}
\quad 
\includegraphics[page=1,scale=0.8,valign=c]{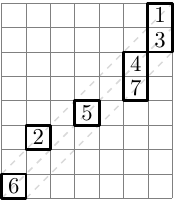}
}

\subsection{Colorings}

Let $P\in\SCH{n}$ be a Schr{\"o}der path.
A \defin{coloring} of $\Gamma_P$ is a map $\kappa : [n] \to \setNN$ such that $\kappa(u) < \kappa(v)$ for all $u,v\in[n]$ such that $u<v$ and $\edge{u}{v}$ is a strict edge.
An \defin{ascent} of a coloring is a a non-strict edge $\edge{u}{v}$ of $\Gamma_P$
such that $u<v$ and $\kappa(u) < \kappa(v)$.
We let $\mathdefin{\asc(\kappa)}$ denote the number of ascents of $\kappa$.

We remark that our terminology stems from the world of chromatic quasisymmetric functions.
Colorings are also closely related to parking functions, and the ascent statistic is called $\area'$ in that setting, see~\cite[Ch.~5]{qtCatalanBook}.
We define the \defin{area} of $P$, denoted $\mathdefin{\area(P)}\coloneqq \abs{E\setminus S}$, as the number of non-strict edges in $\uig_P$.

\begin{exa}{lltColoringExample}
Recall the path $P=\texttt{nndnnenedeee}$ from \reff{schroder_path_and_graph}.
The following diagram illustrates a coloring of $P$ where we have labeled 
$i$ in $\Gamma_P$ with $\kappa(i)$:
\begin{eq*}
\ytableausetup{boxsize=1.0em}
\begin{ytableau}
*(lightgray)  &*(lightgray)&*(lightgray)&*(lightgray) \to & & \circ & *(yellow) 2 \\
*(lightgray)  &*(lightgray) &&&& *(yellow) 1 \\
*(lightgray)  & \circ && \circ & *(yellow) 3 \\
*(lightgray)  &&& *(yellow) 1 \\
*(lightgray) \to  & \circ & *(yellow) 5 \\
  & *(yellow)  2 \\
*(yellow)  4
\end{ytableau}.
\end{eq*}
The four edges contributing to $\asc$ have been marked with $\circ$,
and this coloring contributes with $q^4 x_1^2 x_2^2 x_3 x_4 x_5$
to the sum in \refq{lltDefinition}.
\end{exa}

\subsection{Vertical-strip LLT polynomials}

For $P \in \SCH{n}$
the \defin{vertical-strip LLT polynomial} $\LLT_P(\xvec;q)$ is defined as
\begin{eq}{lltDefinition}
 \mathdefin{\LLT_P(\xvec;q)} \coloneqq 
 \sum_{\kappa} q^{\asc(\kappa)} x_{\kappa(1)} x_{\kappa(2)} \dotsm x_{\kappa(n)},
\end{eq}
where the sum ranges over all colorings of $\Gamma_P$.
One can show that $\LLT_P(\xvec;q)$ is a symmetric function and
a straightforward proof of this fact is given in \cite{AlexanderssonPanova2016},
which is an adaptation of the proof given in the appendix of \cite{HaglundHaimanLoehr2005}.
Alternatively this also follows from the $\elementaryE$-expansion provided in this paper.
Every LLT polynomial indexed by a $k$-tuple of vertical-strip skew shapes in the 
Bylund--Haiman model can be realized as $\LLT_P(\xvec;q)$ for some $P\in\SCH{n}$,
see \cite{CarlssonMellit2017,AlexanderssonPanova2016}.
The family of vertical-strip LLT polynomials indexed by Dyck paths are referred to as \defin{unicellular LLT polynomials}.

The vertical-strip LLT polynomials satisfy the recursion in \reft{recursion}.

\begin{thm}{LLT}
Let $F:\SCH{}\to\Lambda$ be defined by $F_P\xqvar = \LLT_P\xqvar$.
Then $F$ satisfies the conditions in \refti{recursion}{e}--\refi{recursion:unicellular}.
Moreover let $P\in\SCH{}$ be a Schr{\"o}der path, and let $(x,z)\in\setZ^2$ be a point on $P$ with $x+1<z$ such that the bounce decomposition of $P$ at $(x,z)$ is given by 
$P=Us \bs tV\dstep \bs \estep W$
for some $st\in\{\nstep\nstep,\dstep\nstep,\nstep\dstep\}$.
Then
\begin{eq*}
F_{P}\xqvar
=
\begin{cases}
qF_{U\nstep\nstep V\estep\dstep W}\xqvar
&\quad\text{if }st=\nstep\nstep,
\\
F_{U\nstep\dstep V\estep\dstep W}\xqvar
&\quad\text{if }st=\dstep\nstep
, \\
(q-1)F_{U\nstep\dstep V\estep\dstep W}\xqvar
+
qF_{U\dstep\nstep V\estep\dstep W}\xqvar
&\quad\text{if }st=\nstep\dstep
.
\end{cases}
\end{eq*}
\end{thm}

A similar result was shown by M.~D'Adderio in \cite{DAdderio2020} 
in the language of the Dyck path algebra.
D'Adderio in turn built on the results that E.~Carlsson and A.~Mellit 
derived in their proof of the Shuffle Theorem \cite{CarlssonMellit2017}.
In \refs{colorings} we give an alternative proof of the fact that the vertical-strip LLT polynomials satisfy the conditions of \reft{recursion}.
This proof makes use of the combinatorics of colorings of unit-interval graphs and bijective arguments 
rather than relations between operators in the Dyck path algebra.
Note that \reft{LLT} contains more general relations than \reft{recursion}.
In particular it contains a third type of bounce relation which is illustrated by
\begin{eq}{LLT_caseIII_diagram}
\begin{tikzpicture}[scale=0.5]
\draw pic {chromatic};
\draw pic {chrom-bounceLL};
\draw pic {chrom-nd};
\draw pic {chrom-de};
\node at (5,-1.1) {$U{\nstep\bs\dstep}V{\dstep\bs\estep}W$};
\end{tikzpicture}
\quad
&=
\quad
(q-1)\times
\begin{tikzpicture}[scale=0.5]
\draw pic {chromatic};
\draw pic {chrom-bounceSL};
\draw pic {chrom-nd};
\draw pic {chrom-ed};
\node at (5,-1.1) {$U{\nstep\bs\dstep}V{\estep\bs\dstep}W$};
\end{tikzpicture}
\quad
+
\quad
q \times
\begin{tikzpicture}[scale=0.5]
\draw pic {chromatic};
\draw pic {chrom-bounceSS};
\draw pic {chrom-dn};
\draw pic {chrom-ed};
\node at (5,-1.1) {$U{\dstep\bs\nstep}V{\estep\bs\dstep}W$};
\end{tikzpicture}
.
\end{eq}
A proof of \reft{LLT} is achieved by combining the results of \refs{colorings} with \refp{equivalentRelations}, which deduces the third bounce relation from the first two, and with \reft{bounce_reduction}, which yields bounce relations with arbitrarily many bounce points.

\reft{LLT} clearly implies the existence statement in \reft{recursion}.
Moreover Theorems~\ref{Theorem:recursion} and~\ref{Theorem:LLT} 
provide a new characterization of vertical-strip LLT polynomials.
\begin{cor}{LLT}
The vertical-strip LLT-polynomials $\LLT_P\xqvar$ are uniquely determined 
by the initial condition and relations in \reft{recursion}.
\end{cor}

In this paper we make use of this characterization to prove an 
explicit positive expansion of $\LLT_P(\xvec,q+1)$ in terms of elementary symmetric functions.

\subsection{Orientations}

Let $\Gamma=(V,E)$ be a graph.
An \defin{orientation} of $\Gamma$ is a function $\theta:E\to V^2$ that assigns 
to each edge $uv\in E$ a directed edge $\dedge{u}{v}$ or $\dedge{v}{u}$.
Alternatively, we may view
$\theta=\{\theta(e):e\in E\}\subseteq V^2$
as a set that contains a unique directed edge for each edge in $E$.
We adopt the convention of viewing orientations as sets, 
which is more convenient if we want to add or remove edges.
Let $\Gamma'=(V,E')$ be a subgraph of $\Gamma$, that is, $E' \subseteq E$.
An orientation $\theta$ of $\Gamma$ yields an orientation $\theta'$ of $\Gamma'$ defined by
\begin{eq*}
\theta'
=
\{\dedge{u}{v}\in\theta:\edge{u}{v}\in E'\}
.
\end{eq*}
We call $\theta'$ the \defin{restriction} of $\theta$ to $E'$.
If $V=[n]$ then $\Gamma$ is equipped with the \defin{natural orientation} that assigns to each edge $uv$ the directed edge $\dedge{u}{v}$ where $u<v$.

Let $P\in\SCH{n}$ and let $\uig_P =(V,E,S)$ be the corresponding decorated unit-interval graph.
We let $\mathdefin{\ORIENTS(P)}$ denote the set of orientations $\theta$
of $\uig_P$ such that the restriction of $\theta$ to $S$ is the natural orientation.
Given $\theta \in \ORIENTS(P)$, an edge $\dedge{u}{v}\in\theta$ 
is an \defin{ascending edge} in $\theta$
if $u<v$ and $\edge{u}{v}\notin S$.
We let $\mathdefin{\asc(\theta)}$ denote the number of ascending edges in $\theta$.

\begin{exa}{decorated_orientation}
We use the diagram notation as in \reff{schroder_path_and_graph},
to illustrate an orientation $\theta$.
The ascending edges in $\theta$ are marked with $\to$ in the diagram,
and $\asc(\theta) = 5$.
\begin{eq*}
\ytableausetup{boxsize=1.0em}
\begin{ytableau}
*(lightgray)  &*(lightgray)&*(lightgray)&*(lightgray) \to&&& *(yellow) 7 \\
*(lightgray)  &*(lightgray) &&\to&\to& *(yellow) 6 \\
*(lightgray)  &\to&\to&& *(yellow) 5 \\
*(lightgray)  &&\to& *(yellow) 4 \\
*(lightgray) \to  &  & *(yellow) 3 \\
  & *(yellow)  2 \\
*(yellow)  1
\end{ytableau}
,\quad
\theta = \{ 
\dedge{2}{1},
\dedge{1}{3}, 
\dedge{3}{2},
\dedge{4}{2},
\dedge{2}{5},
\dedge{3}{4},
\dedge{3}{5},
\dedge{6}{3},
\dedge{5}{4},
\dedge{4}{6},
\dedge{4}{7},
\dedge{5}{6},
\dedge{7}{5},
\dedge{7}{6}
\}
\end{eq*}
\end{exa}

\subsection{The highest reachable vertex}
Let $\theta\in\ORIENTS(P)$.
For a vertex $u$ of $\uig_P$ the \defin{highest reachable vertex}, denoted $\mathdefin{\hrv{\theta}{u}}$,
is defined as the maximal $v$ such that there is a directed path from $u$ to $v$ in $\theta$ 
\emph{using only strict and ascending edges}.
By definition $\hrv{\theta}{u}\geq u$ for all $u$.
The orientation $\theta$ defines a set partition $\mathdefin{\hrvp(\theta)}$ of the vertices of $\Gamma_P$,
where two vertices belong to the same block if and only if they have the same highest reachable vertex.
Finally, let $\mathdefin{\hrvpp(\theta)}$ denote the integer 
partition given by the sizes of the blocks in $\hrvp(\theta)$.

\begin{exa}{LLTe}[{Taken from \cite{Alexandersson2019llt}}]
Below, we illustrate an orientation $\theta \in O(P)$,
where $P = \mathtt{nnnddeneee}$.
Strict edges and edges contributing to $\asc(\theta)$ are marked with $\to$.
\[
\ytableausetup{boxsize=1em,centertableaux}
\begin{ytableau}
*(lightgray)   & *(lightgray)   & *(lightgray)   &      & \to   & *(yellow) 6 \\
*(lightgray)   & *(lightgray)  \rightarrow &     &     & *(yellow) 5\\
*(lightgray)   \rightarrow &  \to &  \to   & *(yellow) 4\\
  \to  & \to    & *(yellow) 3\\
    & *(yellow) 2\\
*(yellow) 1
\end{ytableau}
\]
We have that $\hrv{\theta}{2}=\hrv{\theta}{5}=\hrv{\theta}{6}=6$ and $\hrv{\theta}{1}=\hrv{\theta}{3}=\hrv{\theta}{4}=4$.
Thus $\hrvp(\theta) = \{652,431\}$ and the orientation $\theta$ contributes 
with $q^{5}\elementaryE_{33}(\xvec)$ in \refq{conjFormula}.
\end{exa}

\subsection{An explicit \texorpdfstring{$e$}{e}-expansion}

Given a Schr{\"o}der path $P \in \SCH{n}$ define the symmetric function \defin{$\LLTe_{P}(\xvec;q)$} by
\begin{eq}{conjFormula}
\mathdefin{\LLTe_{P}(\xvec;q+1)} \coloneqq \sum_{\theta \in \ORIENTS(P)} q^{\asc(\theta)}
\elementaryE_{\hrvpp(\theta)}(\xvec).
\end{eq}
The second main result of this paper states that the symmetric 
functions defined in \refq{conjFormula} satisfy the initial conditions and relations in \reft{recursion}.

\begin{thm}{LLTe}
Let $F:\SCH{}\to\Lambda$ be defined by $F_P\xqvar=\LLTe_P\xqvar$.
Then $F$ satisfies the conditions in \reft{recursion}.
\end{thm}

Combining Theorems~\ref{Theorem:recursion}, \ref{Theorem:LLT} and~\ref{Theorem:LLTe} we obtain the first explicit expansion of the vertical-strip LLT polynomials in terms of elementary symmetric functions.

\begin{cor}{expansion}
For all $P\in\SCH{}$ we have $\LLT_P\xqvar=\LLTe_P\xqvar$.
In particular $\LLT_P(\xvec;q+1)$ expands positively into elementary symmetric functions.
\end{cor}

The positivity result in \refc{expansion} was obtained by M.~D'Adderio in~\cite{DAdderio2020}.
The formula in \refq{conjFormula} was conjectured 
in \cite{GarsiaHaglundQiuRomero2019} and independently by the 
first named author in \cite{Alexandersson2019llt}.
Combinatorial positive expansions into elementary symmetric functions were obtained in a few special cases in \cite{AlexanderssonPanova2016,Alexandersson2019llt}, albeit with 
a different statistic in place of $\hrvpp(\theta)$.

It is worth noting that the right-hand side of \refq{conjFormula} is manifestly symmetric.
In contrast, previously known formulas for LLT polynomials in terms of power-sum 
symmetric functions or Schur functions (with signs) rely on the expansion 
into fundamental quasisymmetric functions.
To derive these formulas one needs to use the fact that LLT polynomials 
are symmetric in the first place, and not just quasisymmetric.

\begin{rem}{AbreuNigro}
A stronger version of \refq{conjFormula} has been proved recently in 
\cite{AbreuNigro2020bx,AbreuNigro2020ax}. The authors mainly consider
the unicellular case, and they use the modular law and its dual in their proof.
\end{rem}

\section{Bounce relations}
\label{Section:bounce}

In this section we prove the uniqueness statement in \reft{recursion}.
In a first step we show that the bounce relations in \refti{recursion}{bounce} are equivalent to, and hence imply, other types of bounce relations.
This is the content of \refp{equivalentRelations}.
We then prove the presence of bounce relations with arbitrarily many bounce points in \reft{bounce_reduction}.
With these more general relations at our disposal the proof of the uniqueness statement can be carried out.
We conclude the section with a characterization of unicellular LLT polynomials that uses bounce relations for Dyck paths in \reft{dyckRecursion}. 

Recall that we use dots in the bounce decomposition to highlight the starting point and endpoint of the bounce path in consideration.
Note that the location of these dots does not change the word encoding a path.
Moreover, slightly different choices of starting points may result (essentially) in the same bounce path.
For example, 
\[
U\nstep\estep\bs\nstep V\estep \nstep\bs\estep W
\text{ and }
U\nstep\estep\bs\nstep V\estep \bs \nstep\estep W
\]
describe the same Schr{\"o}der path, and the indicated bounce paths 
have all bounce points in common.

\begin{prop}{equivalentRelations}
Let $A$ be a $\field$-algebra, and let $F:\SCH{}\to A$, $P\mapsto F_P\xqvar$ be a function 
that satisfies the unicellular relation, that is,
\begin{enumerate}[(i)]
\setcounter{enumi}{2}
\myi{equivalentRelations:unicellular}
For all $U,V\in\{\nstep,\dstep,\estep\}^*$ such that $U\dstep V\in\SCH{}$ we have
$F_{U\nstep\estep V}\xqvar -  F_{U\estep\nstep V}\xqvar = (q-1)F_{U\dstep V}\xqvar$.
\end{enumerate}

Then the following additional sets of relations are equivalent:
\begin{enumerate}[(i)]
\setcounter{enumi}{3}
\myi{equivalentRelations:A}
\textbf{Schr{\"o}der Relations A.}
Let $P\in\SCH{}$ be a Schr{\"o}der path, and let $(x,z)\in\setZ^2$ be a point on $P$ with $x+1<z$ such that the bounce path of $P$ at $(x,z)$ has only one bounce point, and such that the bounce decomposition of $P$ at $(x,z)$ is given by 
$P=Us\bs tV\dstep\bs \estep W$
for some $st\in\{\nstep\nstep,\dstep\nstep
\}$.
Then
\begin{numcases}
{F_{P}\xqvar=}
q F_{U\nstep\bs\nstep V \estep\bs\dstep W}\xqvar
&\quad\text{if }st=\nstep\nstep
,  
\label{Equation:BounceIa}
\\
F_{U\nstep\bs\dstep V \estep\bs\dstep W}\xqvar
&\quad\text{if }st=\dstep\nstep.
\label{Equation:BounceIIa}
\end{numcases}

\myi{equivalentRelations:B}
\textbf{Schr{\"o}der Relations B.}
Let $P\in\SCH{}$ be a Schr{\"o}der path, and let $(x,z)\in\setZ^2$ be a point on $P$ with $x+1<z$ such that the bounce path of $P$ at $(x,z)$ has only one bounce point, and such that the bounce decomposition of $P$ at $(x,z)$ is given by $P=Us\bs tV\dstep\bs\estep W$ for some $st\in\{\nstep\nstep,\nstep\dstep\}$.
Then
\begin{numcases}
{F_{P}\xqvar=}
q F_{U\nstep\bs\nstep V \estep\bs\dstep W}\xqvar
&\quad\text{if }st=\nstep\nstep
,  \label{Equation:BounceIb} \\
(q-1)F_{U\nstep\bs\dstep V \estep\bs\dstep W}\xqvar +
q F_{U\dstep\bs\nstep V \estep\bs\dstep W}\xqvar
&\quad\text{if }st=\nstep\dstep
.
\label{Equation:BounceIIb}
\end{numcases}

\myi{equivalentRelations:dyck}
\textbf{Dyck Relations.}
Let $P\in\SCH{}$ be a Schr{\"o}der path, and let $(x,z)\in\setZ^2$ be a point on $P$ with $x+1<z$ such that the bounce path of $P$ at $(x,z)$ has only one bounce point.
If the bounce decomposition of $P$ at $(x,z)$ is given by $P=U\nstep\bs\nstep V\nstep\estep\bs\estep W$ then
\begin{eq}{dyckBounceI}
F_{U\nstep\bs\nstep V\nstep \estep\bs\estep W}\xqvar
&= (q+1)F_{U\nstep\bs\nstep V\estep\bs\nstep\estep W}\xqvar
- q F_{U\nstep\bs\nstep V\estep\bs\estep\nstep W}\xqvar.
\end{eq}
If the bounce decomposition of $P$ at $(x,z)$ is given by $P=U\nstep\estep\bs\nstep V\nstep\estep\bs\estep W$ then
\begin{eq}{dyckBounceII}
&F_{U{\estep\nstep\bs\nstep}V{\estep\nstep\bs\estep}W}\xqvar
+
F_{U{\nstep\estep\bs\nstep}V{\nstep\estep\bs\estep}W}\xqvar
+
F_{U{\nstep\bs\nstep\estep}V{\estep\bs\estep\nstep}W}\xqvar = \\
&F_{U{\estep\nstep\bs\nstep}V{\nstep\estep\bs\estep}W}\xqvar
+
F_{U{\nstep\estep\bs\nstep}V{\estep\bs\estep\nstep}W}\xqvar
+
F_{U{\nstep\bs\nstep\estep}V{\estep\bs\nstep\estep}W}\xqvar.
\end{eq}
\end{enumerate}
\end{prop}

\begin{rem}{cross-product}
Curiously \refq{dyckBounceII} is reminiscent of a cross-product and can be expressed as the following vanishing determinant, where multiplication is now replaced by concatenation:
\begin{eq*}
\begin{vmatrix}
U & U & U \\
\estep\nstep\nstep V & \nstep\estep\nstep V & \nstep\nstep\estep V \\
\estep\estep\nstep W & \estep\nstep\estep W & \nstep\estep\estep W
\end{vmatrix}
=0.
\end{eq*}
Expanding this determinant according to Sarrus' rule gives 
exactly the terms in \refq{dyckBounceII} with correct sign.
\end{rem}

\begin{proof}[Proof of \refp{equivalentRelations}]
We first show that \refq{BounceIa} and \refq{BounceIb} are equivalent to \refq{dyckBounceI}.
We start with \refq{dyckBounceI},
\begin{eq*}
F_{U\nstep\bs\nstep V\nstep \estep\bs\estep W}
&=
(q+1)F_{U\nstep\bs\nstep V\estep\bs\nstep\estep W}
- q F_{U\nstep\bs\nstep V\estep\bs\estep\nstep W}
.
\end{eq*}
Bringing one term to the other side we obtain
\begin{eq*}
F_{U\nstep\bs\nstep V\nstep\estep\bs\estep W} - F_{U\nstep\bs\nstep V\estep\nstep\bs\estep W}
&=
q(F_{U\nstep\bs\nstep V\estep\nstep\bs\estep W}-F_{U\nstep\bs\nstep V\estep\bs\estep\nstep W}) 
.
\end{eq*}
Applying the unicellular relation~\refi{equivalentRelations:unicellular} on both sides we obtain
\begin{eq*}
(q-1) F_{U\nstep\bs\nstep V\dstep\bs\estep W}
&=
q (q-1) F_{U\nstep\bs\nstep V\estep\bs\dstep W}
,
\end{eq*}
which after the cancellation of a common factor becomes \refq{BounceIa}.
Clearly each step is reversible and we obtain the desired equivalence.

\medskip 
Next we show that \refq{BounceIIa} is equivalent to \refq{dyckBounceII}. 
We start with \refq{dyckBounceII}, which is rearranged as
\begin{eq*}
\begin{aligned}
\begin{tikzpicture}[scale=0.5]
\draw pic {chromatic};
\draw pic {chrom-bounceLS};
\draw pic {chrom-nen};
\draw pic {chrom-nee};
\node at (5,-1.1) {$U{\nstep\estep\bs\nstep}V{\nstep\estep\bs\estep}W$};
\end{tikzpicture}
-
\begin{tikzpicture}[scale=0.5]
\draw pic {chromatic};
\draw pic {chrom-bounceSS};
\draw pic {chrom-nen};
\draw pic {chrom-een};
\node at (5,-1.1) {$U{\nstep\estep\bs\nstep}V{\estep\bs\estep\nstep}W$};
\end{tikzpicture}
&=
\begin{tikzpicture}[scale=0.5]
\draw pic {chromatic};
\draw pic {chrom-bounceSS};
\draw pic {chrom-enn};
\draw pic {chrom-ene};
\node at (5,-1.1) {$U{\estep\nstep\bs\nstep}V{\estep\nstep\bs\estep}W$};
\end{tikzpicture}
-
\begin{tikzpicture}[scale=0.5]
\draw pic {chromatic};
\draw pic {chrom-bounceLS};
\draw pic {chrom-enn};
\draw pic {chrom-nee};
\node at (5,-1.1) {$U{\estep\nstep\bs\nstep}V{\nstep\estep\bs\estep}W$};
\end{tikzpicture}+ \\
&
+
\begin{tikzpicture}[scale=0.5]
\draw pic {chromatic};
\draw pic {chrom-bounceSL};
\draw pic {chrom-nne};
\draw pic {chrom-een};
\node at (5,-1.1) {$U{\nstep\bs\nstep\estep}V{\estep\bs\estep\nstep}W$};
\end{tikzpicture}
-
\begin{tikzpicture}[scale=0.5]
\draw pic {chromatic};
\draw pic {chrom-bounceSL};
\draw pic {chrom-nne};
\draw pic {chrom-ene};
\node at (5,-1.1) {$U{\nstep\bs\nstep\estep}V{\estep\bs\nstep\estep}W$};
\end{tikzpicture}.
\end{aligned}
\end{eq*}
Applying the unicellular relation twice on the right-hand side we get
\begin{eq*}
\begin{aligned}
\begin{tikzpicture}[scale=0.5]
\draw pic {chromatic};
\draw pic {chrom-bounceLS};
\draw pic {chrom-nen};
\draw pic {chrom-nee};
\node at (5,-1.1) {$U{\nstep\estep\bs\nstep}V{\nstep\estep\bs\estep}W$};
\end{tikzpicture}
\!
-
\!
\begin{tikzpicture}[scale=0.5]
\draw pic {chromatic};
\draw pic {chrom-bounceSS};
\draw pic {chrom-nen};
\draw pic {chrom-een};
\node at (5,-1.1) {$U{\nstep\estep\bs\nstep}V{\estep\bs\estep\nstep}W$};
\end{tikzpicture}
&=
(q-1)
\times
\begin{tikzpicture}[scale=0.5]
\draw pic {chromatic};
\draw pic {chrom-bounceLS};
\draw pic {chrom-enn};
\draw pic {chrom-de};
\node at (5,-1.1) {$U{\estep\nstep\bs\nstep}V{\dstep\bs\estep}W$};
\end{tikzpicture}+
\\
&\phantom{q}+
(q-1)\times
\begin{tikzpicture}[scale=0.5]
\draw pic {chromatic};
\draw pic {chrom-bounceSL};
\draw pic {chrom-nne};
\draw pic {chrom-ed};
\node at (5,-1.1) {$U{\nstep\bs\nstep\estep}V{\estep\bs\dstep}W$};
\end{tikzpicture}
.
\end{aligned}
\end{eq*}
After additional applications of the unicellular relation on the left-hand side
we get (after dividing by $q-1$)
\begin{eq}{dyckReduction}
\begin{tikzpicture}[scale=0.5]
\draw pic {chromatic};
\draw pic {chrom-bounceLS};
\draw pic {chrom-nen};
\draw pic {chrom-de};
\node at (5,-1.1) {$U{\nstep\estep\bs\nstep}V{\dstep\bs\estep}W$};
\end{tikzpicture}
+
\begin{tikzpicture}[scale=0.5]
\draw pic {chromatic};
\draw pic {chrom-bounceSS};
\draw pic {chrom-nen};
\draw pic {chrom-ed};
\node at (5,-1.1) {$U{\nstep\estep\bs\nstep}V{\estep\bs\dstep}W$};
\end{tikzpicture}
=
\begin{tikzpicture}[scale=0.5]
\draw pic {chromatic};
\draw pic {chrom-bounceLS};
\draw pic {chrom-enn};
\draw pic {chrom-de};
\node at (5,-1.1) {$U{\estep\nstep\bs\nstep}V{\dstep\bs\estep}W$};
\end{tikzpicture}
+
\begin{tikzpicture}[scale=0.5]
\draw pic {chromatic};
\draw pic {chrom-bounceSL};
\draw pic {chrom-nne};
\draw pic {chrom-ed};
\node at (5,-1.1) {$U{\nstep\bs\nstep\estep}V{\estep\bs\dstep}W$};
\end{tikzpicture}
.
\end{eq}
Rearranging the terms again we arrive at
\begin{eq*}
\begin{tikzpicture}[scale=0.5]
\draw pic {chromatic};
\draw pic {chrom-bounceLS};
\draw pic {chrom-nen};
\draw pic {chrom-de};
\node at (5,-1.1) {$U{\nstep\estep\bs\nstep}V{\dstep\bs\estep}W$};
\end{tikzpicture}
-
\begin{tikzpicture}[scale=0.5]
\draw pic {chromatic};
\draw pic {chrom-bounceLS};
\draw pic {chrom-enn};
\draw pic {chrom-de};
\node at (5,-1.1) {$U{\estep\nstep\bs\nstep}V{\dstep\bs\estep}W$};
\end{tikzpicture}
=
\begin{tikzpicture}[scale=0.5]
\draw pic {chromatic};
\draw pic {chrom-bounceSL};
\draw pic {chrom-nne};
\draw pic {chrom-ed};
\node at (5,-1.1) {$U{\nstep\bs\nstep\estep}V{\estep\bs\dstep}W$};
\end{tikzpicture}
-
\begin{tikzpicture}[scale=0.5]
\draw pic {chromatic};
\draw pic {chrom-bounceSS};
\draw pic {chrom-nen};
\draw pic {chrom-ed};
\node at (5,-1.1) {$U{\nstep\estep\bs\nstep}V{\estep\bs\dstep}W$};
\end{tikzpicture}
,
\end{eq*}
which after a final application of the unicellular 
relation on both sides gives \refq{BounceIIa}.

\medskip 
To complete the proof we now show that \refq{BounceIIb} is equivalent to \refq{dyckBounceII} assuming that \refq{BounceIb} and \refq{dyckBounceI} are valid.
We start with \refq{BounceIIb} which becomes
\begin{eq*}
\begin{tikzpicture}[scale=0.5]
\draw pic {chromatic};
\draw pic {chrom-bounceLL};
\draw pic {chrom-nd};
\draw pic {chrom-de};
\node at (5,-1.1) {$U{\nstep\bs\dstep}V{\dstep\bs\estep}W$};
\end{tikzpicture}
+
\begin{tikzpicture}[scale=0.5]
\draw pic {chromatic};
\draw pic {chrom-bounceSL};
\draw pic {chrom-nd};
\draw pic {chrom-ed};
\node at (5,-1.1) {$U{\nstep\bs\dstep}V{\estep\bs\dstep}W$};
\end{tikzpicture}
=
q
\begin{tikzpicture}[scale=0.5]
\draw pic {chromatic};
\draw pic {chrom-bounceSL};
\draw pic {chrom-nd};
\draw pic {chrom-ed};
\node at (5,-1.1) {$U{\nstep\bs\dstep}V{\estep\bs\dstep}W$};
\end{tikzpicture}
+
q
\begin{tikzpicture}[scale=0.5]
\draw pic {chromatic};
\draw pic {chrom-bounceSS};
\draw pic {chrom-dn};
\draw pic {chrom-ed};
\node at (5,-1.1) {$U{\dstep\bs\nstep}V{\estep\bs\dstep}W$};
\end{tikzpicture}
\end{eq*}
after moving one term to the left-hand side.
We now use the unicellular relation to eliminate one diagonal step in each path.
In the left-hand side, we keep the first diagonal step and eliminate the second one.
In the right-hand side we keep the second diagonal step.
After cancelling a common factor of $(q-1)$, the left-hand side simplifies to
\begin{eq*}
(F_{U\nstep\bs \dstep V\nstep \estep\bs\estep W}
-
F_{U\nstep\bs \dstep V\estep\bs \nstep\estep W}
)+(
F_{U\nstep\bs \dstep V\estep\bs\nstep\estep W}
&- 
F_{U\nstep\bs \dstep V\estep\bs\estep\nstep W}
)
\\
&=
F_{U\nstep\bs \dstep V\nstep \estep\bs\estep W}
-
F_{U\nstep\bs \dstep V\estep\bs\estep\nstep W}
\end{eq*}
Similarly the right-hand side becomes
\begin{align*}
q(
F_{U\nstep\bs \nstep\estep V\estep\bs\dstep W}
-
F_{U\nstep\estep\bs\nstep V\estep\bs\dstep W}
)
+
q(
F_{U\nstep\estep\bs\nstep V\estep\bs\dstep W}
&-
F_{U\estep\nstep \bs \nstep V\estep\bs\dstep W}
)
\\
&=
q(F_{U\nstep\bs \nstep\estep V\estep\bs\dstep W}
-
F_{U\estep\nstep \bs \nstep V\estep\bs\dstep W})
.
\end{align*}
Now we use \refq{BounceIb} on both terms in the right-hand side, that is, 
$
q F_{U\nstep\bs \nstep\estep V\estep\bs\dstep W} = F_{U\nstep\bs \nstep\estep V\dstep\bs\estep W}
$
and 
$
q F_{U\estep\nstep\bs \nstep V\estep\bs\dstep W} = F_{U\estep\nstep\bs \nstep V\dstep\bs\estep W}
$.
Moving all terms to the left-hand side, we obtain the relation
\begin{eq*}
F_{U\estep\nstep \bs \nstep V\dstep\bs\estep W}
-
F_{U\nstep\bs \nstep\estep V\dstep\bs\estep W}
+
F_{U\nstep\bs \dstep V\nstep \estep\bs\estep W}
- 
F_{U\nstep\bs \dstep V\estep\bs\estep\nstep W}
=0.
\end{eq*}
Finally, we eliminate the remaining diagonal steps using the unicellular relation.
After cancellation of a common factor this gives the identity
\begin{align*}
0 &= \left(
F_{U\estep\nstep \bs \nstep V\nstep\estep\bs\estep W}
-
F_{U\estep\nstep \bs \nstep V\estep\bs\nstep\estep W}
\right) 
-
\left(
\underline{F_{U\nstep\bs\nstep\estep V\nstep\estep\bs\estep W}}
-
F_{U\nstep\bs \nstep\estep V\estep\bs\nstep\estep W}
\right) \\
&\phantom{=}+
\left(
\underline{F_{U\nstep\bs\nstep\estep V\nstep\estep\bs\estep W}}
-
F_{U\nstep\estep\bs \nstep V\nstep \estep\bs\estep W}
\right) 
- 
\left(
F_{U\nstep\bs \nstep\estep V\estep\bs\estep\nstep W}
-
F_{U\nstep\estep\bs\nstep V\estep\bs\estep\nstep W}
\right).
\end{align*}
After cancelling the underlined terms we arrive at \refq{dyckBounceII}.
Since all steps in this deduction are invertible we have the desired equivalence.
\end{proof}

Our next goal is to show that the relations in \reft{recursion} 
are strong enough to impose relations involving bounce paths with \emph{arbitrarily many} bounce points.

\begin{thm}{bounce_reduction}
Let $A$ be a $\field$-algebra, and let $F:\SCH{}\to A$, $P\mapsto F_P\xqvar$ be a function that 
satisfies Conditions~\refi{recursion:unicellular} and~\refi{recursion:bounce} in \reft{recursion}.
Let $P\in\SCH{}$ be a Schr{\"o}der path and let $(x,z)\in\setZ^2$ be 
a point on $P$ with $z>x+1$ such that the bounce decomposition of $P$ at $(x,z)$ is given by 
$P=Us\bs tV\dstep \bs \estep W$
for some $V\in\{\nstep,\dstep\}^*$ and some $st\in\{\nstep\nstep,\nstep\dstep,\dstep\nstep\}$.
Then
\begin{eq*}
F_{P}\xqvar
=
\begin{cases}
qF_{U\nstep\nstep V\estep\dstep W}\xqvar
&\quad\text{if }st=\nstep\nstep,
\\
F_{U\nstep\dstep V\estep\dstep W}\xqvar
&\quad\text{if }st=\dstep\nstep,
\\
(q-1)F_{U\nstep\dstep V\estep\dstep W}\xqvar
+
qF_{U\dstep\nstep V\estep\dstep W}\xqvar
&\quad\text{if }st=\nstep\dstep.
\end{cases}
\end{eq*}
\end{thm}
\begin{proof}
The claim is shown by induction on the number of bounce points of the bounce path of $P$ at $(x,z)$.
First assume that the bounce path of $P$ at $(x,z)$ has only one bounce point.
If $st\in\{\nstep\nstep,\dstep\nstep\}$ then there is nothing to show.
If $st=\nstep\dstep$ then the claim follows from \refp{equivalentRelations}.
To see this note that the proof of the equivalence of \refpi{equivalentRelations}{A} and~\refi{equivalentRelations:B} goes through in the same way if we impose the 
restriction $V\in\{\nstep,\dstep\}^*$ in both cases.
Thus the base case of the induction is taken care of.

Otherwise the bounce partition of $P$ at $(x,z)$ has length $r>3$ and is given by $(z,x,v,\dotsc)$.
Set $y=x+1$ and $w=v+1$.
In this case there are at least two bounce points and we may write
\begin{eq*}
P=Us\bs tV'\dstep\dstep V''\dstep\bs\estep W,
\end{eq*}
where $V',V''\in\{\nstep,\dstep\}^*$ such that $\dstep V''\dstep$ defines a path from $(v,x)$ to $(x,z)$.
Note that here we use the fact that $z>y$, which also implies $x>w$.
By the unicellular relation 
we have
\begin{eq}{LLTe_P}
F_{P}\xqvar
=
\frac{1}{q-1}
\Big(
F_{Us tV'\dstep\bs \nstep\estep V''\dstep\bs \estep W}\xqvar
-
F_{Us\bs tV'\dstep\bs \estep\nstep V''\dstep \estep W}\xqvar
\Big),
\end{eq}
which is illustrated by diagrams as
\begin{eq}{LLTe_P_diagram}
\includegraphics[page=1,scale=0.8,valign=c]{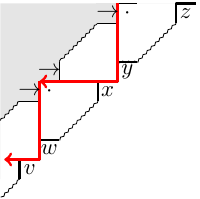}%
=
\frac{1}{q-1}
\includegraphics[page=2,scale=0.8,valign=c]{figuresBounceRelations.pdf}%
-
\frac{1}{q-1}
\includegraphics[page=3,scale=0.8,valign=c]{figuresBounceRelations.pdf}%
.
\end{eq}

To simplify notation suppose for the moment that $st=\dstep\nstep$.
By applying first \refti{recursion}{bounce} at the point $(x,z)$, 
and then the induction hypothesis at the point $(v,y)$ we obtain
\begin{eq}{LLTe_1}
F_{U\dstep\nstep V'\dstep\bs\nstep\estep V''\dstep\bs\estep W}\xqvar
&=
F_{U\dstep\bs\nstep V'\nstep\bs\dstep\estep V''\estep\bs\dstep W}\xqvar
=
F_{U\nstep\dstep V'\nstep\estep\dstep V''\estep\dstep W}\xqvar
,
\end{eq}
which corresponds to
\begin{eq}{LLTe_1_diagram}
\includegraphics[page=4,scale=0.8,valign=c]{figuresBounceRelations.pdf}%
=
\includegraphics[page=5,scale=0.8,valign=c]{figuresBounceRelations.pdf}%
=
\includegraphics[page=6,scale=0.8,valign=c]{figuresBounceRelations.pdf}%
.
\end{eq}
Note that the bounce path of $Us\bs tV'\nstep\dstep\estep V''\estep\bs\dstep W$ at $(v,y)$ 
has the same bounce points as the bounce path of $P$ at $(x,z)$ except for $(x,x)$.
Thus the induction hypothesis can be applied.

Similarly, by applying first induction hypothesis at the 
point $(v,x)$, and then \refti{recursion}{bounce} at the point $(x,z)$ the we obtain
\begin{eq}{LLTe_2}
F_{U\dstep\bs\nstep V'\dstep\bs\estep\nstep V''\dstep\estep W}\xqvar
&=
F_{U\nstep\dstep V'\estep\dstep\bs\nstep V''\dstep\bs\estep W}\xqvar
=
F_{U\nstep\dstep V'\estep\nstep\dstep V''\estep\dstep W}\xqvar
.
\end{eq}
This corresponds to
\begin{eq}{LLTe_2_diagram}
\includegraphics[page=7,scale=0.8,valign=c]{figuresBounceRelations.pdf}%
=
\includegraphics[page=8,scale=0.8,valign=c]{figuresBounceRelations.pdf}%
=
\includegraphics[page=9,scale=0.8,valign=c]{figuresBounceRelations.pdf}%
.
\end{eq}
Note that the bounce path of $Us\bs tV'\dstep\estep\nstep V''\estep\bs\dstep W$ 
at $(v,x)$ has the same bounce points as the bounce path of $P$ at $(x,z)$ except for $(x,x)$.
Thus the induction hypothesis can be applied.

Combining~\refq{LLTe_P}, \refq{LLTe_1} and~\refq{LLTe_2}, 
and using the unicellular relation 
one more time, we obtain
\begin{eq*}
F_{P}\xqvar
&=
\frac{1}{q-1}
\Big(
F_{U\nstep\dstep V'\nstep\estep\dstep V''\estep\dstep W}\xqvar
-
F_{U\nstep\dstep V'\estep\nstep\dstep V''\estep\dstep W}\xqvar
\Big)
\\&=
F_{U\nstep\bs\dstep V'\dstep\dstep V''\estep\bs\dstep W}\xqvar
\\&=
F_{U\nstep\bs\dstep V\estep\bs\dstep W}\xqvar
.
\end{eq*}
This proves the claim in the case $st=\dstep\nstep$.
The other two cases $st\in\{\nstep\nstep,\nstep\dstep\}$ follow in the exact same manner.
\end{proof}

The uniqueness in \reft{recursion} can now be shown using an intricate induction 
argument devised by M.~D'Adderio in the proof of \cite[Thm.~5.4]{DAdderio2020}.
Given a Schr{\"o}der path $P$ our strategy is to show that $F_P$ can 
be expressed in terms of elementary symmetric functions as well as 
functions $F_Q$ where $Q$ has size less than $P$, or $Q$ has fewer east 
steps than $P$, or $Q$ is obtained from $P$ by moving an east step of $P$ to the left.

\begin{proof}[Proof of uniqueness in \reft{recursion}.]
Let $P$ be a Schr{\"o}der path of size $n$ with $r$ east steps, and 
assume that $P=X\estep W$ where $X\in\{\nstep,\dstep\}^{*}$ defines 
a path from $(0,0)$ to $(x,z)$, and $W\in\{\nstep,\dstep,\estep\}$.
That is, $X$ is the initial segment of $P$ that leads up to the first east step of $P$.

We first use induction on the size $n$.
By \refti{recursion}{e} the only path of size $1$, namely $\nstep\estep$, is sent to $1$ by $F$.
Thus we may assume that $n>1$ and that $F_Q$ is 
uniquely determined for all $Q\in\SCH{m}$ for all $m<n$.

Secondly we use induction on $r$.
If $P$ has only one east step, then $P$ is of the 
form $P=\nstep\dstep^{n-1}\estep$ and $F_P=\elementaryE_{n}(\xvec)$ is 
determined by the initial condition in \refti{recursion}{e}.
Thus the base case of the second induction is taken care of.
We may assume that $P$ has at least two east steps, and that $F_Q$ 
is determined uniquely for all Schr{\"o}der paths $Q\in\SCH{n}$ with fewer than $r$ east steps.

Thirdly we use induction on $z$.
The base case for the third induction is $z=x+1$.
If $z=x+1$ then $X\estep$ and $W$ are Schr{\"o}der paths of size less than $n$.
Note that here we use the fact that $P$ contains two or more east steps. 
By \refti{recursion}{mult} and by induction on $n$ it follows that $F_P=F_{X\estep}F_W$ is determined uniquely.
Thus we may assume that $z>x+1$ and that $F_Q$ is determined uniquely 
for all $Q\in\SCH{n}$ with $r$ east steps that are of the form $Q=X'\estep W'$ 
where $X'\in\{\nstep,\dstep\}^*$ defines a path from $(0,0)$ to $(x',z')$ with $z'<z$.
Let $L\subseteq\SCH{n}$ denote the set of all such Schr{\"o}der paths.

Now if $X=Y\nstep$, that is $P=Y\nstep\estep W$, then $Y\dstep W$ 
is a Schr{\"o}der path of size $n$ with fewer east steps than $P$.
Note that here we use the fact that $z>x+1$.
Similarly $Y\estep\nstep W$ is a Schr{\"o}der path in $L$.
By \refti{recursion}{unicellular} and by induction on $r$ and $z$ it follows that $F_P=(q-1)F_{Y\dstep W}+F_{Y\estep\nstep W}$ is determined uniquely.

If on the other hand $X=Y\dstep$, that is $P=Y\dstep\estep W$, 
then let $P=Us\bs tV\dstep\bs\estep W$ be the bounce decomposition of $P$ at $(x,z)$.
Clearly $V\in\{\nstep,\dstep\}^*$ and $st\in\{\nstep\nstep,\nstep\dstep,\dstep\nstep\}$.
Thus by \reft{bounce_reduction} we may write $F_P$ as a linear 
combination of functions $F_Q$ for certain paths $Q\in L$.
Therefore $F_P$ is determined uniquely by induction on $z$.
This completes the proof.
\end{proof}

\begin{rem}{otherRecursion}
The conditions in \reft{recursion} were chosen because they are both natural and relatively easy to verify in the following sections.
However, there are various equivalent formulations that might be more practical in other contexts.
By \refp{equivalentRelations} one can replace the relations in \refti{recursion}{bounce} by one of the other sets of bounce relations given in \refpi{equivalentRelations}{B} or~\refi{equivalentRelations:dyck}.
Moreover, the multiplicativity condition \refti{recursion}{mult} can be dropped if one replaces \refti{recursion}{e} by the stronger initial condition that $F_P=e_{\lambda}$ for all partitions $\lambda$, where $P=\nstep\dstep^{\lambda_1-1}\estep\nstep\dstep^{\lambda_{2}-1}\estep\cdots\nstep\dstep^{\lambda_{\ell}-1}\estep$.
In this case a similar induction argument as in the proof above works.
We refer again to the proof of \cite[Thm.~5.4]{DAdderio2020}.
\end{rem}




Moreover we can prove a Dyck path analog of \reft{recursion}.
Let $\DYCK{}$ denote the set of all Dyck paths, that is, the set 
of Schr{\"o}der paths that contain no diagonal steps.
Note that using our conventions, if $P$ is a Dyck path and $(x,z)$ 
is a point on $P$ then the bounce path of $P$ at $(x,z)$ must have a single bounce point.

\begin{thm}{dyckRecursion}
The function that assigns to each Dyck path the corresponding unicellular LLT polynomial is the unique function $F:\DYCK{}\to\Lambda$, $P\mapsto F_P\xqvar$ that satisfies the following conditions:

\begin{enumerate}[(i)]
\myi{dyckRecursion:e}
For all $k\in\setN$ there holds the initial condition
\begin{eq*}
F_{\nstep(\nstep\estep
)^{k}\estep}(\xvec;q)
=
\sum_{\alpha \vDash k+1} (q-1)^{k-\length(\alpha)+1}\elementaryE_\alpha(\xvec)
.
\end{eq*}
\myi{dyckRecursion:mult}
The function $F$ is multiplicative, that is,
$F_{PQ}\xqvar
=
F_{P}\xqvar
F_{Q}\xqvar
$
for all $P,Q\in\DYCK{}$.
\setcounter{enumi}{5}
\myi{dyckRecursion:bounce}
Let $P\in\DYCK{}$ be a Dyck path and let $(x,z)\in\setZ^2$ be a point on $P$ with $x+1<z$.
If the bounce decomposition of $P$ at $(x,z)$ is given by $P=U\nstep\bs\nstep V\nstep\estep\bs\estep W$ then
\begin{eq*}
F_{U\nstep\bs\nstep V\nstep \estep\bs\estep W}\xqvar
&= (q+1)F_{U\nstep\bs\nstep V\estep\bs\nstep\estep W}\xqvar
- q F_{U\nstep\bs\nstep V\estep\bs\estep\nstep W}\xqvar.
\end{eq*}
If the bounce decomposition of $P$ at $(x,z)$ is given by $P=U\nstep\estep\bs\nstep V\nstep\estep\bs\estep W$ then
\begin{eq*}
&F_{U{\estep\nstep\bs\nstep}V{\estep\nstep\bs\estep}W}\xqvar
+
F_{U{\nstep\estep\bs\nstep}V{\nstep\estep\bs\estep}W}\xqvar
+
F_{U{\nstep\bs\nstep\estep}V{\estep\bs\estep\nstep}W}\xqvar = \\
&F_{U{\estep\nstep\bs\nstep}V{\nstep\estep\bs\estep}W}\xqvar
+
F_{U{\nstep\estep\bs\nstep}V{\estep\bs\estep\nstep}W}\xqvar
+
F_{U{\nstep\bs\nstep\estep}V{\estep\bs\nstep\estep}W}\xqvar.
\end{eq*}
\end{enumerate}
\end{thm}

\begin{proof}
Suppose that $F:\DYCK{}\to\Lambda$ satisfies the conditions in the theorem.
First note that $F$ can be extended to a function $\bar{F}:\SCH{}\to\Lambda$ on Schr{\"o}der paths in a unique way such that the unicellular relation \refpi{equivalentRelations}{unicellular} holds.
Moreover $\bar{F}$ is multiplicative on Schr{\"o}der paths.

We next show that a function $\bar{F}$ satisfies the relations \refpi{equivalentRelations}{dyck}.
To see this let $P$ be a Schr{\"o}der path and $(x,z)\in\setZ^2$ be a point on $P$ with $x+1<z$, such that the bounce path of $P$ at $(x,z)$ has only one bounce point, and such that the bounce decomposition of $P$ at $(x,z)$ is of the form $P=U{\nstep\bs\nstep}V{\nstep\estep\bs\estep}W$ for some $U,V,W\in\{\nstep,\dstep,\estep\}^*$.
Using the unicellular relation (many times) we can write $\bar{F}_{P}$ as a linear combination of symmetric functions of $\bar{F}_Q=F_Q$, where $Q$ is a Dyck path of the form $Q=U'{\nstep\bs\nstep}V'{\nstep\estep\bs\estep}W'$.
In particular the bounce paths of $P$ and $Q$ at $(x,z)$ coincide.
We may now apply the linear relation from the assumption to the symmetric functions $F_Q$ and use the unicellular relation to reintroduce the diagonal steps.
In this way we have show that $\bar{F}$ satisfies \refq{dyckBounceI} and the same can be done for \refq{dyckBounceII}.

It now follows from \refp{equivalentRelations} and \reft{recursion} that the function $\bar{F}$, and in particular the function $F$ must be unique.
Moreover, by \reft{LLT} and \refp{equivalentRelations} the function that assigns to each Dyck path the corresponding unicellular LLT polynomial satisfies Conditions~\refi{dyckRecursion:mult} and \refi{dyckRecursion:bounce} above.
The fact that this function also satisfy the initial 
condition in~\refi{dyckRecursion:e} is an easy consequence of the $\elementaryE$-expansion in \refc{expansion}\footnote{The $\elementaryE$-expansion of the unicellular LLT polynomial of the path graph was also obtained in \cite[Prop.~5.18]{AlexanderssonPanova2016}}.
This completes the proof.
\end{proof}

\section{Bijections on colorings}
\label{Section:colorings}

In this section we prove \reft{LLT}.
We first verify 
\reft{recursion}~\refi{recursion:e}--\refi{recursion:unicellular}
 in \refp{coloringProperties}.
We then proceed to demonstrate the two bounce relations 
of \refti{recursion}{bounce} in Theorems~\ref{Theorem:colorings:bounceI} and~\ref{Theorem:colorings:bounceII}.
The method of proof is purely bijective and similar to that in \cite[Prop.~18]{Alexandersson2019llt}.
Indeed, \refi{recursion:unicellular} as well as the $st=\nstep\nstep$ case of \refi{recursion:bounce} were already proved therein.
To obtain the full claim of \reft{LLT} one can then use the results of \refs{bounce}, namely \refp{equivalentRelations} and \reft{bounce_reduction}.
We remark that it is also possible to give a bijective proof for the third type of bounce relation in \reft{LLT}, that is, the case $st=\nstep\dstep$, if there is only one bounce point.
However, \reft{bounce_reduction} is currently our only way to prove bounce relations with many bounce points.

In the proof of \reft{bounce_reduction},
we used diagrams to get a good overview of the recursions.
In this section we use such diagrams to denote 
weighted sums over vertex colorings of decorated unit-interval graphs. 
To simplify the notation, we shall only 
show the vertices and edges of the diagrams which matter.
The following example introduces the necessary conventions.

\begin{exa}{coloringSum}
The LLT polynomials $\LLT_P(\xvec;q)$ and 
$\LLT_Q(\xvec;q)$ for $P=\texttt{nndneee}$
and $Q=\texttt{nnddee}$ are given as sums over vertex colorings.
We let $A(\xvec;q)$ be the sum over colorings
of $\Gamma_P$, with the extra condition that $\kappa(2)\geq \kappa(4)$.
Expressed using diagrams, we have
\[
\LLT_P(\xvec;q) = 
\begin{ytableau}
*(lightgray) \; & \cdot & \; &  *(yellow) 4\\
*(lightgray) \to & \; &  *(yellow) 3\\
\; &  *(yellow) 2\\
*(yellow)  1
\end{ytableau}
\qquad 
\LLT_Q(\xvec;q) = 
\begin{ytableau}
*(lightgray) \; & *(lightgray) \to & \; &  *(yellow) 4\\
*(lightgray) \to & \; &  *(yellow) 3\\
\; &  *(yellow) 2\\
*(yellow)  1
\end{ytableau}
\qquad 
A(\xvec;q) = 
\begin{ytableau}
*(lightgray) \; & *(lightgray) \dnw & \; &  *(yellow) 4\\
*(lightgray) \to & \; &  *(yellow) 3\\
\; &  *(yellow) 2\\
*(yellow)  1
\end{ytableau}
.
\]
Note that we use $\mathdefin{\dnw}$ to impose a \emph{weak} inequality.
That is, if an edge $\edge{x}{y}$ is marked with $\dnw$,
then we require that $\kappa(x)\geq \kappa(y)$.
In the same manner, $\mathdefin{\dns}$ is used to 
indicate a strict inequality which is a non-ascent.
We utilize the convention that edges in gray boxes do not contribute 
to the ascent statistic (this is consistent with marking all strict edges gray). 
Moreover, edges of particular importance 
are marked with a center-dot; we shall possibly refine our argument 
depending on if the colorings under consideration have this edge as ascending or not.

Now, a simple bijective argument shows that
\[
 \LLT_P(\xvec;q) = q \times \LLT_Q(\xvec;q) + A(\xvec;q),
\]
since a coloring on the left-hand side appears as 
a coloring of exactly one of the diagrams in the righ hand side.
Note that it is only the edge from $2$ to $4$ that really matters.

Stated in full generality, for any edge $\dedge{w}{y}$ in any diagram,
we have the identity
\begin{eq}{colorCases}
\ytableausetup{boxsize=1.0em}
\begin{ytableau}
\cdot &  *(yellow) y\\
*(yellow)  w
\end{ytableau}
=
q \times 
\begin{ytableau}
*(lightgray) \to &  *(yellow) y\\
*(yellow)  w
\end{ytableau}
+
\begin{ytableau}
*(lightgray) \dnw &  *(yellow) y\\
*(yellow)  w
\end{ytableau}.
\end{eq}
This simply expresses the fact that a set of colorings $\kappa$
can be partitioned into two sets, the first set where $\kappa(w)>\kappa(y)$ and 
the second set where $\kappa(w) \leq \kappa(y)$.
\end{exa}

Note that most diagrams we use do not represent vertical-strip LLT polynomials,
but merely a sum $\sum_{\kappa} q^{\asc(\kappa)} x_{\kappa(1)} \dotsb x_{\kappa(n)}$,
where the sum ranges over colorings compatible with the restrictions 
imposed by the diagram.

\begin{prop}{coloringProperties}
Let $F:\SCH{}\to\Lambda$ be the function 
that assigns to each Schr{\"o}der path $P$ the symmetric 
function $F_P=\LLT_P(\xvec;q)$ defined in~\refq{lltDefinition}.
\begin{enumerate}[(i)]
\myi{coloring:init}
Then $F$ satisfies the initial condition $F_{\nstep\dstep^k\estep}=\elementaryE_{k+1}$ for all $k\in\setN$.

\myi{coloring:mult}
The function $F$ is multiplicative, that is, $F_{PQ} = F_P F_Q$ for all $P,Q\in\SCH{}$.

\myi{coloring:unicellular}
For all $U,V\in\{\nstep,\dstep,\estep\}^*$ such that $U\dstep V\in\SCH{}$ we have
$F_{U\nstep\estep V} - F_{U\estep\nstep V} = (q-1) F_{U\dstep V}$.
\end{enumerate}
\end{prop}
\begin{proof}
For \refi{coloring:init}, we have by definition that $F_{\nstep\dstep^k\estep}$ 
is a sum over colorings $\kappa:[k+1] \to \setNN$,
such that $\kappa(1)<\kappa(2)<\dotsb < \kappa(k+1)$.
Now it is clear that we indeed obtain $\elementaryE_{k+1}$.

Moreover, \refi{coloring:mult} is immediate from the definition,
as the graph $\Gamma_{PQ}$ is (up to a relabeling of the vertices) the disjoint union of the graphs
$\Gamma_P$ and $\Gamma_Q$. 
Colorings can thus be done on each component independently and the statement follows.

It remains to prove  \refi{coloring:unicellular}.
Using diagrams as in \refq{colorCases}, this identity can be expressed as follows:
\[
\ytableausetup{boxsize=1.0em}
\begin{ytableau}
\; &  *(yellow) y\\
*(yellow)  w
\end{ytableau}
-
\begin{ytableau}
*(lightgray) &  *(yellow) y \\
*(yellow) w
\end{ytableau}
=
(q-1)\;
\begin{ytableau}
*(lightgray) \to & *(yellow) y  \\
*(yellow) w
\end{ytableau}
\iff
\begin{ytableau}
\; &  *(yellow) y\\
*(yellow)  w
\end{ytableau}
+
\begin{ytableau}
*(lightgray) \to & *(yellow)  y \\
*(yellow)  w
\end{ytableau}
=
q\;
\begin{ytableau}
*(lightgray) \to & *(yellow) y  \\
*(yellow) w
\end{ytableau}
+
\begin{ytableau}
*(lightgray) &  *(yellow) y \\
*(yellow) w
\end{ytableau}
.
\]
By splitting the first and the last diagram into subcases 
depending on whether $\kappa(w)<\kappa(y)$ or $\kappa(w) \geq \kappa(y)$
as in \refq{colorCases},
\[
\ytableausetup{boxsize=0.85em}
\left(
\begin{ytableau}
*(lightgray) \dnw &  *(yellow) y \\
*(yellow) w
\end{ytableau}
+
q\;
\begin{ytableau}
*(lightgray) \to &  *(yellow) y \\
*(yellow) w
\end{ytableau}
\right)
+
\begin{ytableau}
*(lightgray) \to & *(yellow) y \\
*(yellow) w
\end{ytableau}
=
q\;
\begin{ytableau}
*(lightgray) \to & *(yellow) y \\
*(yellow) w
\end{ytableau}
+
\left(
\begin{ytableau}
*(lightgray) \dnw &  *(yellow) y \\
*(yellow) w
\end{ytableau}
+
\begin{ytableau}
*(lightgray) \to &  *(yellow) y \\
*(yellow) w
\end{ytableau}
\right)
\]
it becomes evident that the identity holds.
\end{proof}

In the remainder of this section we assume that $x+1 = y$
so that they are adjacent entries on the diagram diagonal.

\subsection{Bounce relations for colorings}\label{sec:swapMap}

We need a simple bijection between certain colorings
which is referred to as the \defin{swap map}.
The swap map sends a coloring $\kappa$ to the coloring $\kappa'$ defined as 
\begin{equation}\label{eq:swapMap}
\kappa'(j) \coloneqq 
\begin{cases}
 \kappa(j) & \text{ if } j \notin \{x,y\} \\
 \kappa(y) & \text{ if } j =x \\
 \kappa(x) & \text{ if } j =y.
\end{cases}
\end{equation}
This map is used in all further proofs in this section.

\begin{lem}{swapMap}[Swap Map]
Let $P \in \SCH{}$ and let $(x,z-1)\in\setZ^2$ be a point on $P$ such that the bounce path of $P$ at $(x,z-1)$ has one bounce point, and such that the bounce decomposition of $P$ at $(x,z-1)$ is given by $P=U\nstep\bs\nstep V\estep\bs\estep W$.
Set $y=x+1$. Then 
\[
 \sum_{\kappa : \kappa(x)<\kappa(y)} q^{\asc(\kappa)} x_{\kappa(1)}
 \dotsb x_{\kappa(n)} = q \times 
 \sum_{\kappa' : \kappa'(x)>\kappa'(y)} q^{\asc(\kappa')} x_{\kappa'(1)}
 \dotsb x_{\kappa'(n)},
\]
where we sum over all colorings of $P$ on both sides, 
with the indicated added restrictions on the colors of $x$ and $y$.
\end{lem}
\begin{proof}
Expressing the statement as an identity of diagrams, we have
\begin{equation}\label{eq:coloringsSwapMap}
\includegraphics[page=10,scale=0.8,valign=c]{figuresBounceRelations.pdf}%
=
q \times \;
\includegraphics[page=11,scale=0.8,valign=c]{figuresBounceRelations.pdf}%
.
\end{equation}
In order to prove \eqref{eq:coloringsSwapMap} we show that the swap map $\kappa\mapsto\kappa'$ preserves the number of ascents, not counting the contribution of the edge $\edge{x}{y}$.
Note that for all $j$ with $w<j<x$  or $y <j<z$ we have that 
\begin{eq*}
\kappa(x)<\kappa(j) &\iff \kappa'(y)<\kappa'(j),
\\
\kappa(y)<\kappa(j) &\iff \kappa'(x)<\kappa'(j).
\end{eq*}
This implies that for any coloring $\kappa$ appearing in 
the left-hand side of \eqref{eq:coloringsSwapMap}, 
we have that $\asc(\kappa) = 1+\asc(\kappa')$.
\end{proof}

We proceed by showing that the vertical-strip LLT polynomials satisfy the 
two bounce relations in \reft{recursion}.

\begin{thm}{colorings:bounceI}[The first bounce relation for colorings]
Let $P\in\SCH{}$ be a Schr{\"o}der path and let $(x,z)\in\setZ^2$ 
be a point on $P$ with $z>x+1$ such that the bounce path of $P$ at $(x,z)$ 
has a single bounce point, and such that the bounce 
decomposition of $P$ at $(x,z)$ is 
given by $P=U\nstep\bs\nstep V\dstep\bs\estep W$.
Then
\begin{eq*}
\LLT_{P}(\xvec;q) = q\LLT_{U\nstep\bs\nstep V\estep\bs\dstep W}(\xvec;q) .
\end{eq*}
\end{thm}
\begin{proof}
We want to prove the following identity expressed as diagrams:
\begin{eq}{eq:colorings:bounceI}
\includegraphics[page=12,scale=0.8,valign=c]{figuresBounceRelations.pdf}%
=
q \times \;
\includegraphics[page=13,scale=0.8,valign=c]{figuresBounceRelations.pdf}%
.
\end{eq}
If $\dedge{x}{z}$ and $\dedge{y}{z}$ are ascending edges on both sides, 
we use the identity map on the set of colorings. 
By using the transitivity of the forced inequalities, 
it then suffices to prove the identity
\begin{eq*}
\ytableausetup{boxsize=1.0em}
\begin{ytableau}
*(lightgray)      & *(lightgray)  \to & \dnw & *(yellow) z\\
 \cdot & \to  & *(yellow) y\\
 \cdot &  *(yellow) x\\
*(yellow)  w
\end{ytableau}
=
q 
\;
\times
\;
\begin{ytableau}
*(lightgray)      & *(lightgray)  \dnw & *(lightgray)  \to & *(yellow) z\\
 \cdot & \dns  & *(yellow) y\\
 \cdot &  *(yellow) x\\
*(yellow)  w
\end{ytableau}
.
\end{eq*}
But this follows from \refl{swapMap}.
\end{proof}

\begin{thm}{colorings:bounceII}[The second bounce relation for colorings]
Let $P\in\SCH{}$ be a Schr{\"o}der path and let $(x,z)\in\setZ^2$ 
be a point on $P$ with $z>x+1$ such that the bounce path of $P$ 
at $(x,z)$ has a single bounce point, and such that the 
bounce decomposition of $P$ at $(x,z)$ is given 
by $P=U\dstep\bs\nstep V\dstep\bs\estep W$.
Then
\begin{eq*}
\LLT_{P}(\xvec;q)
=
\LLT_{U\nstep\bs\dstep V\estep\bs\dstep W}(\xvec;q).
\end{eq*}
\end{thm}
\begin{proof}
We want to prove the identity expressed as diagrams in 
\refq{LLT_caseII_diagram}. For convenience, it is presented here as follows:
\[
\ytableausetup{boxsize=1.0em}
\begin{ytableau}
*(lightgray)      & *(lightgray)  \to & \cdot & *(yellow) z\\
 *(lightgray) \cdot & \cdot  & *(yellow)  y\\
 *(lightgray) \to &  *(yellow)  x\\
*(yellow)   w
\end{ytableau}
\quad
=
\quad
\begin{ytableau}
*(lightgray)    & *(lightgray)  \cdot & *(lightgray)  \to & *(yellow)  z\\
 *(lightgray) \to & \cdot  & *(yellow)  y\\
 \cdot &  *(yellow)  x\\
*(yellow)   w
\end{ytableau}.
\]
We split both sides into cases according to \refq{colorCases}:
\[
\ytableausetup{boxsize=0.8em}
\begin{ytableau}
*(lightgray)      & *(lightgray)  \to & \to & *(yellow)  z\\
 *(lightgray) \to & \cdot  & *(yellow)  y\\
 *(lightgray) \to &  *(yellow)  x\\
*(yellow)   w
\end{ytableau}
+
\begin{ytableau}
*(lightgray)      & *(lightgray)  \to & \dnw & *(yellow)  z\\
 *(lightgray) \to & \cdot  & *(yellow)  y\\
 *(lightgray) \to &  *(yellow)  x\\
*(yellow)   w
\end{ytableau}
+
\begin{ytableau}
*(lightgray)      & *(lightgray)  \to & \to & *(yellow)  z\\
 *(lightgray) \dnw & \cdot  & *(yellow)  y\\
 *(lightgray) \to &  *(yellow)  x\\
*(yellow)   w
\end{ytableau}
+
\begin{ytableau}
*(lightgray)      & *(lightgray)  \to & \dnw & *(yellow)  z\\
 *(lightgray) \dnw & \cdot  & *(yellow)  y\\
 *(lightgray) \to &  *(yellow)  x\\
*(yellow)   w
\end{ytableau}
=
\begin{ytableau}
*(lightgray)    & *(lightgray)  \to & *(lightgray)  \to & *(yellow)  z\\
 *(lightgray) \to & \cdot  & *(yellow)  y\\
 \to &  *(yellow)  x\\
*(yellow)   w
\end{ytableau}
+
\begin{ytableau}
*(lightgray)    & *(lightgray)  \dnw & *(lightgray)  \to & *(yellow)  z\\
 *(lightgray) \to & \cdot  & *(yellow)  y\\
 \to &  *(yellow)  x\\
*(yellow)   w
\end{ytableau}
+
\begin{ytableau}
*(lightgray)    & *(lightgray)  \to & *(lightgray)  \to & *(yellow)  z\\
 *(lightgray) \to & \cdot  & *(yellow)  y\\
 \dnw &  *(yellow)  x\\
*(yellow)   w
\end{ytableau}
+
\begin{ytableau}
*(lightgray)    & *(lightgray)  \dnw & *(lightgray)  \to & *(yellow)  z\\
 *(lightgray) \to & \cdot  & *(yellow)  y\\
 \dnw &  *(yellow)  x\\
*(yellow)   w
\end{ytableau}
.
\]
The first term on both sides cancel, and there are no 
colorings compatible with the inequalities in the last term on both sides.
Moreover, transitivity forces the last inequality in the remaining cases:
\[
\ytableausetup{boxsize=0.9em}
\begin{ytableau}
*(lightgray)      & *(lightgray)  \to & \dnw & *(yellow)  z\\
 *(lightgray) \to & \to  & *(yellow)  y\\
 *(lightgray) \to &  *(yellow)  x\\
*(yellow)   w
\end{ytableau}
+
\begin{ytableau}
*(lightgray)      & *(lightgray)  \to & \to & *(yellow)  z\\
 *(lightgray) \dnw & \dn  & *(yellow)  y\\
 *(lightgray) \to &  *(yellow)  x\\
*(yellow)   w
\end{ytableau}
=
\begin{ytableau}
*(lightgray)    & *(lightgray)  \dnw & *(lightgray)  \to & *(yellow)  z\\
 *(lightgray) \to & \dn  & *(yellow)  y\\
 \to &  *(yellow)  x\\
*(yellow)   w
\end{ytableau}
+
\begin{ytableau}
*(lightgray)    & *(lightgray)  \to & *(lightgray)  \to & *(yellow)  z\\
 *(lightgray) \to & \to  & *(yellow)  y\\
 \dnw &  *(yellow)  x\\
*(yellow)   w
\end{ytableau}
.
\]
All four remaining diagrams have the same $q$-weight, namely one fixed ascent (recall the convention for shaded boxes).
The swap map sends the terms on the left-hand side
to the terms on the right-hand side. This concludes the proof.
\end{proof}

\section{Bijections on orientations}
\label{Section:orientations}

In this section we prove \reft{LLTe}.
To be precise we first treat the easier claims in \refp{O:e-mult-diag}.
We then show that the symmetric functions $\LLTe_P(\xvec;q)$ satisfy the bounce relations of \refti{recursion}{bounce} in Theorems~\ref{Theorem:orientations:bounceI} and~\ref{Theorem:orientations:bounceIII}.
Since we want to work with orientations we actually show that the symmetric functions $\LLTe_P(\xvec;q+1)$ satisfy a shifted version of these relations.
Our proofs are purely bijective and the involved bijections are very simple.
However, the proofs of the bounce relations are a bit tedious in the sense that they require the distinction of quite a few different cases.
We remark that the third type of bounce relations in \reft{LLT}, ending in a pattern $st=\nstep\dstep$, can be given a bijective proof in similar style, assuming that there is only one bounce point.

As with the colorings, we use diagrams to illustrate the different cases.
However, in this section each diagram represents a \emph{sum over orientations} $\theta$
weighted by $q^{\asc(\theta)}\elementaryE_{\hrvpp(\theta)}$.

\begin{exa}{orientation_diagrams}
The symmetric function $\LLTe_Q(\xvec;q+1)$ where $Q=\texttt{nnddee}$ is described by the diagram
\begin{eq*}
 \LLTe_Q(\xvec;q+1) = 
\ytableausetup{boxsize=1.0em}
\begin{ytableau}
*(lightgray) \; & *(lightgray) \to &  &  *(yellow) 4\\
*(lightgray) \to & \; &  *(yellow) 3\\
\; &  *(yellow) 2\\
*(yellow)  1
\end{ytableau}
\end{eq*}
which represents a sum over eight orientations (all possible ways to orient the edges $\edge{1}{2}$, $\edge{2}{3}$ and $\edge{3}{4}$).
By specifying additional edges in the diagram,
we obtain a sum over a smaller set of orientations. For example,
the two following diagrams represent the indicated sums.
\begin{eq*}
\begin{ytableau}
*(lightgray)   & *(lightgray) \to & \dn &  *(yellow)  4\\
*(lightgray) \to &   &  *(yellow)  3\\
\dn &  *(yellow)  2\\
*(yellow)   1
\end{ytableau}
= (1+q) \elementaryE_{22}
\qquad
\text{and}
\qquad 
 \begin{ytableau}
*(lightgray)   & *(lightgray) \to & \dn &  *(yellow)  4\\
*(lightgray) \to &   &  *(yellow)  3\\
*(lightgray) \to &  *(yellow)  2\\
*(yellow)   1
\end{ytableau}
= (1+q)\elementaryE_{31}
\end{eq*}
Notice that in the second diagram, we again use the convention
that edges marked gray do not contribute to ascents.
\end{exa}

We start out by proving that the function $P\mapsto\LLTe_P(\xvec;q+1)$ satisfies the correct initial conditions, is multiplicative, 
and obeys the shifted version of the unicellular relation.
\begin{prop}{O:e-mult-diag}
Let $F:\SCH{}\to\Lambda$ be the function that assigns to 
each Schr{\"o}der path $P$ the symmetric function $F_P=\LLTe_P(\xvec;q+1)$ defined in~\refq{conjFormula}.
\begin{enumerate}[(i)]
\myi{O:e-mult-diag:e}
Then $F$ satisfies the initial condition $F_{\nstep\dstep^k\estep}=\elementaryE_{k+1}$ for all $k\in\setN$.

\myi{O:e-mult-diag:mult}
The function $F$ is multiplicative, that is, $F_{PQ}=F_P F_Q$ for all $P,Q\in\SCH{}$.

\myi{O:e-mult-diag:diag}
For all $U,V\in\{\nstep,\dstep,\estep\}^*$ such that $U\dstep V\in\SCH{}$ we have
$F_{U\nstep\estep V} - F_{U\estep\nstep V} = q F_{U\dstep V}$.
\end{enumerate}
\end{prop}

\begin{proof}
To see \refi{O:e-mult-diag:e} let $P=\nstep\dstep^k\estep\in\SCH{k+1}$.
Note that the decorated unit-interval graph $\Gamma_{P}$ has $k$ strict edges and no other edges.
Thus $\ORIENTS(P)$ contains a unique orientation $\theta$ which has no ascents.
Its highest reachable vertex partition is given by $\hrvpp(\theta)=(k+1)$ since the vertex $k+1$ can be reached from every other vertex using the strict edges.
Hence $F_{P}=\elementaryE_{k+1}$ as claimed.

To see \refi{O:e-mult-diag:mult} let $P,Q\in\SCH{}$.
Note that $\uig_{PQ}$ is (up to a relabeling of the vertices) 
the disjoint union of the graphs $\uig_P$ and $\uig_Q$.
Thus every orientation of $\uig_{PQ}$ is (again up to a relabeling) 
the disjoint union of an orientation of $\uig_P$ and $\uig_Q$.
It follows that
\begin{eq*}
F_{PQ}
=
\sum_{\theta\in\ORIENTS(P)}
\sum_{\theta'\in\ORIENTS(Q)}
q^{\asc(\theta)+\asc(\theta')}\elementaryE_{\hrvpp(\theta)}\elementaryE_{\hrvpp(\theta')}
=
F_PF_Q
.
\end{eq*}
To see \refi{O:e-mult-diag:diag} let $U,V\in\{\nstep,\dstep,\estep\}^*$ such that $U\dstep V\in\SCH{}$.
Let $\edge{w}{y}$, where $w<y$, denote the strict edge in $\uig_{U\dstep V}$ 
corresponding to the diagonal step between $U$ and $V$.
Given $\theta\in\ORIENTS(U\nstep\estep V)$ let $\phi(\theta)=\theta\in\ORIENTS(U\dstep V)$ 
if $\dedge{w}{y}\in\theta$, 
and $\phi(\theta)=\theta\setminus\{\dedge{y}{w}\}\in\ORIENTS(U\estep\nstep V)$ if $\dedge{y}{w}\in\theta$.
This yields a bijection
\begin{alignat*}{3}
\phi:\ORIENTS(U\nstep\estep V)
&\quad \to\quad 
\ORIENTS(U\dstep V)&&\quad \sqcup\quad \ORIENTS(U\estep\nstep V) \\
 \ytableausetup{boxsize=1.0em}
\begin{ytableau}
\dn &  *(yellow) y\\
*(yellow)  w
\end{ytableau}
&\quad \mapsto\quad 
&&\quad \phantom{+}\quad 
\begin{ytableau}
*(lightgray) &  *(yellow) y \\
*(yellow) w
\end{ytableau} \\
\begin{ytableau}
\to &  *(yellow) y\\
*(yellow)  w
\end{ytableau}
&\quad \mapsto\quad 
q\;
\begin{ytableau}
*(lightgray) \rightarrow & *(yellow) y  \\
*(yellow) w
\end{ytableau}
\end{alignat*}
that preserves highest reachable vertices.
Moreover $\asc(\phi(\theta))=\asc(\theta)-1$ if $\dedge{w}{y}\in\theta$, 
and $\asc(\phi(\theta))=\asc(\theta)$ if $\dedge{y}{w}\in\theta$. The claim follows.
\end{proof}

We now turn our attention to the bounce relations.
To begin with we need some additional definitions.
Given a finite totally ordered set $X=\{u_1<u_2<\dots<u_r\}$ and a function $\sigma:X\to\setNN$, we identify $\sigma$ with the word $\sigma_1\sigma_2\dots\sigma_r$ where $\sigma_i=\sigma(u_i)$ for all $i\in[r]$.
Note that this word contains all the information on $\sigma$ if the domain $X$ is known.
The \defin{weak standardisation} of $\sigma$ is defined as the unique surjection $\wst(\sigma):X\to[m]$ that satisfies
\begin{eq*}
\wst(\sigma)(u)\leq\wst(\sigma)(v)
\quad\Leftrightarrow\quad
\sigma(u)\leq \sigma(v)
\end{eq*}
for all $u,v\in X$, where $m$ is (necessarily) the cardinality of the image of $\sigma$.

Let $\mathdefin{\ORIENTS(n)}$ denote the set of orientations of unit-interval graphs on $n$ vertices.
Given a Schr{\"o}der path $P\in\SCH{n}$ and a set of directed edges $D\subseteq [n]^2$, define
\begin{eq*}
\mathdefin{\ORIENTS(P,D)}
\subset \ORIENTS(P)
\end{eq*}
as the set of all orientations $\theta\in\ORIENTS(P)$ such that $D\subseteq\theta$.
Moreover, given $I\subseteq[n]$, $m\in[n]$, and a surjection $\sigma:I\to[m]$, define
\begin{eq*}
\mathdefin{\ORIENTS(P,\sigma,D)}
\subset \ORIENTS(P,D)
\end{eq*}
as the set of all orientations $\theta\in\ORIENTS(P,D)$ such 
that $\sigma$ is the weak standardization of $\hrv{\theta}{.\,}$ on $I$.
That is, for all $i,j\in I$ we have
\begin{eq*}
\sigma_i \leq \sigma_j \iff \hrv{\theta}{i} \leq \hrv{\theta}{j}.
\end{eq*}
If $D=\{\dedge{v}{w},\dedge{x}{y},\dots\}$ has small cardinality we write
\begin{eq*}
\mathdefin{\ORIENTS(P;\dedge{v}{w},\dedge{x}{y},\dots)}
\quad\text{and}\quad
\mathdefin{\ORIENTS(P,\sigma;\dedge{v}{w},\dedge{x}{y},\dots)}
\end{eq*}
instead of $\ORIENTS(P,D)$ and $\ORIENTS(P,\sigma,D)$ in order to make notation less heavy.

\begin{exa}{orient_notation}
Let $P=\texttt{nndneee}$. Then the set $\ORIENTS(P,2212;\dedge{1}{2})$
consists of the two orientations
\[
 \begin{ytableau}
*(lightgray) \; & \to & \dn &  *(yellow) 4\\
*(lightgray) \to & \dn &  *(yellow) 3\\
\to &  *(yellow) 2\\
*(yellow)  1
\end{ytableau}
\quad
\text{ and }
\qquad
 \begin{ytableau}
*(lightgray) \; & \to & \dn &  *(yellow) 4\\
*(lightgray) \to & \to &  *(yellow) 3\\
\to &  *(yellow) 2\\
*(yellow)  1
\end{ytableau}.
\]
\end{exa}

Throughout the remainder of this section let \defin{$P$} be a Schr{\"o}der path 
and let $(\mathdefin{x},\mathdefin{z})\in\setZ^2$ be a point on $P$ with $z>x+1$ such 
that the bounce partition of $P$ at $(x,z)$ is given by $(z,x,\mathdefin{v})$, 
and the bounce decomposition of $P$ at $(x,z)$ is given by $P=Us\bs t V\dstep\bs \estep W$ 
for some $\mathdefin{st}\in\{\nstep\nstep,\dstep\nstep\}$ and $\mathdefin{U},\mathdefin{V},\mathdefin{W}\in\{\nstep,\dstep,\estep\}$.
In particular, the bounce path of $P$ at $(x,z)$ has only one bounce point $(x,x)$.
Set $\mathdefin{w}\coloneqq v+1$ and $\mathdefin{y}\coloneqq x+1$, and 
\begin{eq*}
\mathdefin{Q} \coloneqq
\begin{cases}
U\nstep\nstep V\estep\dstep W
&\quad\text{if }st=\nstep\nstep,
\\
U\nstep\dstep V\estep\dstep W
&\quad\text{if }st=\dstep\nstep
.
\end{cases}
\end{eq*}
Let $\mathdefin{\tau} \coloneqq (x,y)\in\symS_n$ be a transposition.
Given $A\subseteq\ORIENTS(n)$ and a function $f:A\to\ORIENTS(n)$, we say that $f$ is \defin{$\hrvp$-preserving on $A$} if
$\hrvp(f(\theta))=\hrvp(\theta)$
for all $\theta\in A$.
The function $f$ is called \defin{$\hrvp$-switching
on $A$} if $\hrvp(f(\theta))=\tau(\hrvp(\theta))$.
Both of these properties imply that $f$ is \defin{$\hrvpp$-preserving on $A$}, 
that is $\hrvpp(f(\theta))=\hrvpp(\theta)$ for all $\theta\in A$.
To see this note that permuting the entries of a set partition does not change the block structure.

Define a function
\begin{eq*}
\mathdefin{\rfun{v}}:\ORIENTS(n)\to[n]
\end{eq*}
by letting $\rfun{v}(\theta)$ be the maximal vertex in $[n]$ that can be reached from $v$ using only ascending and strict edges in $\theta$ without using any of the edges $\dedge{v}{x}$ or $\dedge{v}{y}$.
Similarly define
\begin{eq*}
\mathdefin{\rfun{x}}:\ORIENTS(n)\to[n]
\end{eq*}
by letting $\rfun{x}(\theta)$ be the maximal vertex in $[n]$ that can be reached from $x$ using only ascending and strict edges in $\theta$ without using any of the edges $\dedge{x}{y}$ or $\dedge{x}{z}$.
Finally define
\begin{eq*}
\mathdefin{\rfun{y}}:\ORIENTS(n)\to[n]
\end{eq*}
by letting $\rfun{y}(\theta)$ be the maximal vertex in $[n]$ that can be reached from $y$ using only ascending and strict edges in $\theta$ without using the edge $\dedge{y}{z}$.
Note that $\rfun{w}(\theta)\leq\hrv{\theta}{w}$, $\rfun{x}(\theta)\leq\hrv{\theta}{x}$ and $\rfun{y}(\theta)\leq\hrv{\theta}{y}$.

\subsection{The first bounce relation for orientations}

We now take care of the case $st=\nstep\nstep$.
Our goal is to prove the following theorem.

\begin{thm}{orientations:bounceI}
Let $P\in\SCH{}$ be a Schr{\"o}der path and let $(x,z)\in\setZ^2$ be a point on $P$ with $z>x+1$ such that the bounce path of $P$ at $(x,z)$ has a single bounce point, and such that the bounce decomposition of $P$ at $(x,z)$ is given by $P=U\nstep\bs\nstep V\dstep\bs\estep W$.
Set $Q\coloneqq U\nstep\bs\nstep V\estep\bs\dstep W$.
Then
\begin{eq*}
\LLTe_{P}(\xvec;q+1)
=
(q+1)\LLTe_{Q}(\xvec;q+1)
.
\end{eq*}
\end{thm}

A proof of \reft{orientations:bounceI} is given at the end of this subsection.
Our strategy is to find a $\hrvpp$-preserving two-to-one 
function $\Phi:\ORIENTS(P)\to\ORIENTS(Q)$ such that the fibres of $\Phi$ satisfy
\begin{eq*}
\sum_{\theta\in\Phi^{-1}(\theta')}q^{\asc(\theta)}
=
(q+1)q^{\asc(\theta')}
\end{eq*}
for all $\theta'\in\ORIENTS(Q)$.
We build such a function
$\Phi$ from three building blocks $\phi$, $\phibar$, and $\psi$
defined below.

Let $\phi:\ORIENTS(P)\to\ORIENTS(Q)$ be the function defined by
\begin{eq*}
\mathdefin{\phi(\theta)}
\coloneqq
\theta\setminus\{\dedge{x}{z},\dedge{z}{y}\}\cup\{\dedge{y}{z}\}.
\end{eq*}

\begin{exa}{II:phiExample}
The map $\phi$ is a bijection for all choices $\alpha,\beta \in \{\to,\dn\}$.
\[
 \phi:
 \includegraphics[page=1,scale=0.8,valign=c]{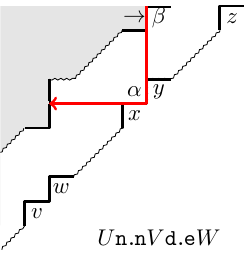}%
\quad\mapsto\quad%
\includegraphics[page=2,scale=0.8,valign=c]{figuresOrientations.pdf}%
\]
\end{exa}

The following result is immediate from this definition.
\begin{lem}{II:phi}
Let $a\in\{\dedge{x}{y},\dedge{y}{x}\}$, $b\in\{\dedge{y}{z},\dedge{z}{y}\}$.
Then
\begin{eq*}
\phi:
\ORIENTS(P;a,b)
\to
\ORIENTS(Q;a)
\end{eq*}
is a bijection.
Moreover $\asc(\phi(\theta))=\asc(\theta)-1$ for all $\theta\in\ORIENTS(P;\dedge{y}{z})$, and $\asc(\phi(\theta))=\asc(\theta)$ for all $\theta\in\ORIENTS(P;\dedge{z}{y})$.
\hfill\qedsymbol
\end{lem}

If we impose restrictions via $\sigma$ we can say more.
\begin{lem}{II:phi_sigma}
Let $\sigma:\{x,y,z\}\to[m]$ be a surjection such 
that $\sigma_3<\sigma_1$ and $\sigma_3<\sigma_2$, 
let $a\in\{\dedge{x}{y},\dedge{y}{x}\}$, and 
let $b\in\{\dedge{y}{z},\dedge{z}{y}\}$.
Then
\begin{eq*}
\phi:
\ORIENTS(P,\sigma;a,b)
\to
\ORIENTS(Q,\sigma;a)
\end{eq*}
is a $\hrvp$-preserving bijection.
\end{lem}

\begin{proof}
Set $A=\ORIENTS(P,\sigma;a,b)$ and $B=\ORIENTS(Q,\sigma;a)$, and let $\theta\in A$.
Note that $\hrv{\theta}{y}>\hrv{\theta}{z}$ implies
\begin{eq*}
\hrv{\theta}{y}
=
\rfun{y}(\theta)
=
\rfun{y}(\theta')
=
\hrv{\theta'}{y}.
\end{eq*}
Similarly $\hrv{\theta}{x}>\hrv{\theta}{z}$ implies that the edge $\dedge{x}{z}$ is not needed to reach $\hrv{\theta}{x}$ from $x$ using only strict and ascending edges in $\theta$.
Hence $\hrv{\theta'}{x}=\hrv{\theta}{x}$.
Consequently $\phi:A\to B$ is well-defined and $\hrvp$-preserving on $A$.

To see that $\phi$ is also surjective let $\theta'\in B$.
By \refl{II:phi} there exists $\theta\in\ORIENTS(P;a,b)$ with $\phi(\theta)=\theta'$.
Since $\hrv{\theta'}{y}>\hrv{\theta'}{z}$ we have
\begin{eq*}
\hrv{\theta'}{y}
=\rfun{y}(\theta')
=\rfun{y}(\theta)
=\hrv{\theta}{y}.
\end{eq*}
Similarly $\hrv{\theta'}{x}>\hrv{\theta'}{z}$ implies that $\hrv{\theta}{x}=\hrv{\theta'}{x}$.
Thus $\theta\in A$ and the proof is complete.
\end{proof}

We now collect the orientations on which we plan to use the map $\phi$ in the proof of \reft{orientations:bounceI}.

\begin{lem}{II:id}
The function $\phi$ is a $\hrvpp$-preserving bijection between the sets
\begin{eq}{II:id}
\ORIENTS(P,221;a,b)
&\to
\ORIENTS(Q,221;a)
\\
\ORIENTS(P,231;\dedge{y}{x},b)
&\to
\ORIENTS(Q,231;\dedge{y}{x})
\\
\ORIENTS(P,321;a,b)
&\to
\ORIENTS(Q,321;a)
\end{eq}
for all $a\in\{\dedge{x}{y},\dedge{y}{x}\}$ and $b\in\{\dedge{y}{z},\dedge{z}{y}\}$, 
where all surjections have domain $\{x,y,z\}$.

Let $A$ be the union of all domains in \refq{II:id} taken over all possible choices of $a$ and $b$.
Similarly let $B$ be the union of all codomains in \refq{II:id} taken over all possible choices of $a$.
Then $\phi:A\to B$ is a two-to-one function, and the fibres of $\phi$ satisfy
\begin{eq*}
\sum_{\theta\in\phi^{-1}(\theta')}q^{\asc(\theta)}
=
(q+1)q^{\asc(\theta')}
.
\end{eq*}
for all $\theta'\in B$.
\end{lem}

\begin{proof}
All cases satisfy the conditions $\sigma_1>\sigma_3$ and $\sigma_2>\sigma_3$.
Hence \refl{II:phi_sigma} applies.
Each codomain is in bijection with two distinct domains.
Depending on the choice of $b$ the $q$-weight of $\theta$ and $\phi(\theta)$ is either equal or differs by one.
\end{proof}


Next let
$
\phibar:\ORIENTS(P)\to\ORIENTS(Q)
$
be the function defined by
\begin{eq*}
\mathdefin{\phibar(\theta)}
\coloneqq 
\phi(\theta)\setminus\{\dedge{y}{x}\}\cup\{\dedge{x}{y}\}
.
\end{eq*}

\begin{exa}{II:phiBarExample}
The map $\phi'$ is a bijection for all choices $\alpha,\beta \in \{\to,\dn\}$.
\[
\phibar:
 \includegraphics[page=3,scale=0.8,valign=c]{figuresOrientations.pdf}%
\quad\mapsto\quad%
\includegraphics[page=4,scale=0.8,valign=c]{figuresOrientations.pdf}%
\]
\end{exa}

Our analysis of $\phibar$ follows the same pattern as for the map $\phi$.

\begin{lem}{II:phi'}
Let $a\in\{\dedge{x}{y},\dedge{y}{x}\}$ and $b\in\{\dedge{y}{z},\dedge{z}{y}\}$.
Then
\begin{eq*}
\phibar:
\ORIENTS(P;a,b)
\to
\ORIENTS(Q;\dedge{x}{y})
\end{eq*}
is a bijection.
Moreover $\asc(\phi(\theta))=\asc(\theta)-1$ for all $\theta\in\ORIENTS(P;\dedge{x}{y},\dedge{y}{z})$, and $\asc(\phi(\theta))=\asc(\theta)$ for all $\theta\in\ORIENTS(P;\dedge{x}{y},\dedge{z}{y})$ and all $\theta\in\ORIENTS(P;\dedge{y}{x},\dedge{y}{z})$.
\hfill\qedsymbol
\end{lem}

If we impose restrictions on $\sigma$ we can say more.
\begin{lem}{II:phi'_sigma}
Let $\sigma:\{x,y,z\}\to[m]$ be a surjection 
with $\sigma_2\leq\sigma_1$ and $\sigma_2=\sigma_3$, and let $a\in\{\dedge{x}{y},\dedge{y}{x}\}$.
Then
\begin{eq*}
\phibar:
\ORIENTS(P,\sigma;a,\dedge{y}{z})
\to
\ORIENTS(Q,\sigma;\dedge{x}{y})
\end{eq*}
is a $\hrvp$-preserving bijection on $\ORIENTS(P,\sigma;a,\dedge{y}{z})$.
\end{lem}

\begin{proof}
Set $A=\ORIENTS(P,\sigma;a,\dedge{y}{z})$ and $B=\ORIENTS(Q,\sigma;\dedge{x}{y})$. 
Let $\theta\in A$ and set $\theta'=\phibar(\theta)$.
Note that $\hrv{\theta}{y}=\hrv{\theta}{z}$ implies
\begin{eq*}
\rfun{y}(\theta')
=
\rfun{y}(\theta)
\leq
\hrv{\theta}{z}
=
\hrv{\theta'}{z}
.
\end{eq*}
Since $\dedge{y}{z}\in\theta'$ is a strict edge it follows that
\begin{eq*}
\hrv{\theta'}{y}
=
\hrv{\theta'}{z}
=
\hrv{\theta}{y}.
\end{eq*}
On the other hand $\dedge{x}{y}\in\theta'$ implies that $\hrv{\theta'}{x}\geq\hrv{\theta'}{y}$.
If $\hrv{\theta}{x}>\hrv{\theta}{z}$ then
\begin{eq*}
\hrv{\theta}{x}
=
\rfun{x}(\theta)
=
\rfun{x}(\theta')
=
\hrv{\theta'}{x}.
\end{eq*}
Otherwise
$\rfun{x}(\theta')
\leq\hrv{\theta'}{y}$ and
\begin{eq*}
\hrv{\theta'}{x}
=
\hrv{\theta'}{y}
=
\hrv{\theta}{y}
=
\hrv{\theta}{x}
.
\end{eq*}
Consequently $\phibar:A\to B$ is well-defined and $\hrvp$-preserving on $A$.

To see that $\phibar$ is also surjective let $\theta'\in B$.
By \refl{II:phi'} there exists a unique $\theta\in\ORIENTS(P;a,\dedge{y}{z})$ with $\phibar(\theta)=\theta'$.
From $\hrv{\theta'}{y}=\hrv{\theta'}{z}$ it follows that
\begin{eq*}
\rfun{y}(\theta)
=
\rfun{y}(\theta')
\leq
\hrv{\theta'}{z}
=
\hrv{\theta}{z}
.
\end{eq*}
Since $\dedge{y}{z}\in\theta$ by assumption we obtain $\hrv{\theta}{y}=\hrv{\theta}{z}$.
If $\hrv{\theta'}{x}>\hrv{\theta'}{z}$ then
\begin{eq*}
\hrv{\theta'}{x}
=
\rfun{x}(\theta')
=
\rfun{x}(\theta)
=
\hrv{\theta}{x}
.
\end{eq*}
Otherwise $\rfun{x}(\theta)\leq\hrv{\theta}{z}$ and
\begin{eq*}
\hrv{\theta}{x}
=
\hrv{\theta}{z}
=
\hrv{\theta'}{x}
\end{eq*}
since $\dedge{x}{z}$ is a strict edge in $\theta$.
We conclude that $\theta\in A$ and the proof is complete.
\end{proof}

We now collect the orientations on which we plan to use the map $\phibar$ in the proof of \reft{orientations:bounceI}.

\begin{lem}{II:changeb}
The function $\phibar$ is a $\hrvpp$-preserving bijection between the sets
\begin{eq}{II:changeb}
\ORIENTS(P,111;a,\dedge{y}{z})
&\to
\ORIENTS(Q,111;\dedge{x}{y})
\\
\ORIENTS(P,211;a,\dedge{y}{z})
&\to
\ORIENTS(Q,211;\dedge{x}{y})
\end{eq}
for all $a\in\{\dedge{x}{y},\dedge{y}{x}\}$, where all surjections $\sigma$ have domain $\{x,y,z\}$.

Let $A$ be the union of all domains in \refq{II:changeb} taken over all possible choices of $a$.
Similarly let $B$ be the union of all codomains in \refq{II:changeb}.
Then $\phibar:A\to B$ is a two-to-one function, and the fibres of $\phibar$ satisfy
\begin{eq*}
\sum_{\theta\in(\phibar)^{-1}(\theta')}q^{\asc(\theta)}
=
(q+1)q^{\asc(\theta')}
\end{eq*}
for all $\theta'\in B$.
\end{lem}

\begin{proof}
In each case $\sigma$ satisfies $\sigma_1\geq\sigma_2=\sigma_3$.
Hence \refl{II:phi'_sigma} applies.
Each codomain is in bijection with two distinct domains depending on a choice of $a$.
The claim on the $q$-weights follows from \refl{II:phi'}.
\end{proof}


Next let $\mathdefin{\psi}:\ORIENTS(P)\to\ORIENTS(Q)$ be the function defined by 
\begin{eq*}
\dedge{x}{i}\in\psi(\theta)
&\iff
\dedge{y}{i}\in\phi(\theta)
\\
\dedge{y}{i}\in\psi(\theta)
&\iff
\dedge{x}{i}\in\phi(\theta)
\end{eq*}
for all $i\in \{w,\dotsc,x-1\} \cup\{y+1,\dotsc,z-1\}$, and
\begin{eq*}
\dedge{i}{j}\in\psi(\theta)
&\iff
\dedge{i}{j}\in\phi(\theta)\setminus\{\dedge{x}{y}\}\cup\{\dedge{y}{x}\}
\end{eq*}
for all other edges.
In words, $\psi(\theta)$ is obtained from $\phi(\theta)$ by exchanging the rows $x$ and $y$, and the columns $x$ and $y$ (excluding the edge $\edge{x}{y}$ and the strict edge $\edge{y}{z}$), and by forcing the inclusion of $\dedge{y}{x}$.

\begin{exa}{II:psiExample}
As seen in the figure here, the map $\psi$
swaps the marked columns above $x$ and $y$, and the marked rows 
to the left of $x$ and $y$. 
It is a bijection for all choices $\alpha,\beta \in \{\to,\dn\}$, and it 
may decrease the number of ascents depending on orientations of $\alpha$ and $\beta$.
\[
\psi:
 \includegraphics[page=5,scale=0.8,valign=c]{figuresOrientations.pdf}%
\quad\mapsto\quad%
\includegraphics[page=6,scale=0.8,valign=c]{figuresOrientations.pdf}%
\]
\end{exa}

Our analysis of the map $\psi$ follows the same pattern as for the functions $\phi$ and $\phibar$ above.

\begin{lem}{II:psi}
Let $a\in\{\dedge{x}{y},\dedge{y}{x}\}$ and $b\in\{\dedge{y}{z},\dedge{z}{y}\}$.
Then
\begin{eq*}
\psi:
\ORIENTS(P;a,b)
\to
\ORIENTS(Q;\dedge{y}{x})
\end{eq*}
is a bijection.
Moreover $\asc(\psi(\theta))=\asc(\theta)-1$ for all $\theta\in\ORIENTS(P;\dedge{y}{x},\dedge{y}{z})$ 
and all $\theta\in\ORIENTS(P;\dedge{x}{y},\dedge{z}{y})$, and $\asc(\psi(\theta))=\asc(\theta)$ 
for all $\theta\in\ORIENTS(P;\dedge{y}{x},\dedge{z}{y})$.
\hfill\qedsymbol
\end{lem}

If we impose restrictions on the data we can say more.
\begin{lem}{II:psi_sigma}
Let $\sigma:\{x,y,z\}\to[m]$ be a surjection, 
let $a\in\{\dedge{x}{y},\dedge{y}{x}\}$, and let $b\in\{\dedge{y}{z},\dedge{z}{y}\}$, such that
the following two conditions are satisfied:
\begin{enumerate}[(i)]
\myi{II:psi_sigma:a}
$a=\dedge{y}{x}$ or $\sigma_1>\sigma_2$ or $\sigma_2=\sigma_3$.

\myi{II:psi_sigma:b}
$b=\dedge{z}{y}$ or $\sigma_2>\sigma_3$.
\end{enumerate}
Then
\begin{eq*}
\psi:
\ORIENTS(P,\sigma;a,b)
\to
\ORIENTS(Q,\sigma_2\sigma_1\sigma_3;\dedge{y}{x})
\end{eq*}
is a $\pi$-switching bijection.
\end{lem}

\begin{proof}
Let $A=\ORIENTS(P,\sigma;a,b)$ and $B=\ORIENTS(Q,\sigma_2\sigma_1\sigma_3;\dedge{y}{x})$.
We first show that $\psi:A\to\ORIENTS(Q)$ is $\pi$-switching.
This implies that $\psi:A\to B$ is well-defined, that is, that $\psi$ really maps $A$ to $B$.
The injectivity of $\psi$ follows from \refl{II:psi}.
To conclude the proof we then show that $\psi:A\to B$ is also surjective.

Let $\theta\in A$ and set $\theta'=\psi(\theta)$.
Clearly $\hrv{\theta'}{i}=\hrv{\theta}{i}$ for all $i>y$.
Since $\dedge{y}{z}$ is a strict edge in $\theta'$ we have
\begin{eq*}
\hrv{\theta'}{y}
=
\max(\rfun{y}(\theta'),\hrv{\theta'}{z})
=
\max(\rfun{x}(\theta),\hrv{\theta}{z})
.
\end{eq*}
By assumption~\refi{II:psi_sigma:a} this implies $\hrv{\theta'}{y}=\hrv{\theta}{x}$.
On the other hand $\hrv{\theta'}{x}=\rfun{y}(\theta)$.
By assumption~\refi{II:psi_sigma:b} this implies $\hrv{\theta'}{x}=\hrv{\theta}{y}$.
For all $i\in[x-1]$ we have
\begin{eq*}
\{\hrv{\theta'}{j}:\dedge{i}{j}\in\theta'\}
=
\{\hrv{\theta}{j}:\dedge{i}{j}\in\theta\}
\end{eq*}
and therefore $\hrv{\theta'}{i}=\hrv{\theta}{i}$.
Thus $\psi$ is $\hrvp$-switching on $A$ as claimed.

To show that $\psi:A\to B$ is surjective let $\theta'\in B$.
By \refl{II:psi} there exists a unique $\theta\in\ORIENTS(P;a,b)$ such that $\psi(\theta)=\theta'$.
Clearly $\hrv{\theta}{z}=\hrv{\theta'}{z}$.
First note that
\begin{eq}{II:psi:surj_hrvy}
\hrv{\theta}{y}
\geq
\rfun{y}(\theta)
=
\rfun{x}(\theta')
=
\hrv{\theta'}{x}
.
\end{eq}
Condition~\refi{II:psi_sigma:b} implies equality in \refq{II:psi:surj_hrvy}.
Since $\dedge{x}{z}\in\theta$ and $\dedge{y}{z}\in\theta'$ are strict edges we obtain
\begin{eq}{II:psi:surj_hrvx}
\hrv{\theta}{x}
\geq
\max(\rfun{x}(\theta),\hrv{\theta}{z})
=
\max(\rfun{y}(\theta'),\hrv{\theta'}{z})
=
\hrv{\theta'}{y}
.
\end{eq}
Condition~\refi{II:psi_sigma:a} implies equality \refq{II:psi:surj_hrvx}, and the proof is complete.
\end{proof}

We now collect the orientations on which we plan to use 
the map $\psi$ in the proof of \reft{orientations:bounceI}.

\begin{lem}{II:switch}
The function $\psi$ is a $\hrvpp$-preserving bijection between sets
\begin{eq}{II:switch}
\ORIENTS(P,111;a,\dedge{z}{y})
&\to
\ORIENTS(Q,111;\dedge{y}{x})
\\
\ORIENTS(P,211;a,\dedge{z}{y})
&\to
\ORIENTS(Q,121;\dedge{y}{x})
\\
\ORIENTS(P,212;a,\dedge{z}{y})
&\to
\ORIENTS(Q,122;\dedge{y}{x})
\\
\ORIENTS(P,121;\dedge{y}{x},b)
&\to
\ORIENTS(Q,211;\dedge{y}{x})
\\
\ORIENTS(P,312;a,\dedge{z}{y})
&\to
\ORIENTS(Q,132;\dedge{y}{x})
\end{eq}
for all $a\in\{\dedge{x}{y},\dedge{y}{x}\}$ and $b\in\{\dedge{y}{z},\dedge{z}{y}\}$, 
where all surjections have domain $\{x,y,z\}$.

Let $A$ be the union of all domains in \refq{II:switch} taken over all possible choices of $a$ and $b$.
Similarly let $B$ be the union of all codomains in \refq{II:switch}.
Then $\psi:A\to B$ is a two-to-one function, and the fibres of $\psi$ satisfy
\begin{eq*}
\sum_{\theta\in(\psi)^{-1}(\theta')}q^{\asc(\theta)}
=
(q+1)q^{\asc(\theta')}
\end{eq*}
for all $\theta'\in B$.
\end{lem}

\begin{proof}
The data $\sigma,a,b$ in each of the domains satisfies 
Conditions~\refi{II:psi_sigma:a} and~\refi{II:psi_sigma:b} in \refl{II:psi_sigma}.
Hence the lemma applies.
Each codomain is in bijection with two distinct domains depending on a choice of $a$ or $b$.
The claim on the $q$-weights follows from \refl{II:psi}.
\end{proof}

The results of this section combine to a proof of the desired 
bounce relation among the symmetric functions $\LLTe_P(\xvec;q+1)$.

\begin{proof}[Proof of \reft{orientations:bounceI}]
Define a function $\Phi:\ORIENTS(P)\to\ORIENTS(Q)$ by letting
\begin{eq*}
\Phi(\theta)
=
\begin{cases}
\phi(\theta)\\
\phibar(\theta)\\
\psi(\theta)
\end{cases}
\end{eq*}
according to Table~\ref{Table:orientations:bounceI}.
The reader should verify that the table covers all orientations in $\ORIENTS(P)$ and $\ORIENTS(Q)$.
Lemmas~\ref{Lemma:II:id}, \ref{Lemma:II:changeb} and~\ref{Lemma:II:switch} 
show that $\Phi$ is a $\hrvpp$-preserving two-to-one function 
such that the fibres of $\Phi$ satisfy
\begin{eq*}
\sum_{\theta\in\phi^{-1}(\theta')}q^{\asc(\theta)}
=
(q+1)q^{\asc(\theta')}
\end{eq*}
for all $\theta'\in\ORIENTS(Q)$.
The claim follows. 
\end{proof}

\subsection{The second bounce relation for orientations}

We now treat the case $st=\dstep\nstep$.
Our goal is to prove the following theorem.

\begin{thm}{orientations:bounceIII}
Let $P\in\SCH{}$ be a Schr{\"o}der path and let $(x,z)\in\setZ^2$ be a 
point on $P$ with $z>x+1$ such that the bounce path of $P$ at $(x,z)$ has a 
single bounce point, and such that the bounce decomposition of $P$ at $(x,z)$ 
is given by $P=U\dstep\bs\nstep V\dstep\bs\estep W$.
Set $Q\coloneqq U\nstep\bs\dstep V\estep\bs\dstep W$.
Then
\begin{eq*}
\LLTe_{P}(\xvec;q+1)
=
\LLTe_{Q}(\xvec;q+1)
.
\end{eq*}
\end{thm}

A proof of \reft{orientations:bounceIII} is given at the end of this subsection.
Our strategy is to find a $\hrvpp$-preserving bijection $\Phi:\ORIENTS(P)\to\ORIENTS(Q)$ such that the $\asc(\Phi(\theta))=\asc(\theta)$ for all $\theta\in\ORIENTS(P)$.
The bijection $\Phi$ is made up of two building blocks $\phi$ and $\psi$.

Let $\phi:\ORIENTS(P)\to\ORIENTS(Q)$ be the function defined by
\begin{eq*}
\mathdefin{\phi(\theta)}
\coloneqq
\begin{cases}
\theta\setminus\{\dedge{x}{z}\}\cup\{\dedge{v}{y}\}
&\quad\text{if }\dedge{y}{z}\in\theta,
\\
\theta\setminus\{\dedge{v}{x},\dedge{x}{z},\dedge{z}{y}\}\cup\{\dedge{x}{v},\dedge{v}{y},\dedge{y}{z}\}
&\quad\text{if }\dedge{z}{y}\in\theta.
\end{cases}
\end{eq*}
In words, $\phi(\theta)$ is obtained from $\theta$ by adding and removing 
some strict edges as needed, by demanding $\dedge{x}{v}\in\phi(\theta)$ if 
and only if $\dedge{y}{z}\in\theta$, and by leaving all other edges unchanged.

\begin{exa}{IV:phiExample}
The function $\phi$ is a bijection for all choices of $\alpha,\beta \in \{\to,\dn\}$.
\[
\phi:
 \includegraphics[page=7,scale=0.8,valign=c]{figuresOrientations.pdf}%
\quad\mapsto\quad%
\includegraphics[page=8,scale=0.8,valign=c]{figuresOrientations.pdf}%
\]
\end{exa}

The following result is immediate from this definition.
\begin{lem}{IV:phi}
Let $b\in\{\dedge{x}{y},\dedge{y}{x}\}$, and 
let $(a,c)\in\{(\dedge{v}{x},\dedge{y}{z}),(\dedge{x}{v},\dedge{z}{y})\}$.
Then
\begin{eq*}
\phi:
\ORIENTS(P;b,c)
\to
\ORIENTS(Q;a,b)
\end{eq*}
is a bijection.
Moreover $\asc(\phi(\theta))=\asc(\theta)$ for all $\theta\in\ORIENTS(P)$.
\hfill\qedsymbol
\end{lem}

If we impose restrictions on $\sigma$ we can say more.
\begin{lem}{IV:phi_sigma}
Let $\sigma:\{x,y,z\}\to[m]$ be a surjection, let $b\in\{\dedge{x}{y},\dedge{y}{x}\}$, 
and let $(a,c)\in\{(\dedge{v}{x},\dedge{y}{z}),(\dedge{x}{v},\dedge{z}{y})\}$, 
such that the following conditions are satisfied:
\begin{enumerate}[(i)]
\myi{IV:phi_sigma:12}
Either $\sigma_1=\sigma_2$, or $\sigma_1>\sigma_2$ and $a=\dedge{v}{x}$.

\myi{IV:phi_sigma:13}
$\sigma_1>\sigma_3$.

\myi{IV:phi_sigma:23}
$\sigma_2>\sigma_3$ or $c=\dedge{y}{z}$.
\end{enumerate}
Then
\begin{eq*}
\phi:
\ORIENTS(P,\sigma;b,c)
\to
\ORIENTS(Q,\sigma;a,b)
\end{eq*}
is a bijection, and
$\hrvp(\phi(\theta))=\hrvp(\theta)$ for all $\theta\in\ORIENTS(P,\sigma;b,c)$.
\end{lem}

\begin{proof}
Set $A=\ORIENTS(P,\sigma;b,c)$ and $B=\ORIENTS(Q,\sigma;a,b)$.

Let $\theta\in A$, and set $\theta'=\phi(\theta)$.
By Condition~\refi{IV:phi_sigma:23} we have
\begin{eq}{IV:phi:hrvy}
\hrv{\theta}{y}
=
\max(\rfun{y}(\theta),\hrv{\theta}{z})
=
\max(\rfun{y}(\theta'),\hrv{\theta'}{z})
=
\hrv{\theta'}{y}.
\end{eq}
Similarly Condition~\refi{IV:phi_sigma:13} implies that the 
edge $\dedge{x}{z}$ is not needed to reach $\hrv{\theta}{x}$ from $x$ 
using strict and ascending edges in $\theta$.
Hence $\hrv{\theta'}{x}=\hrv{\theta}{x}$.
Consequently Condition~\refi{IV:phi_sigma:12} yields
\begin{eq*}
\hrv{\theta}{v}
=
\max(\rfun{v}(\theta),\hrv{\theta}{x})
=
\max(\rfun{v}(\theta'),\hrv{\theta'}{x})
=
\hrv{\theta'}{v}.
\end{eq*}
It follows that $\phi:A\to B$ is well-defined and $\hrvp$-preserving on $A$.

The function $\phi$ is injective by \refl{IV:phi}.
To see that $\phi:A\to B$ is also surjective let $\theta'\in B$.
By \refl{IV:phi} there exists a unique $\theta\in\ORIENTS(P;b,c)$ with $\phi(\theta)=\theta'$.
Condition~\refi{IV:phi_sigma:23} implies $\hrv{\theta}{y}=\hrv{\theta'}{y}$ as in \refq{IV:phi:hrvy}.
Similarly $\hrv{\theta'}{x}>\hrv{\theta'}{z}$ implies that $\hrv{\theta}{x}=\hrv{\theta'}{x}$.
Thus $\theta\in A$ and the proof is complete.
\end{proof}

We now collect the orientations on which we plan to 
use the map $\phi$ in the proof of \reft{orientations:bounceIII}.
\begin{lem}{IV:id}
The function $\phi$ is a $\hrvpp$-preserving bijection between the sets
\begin{eq}{IV:id}
\ORIENTS(P,211;b,\dedge{y}{z})
&\to
\ORIENTS(Q,211;\dedge{v}{x},b)
\\
\ORIENTS(P,221;b,c)
&\to
\ORIENTS(Q,221;a,b)
\\
\ORIENTS(P,321;b,\dedge{y}{z})
&\to
\ORIENTS(Q,321;\dedge{v}{x},b)
,
\end{eq}
for all $b\in\{\dedge{x}{y},\dedge{y}{x}\}$, and 
all $(a,c)\in\{(\dedge{v}{x},\dedge{y}{z}),(\dedge{x}{v},\dedge{z}{y})\}$, 
where all surjections have domain $\{x,y,z\}$.
Moreover $\phi$ preserves the number of ascents.
\end{lem}

\begin{proof}
All cases satisfy the conditions of \refl{IV:phi_sigma}.
The claim on the ascents follows from \refl{IV:phi}.
\end{proof}


Next we define a second map that covers the remaining cases.
Let $\psi:\ORIENTS(P)\to\ORIENTS(Q)$ be the function defined by 
\begin{eq*}
\dedge{x}{i}\in\psi(\theta)
&\Leftrightarrow
\dedge{y}{i}\in\theta
\\
\dedge{y}{i}\in\psi(\theta)
&\Leftrightarrow
\dedge{x}{i}\in\theta
\end{eq*}
for all $i\in\{v+1,\dotsc,x-1\}\cup\{y+1,\dotsc,z-1\}$,
\begin{eq*}
\dedge{v}{x}\in\psi(\theta)
&\Leftrightarrow
\dedge{x}{y}\in\theta,
\\
\dedge{x}{y}\in\psi(\theta)
&\Leftrightarrow
\dedge{y}{z}\in\theta,
\end{eq*}
and
\begin{eq*}
\dedge{i}{j}\in\psi(\theta)
&\Leftrightarrow
\dedge{i}{j}\in\theta\setminus\{\dedge{x}{z},\dedge{z}{y}\}\cup\{\dedge{v}{y},\dedge{y}{z}\}
\end{eq*}
for all other edges.
In words, $\psi(\theta)$ is obtained from $\theta$ by exchanging the rows $x$ and $y$, and the columns $x$ and $y$, except that the edges $\edge{v}{x},\edge{x}{y}$ and $\edge{y}{z}$ (as well as some strict edges) have to be treated separately.

\begin{exa}{IV:psiExample}
As seen in the figure here, the map $\psi$
swaps the marked columns above $x$ and $y$, and the marked rows 
to the left of $x$ and $y$. It is a bijection for all choices of $\alpha,\beta \in \{\to,\dn\}$.
\[
\psi:
 \includegraphics[page=9,scale=0.8,valign=c]{figuresOrientations.pdf}%
\quad\mapsto\quad%
\includegraphics[page=10,scale=0.8,valign=c]{figuresOrientations.pdf}%
\]
\end{exa}

It is immediate from the definition that this map is invertible and preserves the number of ascents.
\begin{lem}{IV:psi}
Let $(a,b)\in\{(\dedge{v}{x},\dedge{x}{y}),(\dedge{x}{v},\dedge{y}{x})\}$ and $(b',c)\in\{(\dedge{x}{y},\dedge{y}{z}),(\dedge{y}{x},\dedge{z}{y})\}$.
Then
\begin{eq*}
\psi:
\ORIENTS(P;b,c)
\to
\ORIENTS(Q;a,b')
\end{eq*}
is a bijection.
Moreover $\asc(\psi(\theta))=\asc(\theta)$ for all $\theta\in\ORIENTS(P)$.
\hfill\qedsymbol
\end{lem}

If we impose restrictions on $\sigma$ we can say more.

\begin{lem}{IV:psi_sigma}
Let $\sigma:\{x,y,z\}\to[m]$ be a surjection, let $(a,b)\in\{(\dedge{v}{x},\dedge{x}{y}),(\dedge{x}{v},\dedge{y}{x})\}$, and $(b',c)\in\{(\dedge{x}{y},\dedge{y}{z}),(\dedge{y}{x},\dedge{z}{y})\}$, such that
the following conditions are satisfied:
\begin{enumerate}[(i)]
\myi{IV:psi_sigma:b}
$b=\dedge{y}{x}$ or $\sigma_1>\sigma_2$ or $\sigma_2=\sigma_3$.

\myi{IV:psi_sigma:c}
$c=\dedge{z}{y}$ or $\sigma_2>\max(\sigma_1,\sigma_3)$ or $\sigma_1=\sigma_3$.
\end{enumerate}
Then
\begin{eq*}
\psi:
\ORIENTS(P,\sigma;b,c)
\to
\ORIENTS(Q,\sigma_2\sigma_1\sigma_3;a,b')
\end{eq*}
is a $\pi$-switching bijection.
\end{lem}

\begin{proof}
Let $A=\ORIENTS(P,\sigma;b,c)$ and $B=\ORIENTS(Q,\sigma_2\sigma_1\sigma_3;a,b')$.
We first show that $\psi:A\to\ORIENTS(Q)$ is $\pi$-switching.
By \refl{IV:psi} this  implies that $\psi:A\to B$ is injective and well-defined, 
that is, that $\psi$ really maps $A$ to $B$.
To conclude the proof we then show that $\psi:A\to B$ is also surjective.

Let $\theta\in A$ and set $\theta'=\psi(\theta)$.
Clearly $\hrv{\theta'}{i}=\hrv{\theta}{i}$ for all $i>y$.
Since $\dedge{y}{z}$ is a strict edge in $\theta'$ we have
\begin{eq*}
\hrv{\theta'}{y}
=
\max(\rfun{y}(\theta'),\hrv{\theta'}{z})
=
\max(\rfun{x}(\theta),\hrv{\theta}{z})
.
\end{eq*}
Condition~\refi{IV:psi_sigma:b} implies that the edge $\edge{x}{y}$ is not 
needed to reach $\hrv{\theta}{x}$ from $x$ using only strict and ascending edges in $\theta$.
Therefore we obtain $\hrv{\theta'}{y}=\hrv{\theta}{x}$.
If $(b',c)=(\dedge{y}{x},\dedge{z}{y})$ then
\begin{eq}{IV:psi:hrvx1}
\hrv{\theta'}{x}
=
\rfun{x}(\theta')
=
\rfun{y}(\theta)
=
\hrv{\theta}{y}
.
\end{eq}
On the other hand, if $(b',c)=(\dedge{x}{y},\dedge{y}{z})$ then Condition~\refi{IV:psi_sigma:c} implies
\begin{eq}{IV:psi:hrvx2}
\hrv{\theta'}{x}
&=
\max(\rfun{x}(\theta'),\hrv{\theta'}{y})
\\&
=
\max(\rfun{y}(\theta),\hrv{\theta}{x})
\\&
=
\max(\rfun{y}(\theta),\hrv{\theta}{z})
\\&
=
\hrv{\theta}{y}
.
\end{eq}
For all $i\in\{v+1,\dotsc,x-1\}$ we have
\begin{eq*}
\{\hrv{\theta'}{j}:\dedge{i}{j}\in\theta'\}
=
\{\hrv{\theta}{j}:\dedge{i}{j}\in\theta\}
\end{eq*}
and therefore $\hrv{\theta'}{i}=\hrv{\theta}{i}$.
If $a=\dedge{x}{v}$ then
\begin{eq*}
\{\hrv{\theta'}{j}:\dedge{v}{j}\in\theta'\}
=
\{\hrv{\theta}{j}:\dedge{v}{j}\in\theta\}.
\end{eq*}
On the other hand, if $(a,b)=(\dedge{v}{x},\dedge{x}{y})$ then
\begin{eq*}
\{\hrv{\theta'}{j}:\dedge{v}{j}\in\theta'\}
=
\{\hrv{\theta}{j}:\dedge{v}{j}\in\theta\}
\cup\{\hrv{\theta}{y}\}.
\end{eq*}
Condition~\refi{IV:psi_sigma:b} implies $\hrv{\theta}{x}\geq\hrv{\theta}{y}$ and consequently $\hrv{\theta'}{v}=\hrv{\theta}{v}$.
It now follows that $\hrv{\theta'}{i}=\hrv{\theta}{i}$ also for all $i\in[v]$, and thus $\psi$ is $\hrvp$-switching on $A$ as claimed.

To show that $\psi:A\to B$ is surjective let $\theta'\in B$.
By \refl{IV:psi} there exists a unique $\theta\in\ORIENTS(P;b,c)$ such that $\psi(\theta)=\theta'$.
Clearly $\hrv{\theta}{z}=\hrv{\theta'}{z}$.
If $(b',c)=(\dedge{y}{x},\dedge{z}{y})$ then $\hrv{\theta}{y}=\hrv{\theta'}{x}$ as in \refq{IV:psi:hrvx1}.
If $(b',c)=(\dedge{x}{y},\dedge{y}{z})$ then Condition~\refi{IV:psi_sigma:c} yields 
\begin{eq*}
\hrv{\theta}{y}
&=
\max(\rfun{y}(\theta),\hrv{\theta}{z})
\\&
=
\max(\rfun{x}(\theta'),\hrv{\theta'}{z})
\\&
=
\max(\rfun{x}(\theta'),\hrv{\theta'}{y})
\\&
=
\hrv{\theta'}{x}
.
\end{eq*}
If $b=\dedge{y}{x}$ then
\begin{eq*}
\hrv{\theta}{x}
=
\max(\rfun{x}(\theta),\hrv{\theta}{z})
=
\max(\rfun{y}(\theta'),\hrv{\theta'}{z})
=
\hrv{\theta'}{y}
.
\end{eq*}
Otherwise $b=\dedge{x}{y}$ and Condition~\refi{IV:psi_sigma:b} implies
\begin{eq*}
\hrv{\theta}{x}
=
\max(\rfun{y}(\theta'),\hrv{\theta'}{x},\hrv{\theta'}{z})
=
\hrv{\theta'}{y}.
\end{eq*}
This completes the proof.
\end{proof}

We now collect the orientations on which we plan to use the 
map $\psi$ in the proof of \reft{orientations:bounceIII}.

\begin{lem}{IV:switch}
The function $\psi$ is a $\hrvpp$-preserving bijection between the sets
\begin{eq}{IV:switch}
\ORIENTS(P,111;b,c)
&\to
\ORIENTS(Q,111;a,b')
\\
\ORIENTS(P,211;b,\dedge{z}{y})
&\to
\ORIENTS(Q,121;a,\dedge{y}{x})
\\
\ORIENTS(P,212;b,\dedge{z}{y})
&\to
\ORIENTS(Q,122;a,\dedge{y}{x})
\\
\ORIENTS(P,121;\dedge{y}{x},c)
&\to
\ORIENTS(Q,211;\dedge{x}{v},b')
\\
\ORIENTS(P,312;b,\dedge{z}{y})
&\to
\ORIENTS(Q,132;a,\dedge{y}{x})
\\
\ORIENTS(P,231;\dedge{y}{x},c)
&\to
\ORIENTS(Q,321;\dedge{x}{v},b')
\\
\ORIENTS(P,321;b,\dedge{z}{y})
&\to
\ORIENTS(Q,231;a,\dedge{y}{x})
\end{eq}
for all $(a,b)\in\{(\dedge{v}{x},\dedge{x}{y}),(\dedge{x}{v},\dedge{y}{x})\}$, and all $(b',c)\in\{(\dedge{x}{y},\dedge{y}{z}),(\dedge{y}{x},\dedge{z}{y})\}$, where all surjections have domain $\{x,y,z\}$.
Moreover the map $\psi$ preserves the number of ascents in each case. 
\end{lem}

\begin{proof}
The data $\sigma,b,c$ in each of the domains satisfies the conditions of \refl{IV:psi_sigma}.
The claim on the ascents follows from \refl{IV:psi}.
\end{proof}

The results of this section combine to a proof of the desired 
relation among the symmetric functions $\LLTe_P(\xvec;q+1)$.

\begin{proof}[Proof of \reft{orientations:bounceI}]
Define a function $\Phi:\ORIENTS(P)\to\ORIENTS(Q)$ by letting
\begin{eq*}
\Phi(\theta)
=
\begin{cases}
\phi(\theta)\\
\psi(\theta)
\end{cases}
\end{eq*}
according to Table~\ref{Table:orientations:bounceII}.
The reader should verify that the table covers all orientations in $\ORIENTS(P)$ and $\ORIENTS(Q)$.
Lemmas~\ref{Lemma:IV:id} and~\ref{Lemma:IV:switch} provide 
that $\Phi$ is a $\hrvpp$-preserving bijection that preserves the number of ascents.
The claim follows. 
\end{proof}

\section{Applications and further directions}\label{Section:applications}

In this final section we briefly visit some areas that are related to our results, and collect several consequences and open problems.
We hope that readers familiar with some but not all of these topics can use it as a starting point for learning more.

\subsection{Schur expansions}

It is a major open problem to find an explicit positive 
formula for the Schur expansion of $\LLT_P(\xvec;q)$.
In this subsection we compare a new signed Schur expansion with the previously known signed Schur expansion.
The following description is a straightforward application 
of techniques in \cite{HaglundHaimanLoehr2005}, 
and appears in \cite[Eq.~3.16]{GarsiaHaglundQiuRomero2019}.
Given a permutation $\sigma \in \symS_n$, let the \defin{inverse ascent set} 
$\ASC(\sigma)$ be defined as
\[
 \mathdefin{\ASC(\sigma)} \coloneqq \{ i \in [n-1] : 
 \sigma^{-1}(i) > \sigma^{-1}(i+1) \}.
\]
With $\ASC(\sigma) = \{d_1,d_2,\dotsc,d_\ell\}$ set
\[
\mathdefin{\alpha(\sigma)} \coloneqq (d_1,d_2-d_1,d_3-d_2\dotsc,d_{\ell}-d_{\ell-1},n-d_{\ell}).
\]
Let $P$ be a Schr{\"o}der path of size $n$ and 
let $\mathdefin{\symS_n(P)}$ denote the set of colorings $\sigma$ of $P$ such that 
$(\sigma(1),\sigma(2),\dotsc,\sigma(n))$ is a permutation in $\symS_n$.
The Egge--Loehr--Warrington theorem~\cite{EggeLoehrWarrington2010}
yields the Schur expansion
\begin{eq}{ELG_Schur}
 \LLT_P(\xvec;q) = \sum_{\sigma \in \symS_n(P)}
 q^{\asc(\sigma)}
 \schurS_{\alpha(\sigma)}(\xvec).
\end{eq}
Note that $\alpha(\sigma)$ is a composition in general.
The expression $\schurS_{\alpha(\sigma)}$ should be computed via the Jacobi--Trudi identity and simplifies either to $0$, or a Schur function with a sign.

\begin{exa}{elw}
For $P=\nstep\nstep\dstep\estep\estep$ we have three colorings in ${\symS_n(P)}$
with the inverse ascent sets $\{1,2\}$, $\{2\}$ and $\{1\}$:
\[
\ytableausetup{boxsize=0.8em}
\begin{ytableau}
*(lightgray) \to  &  & *(yellow) 3 \\
  & *(yellow)  2 \\
*(yellow)  1
\end{ytableau}
\,\,
\{1,2\}
\qquad 
\begin{ytableau}
*(lightgray) \to  &  & *(yellow) 2 \\
  & *(yellow)  3 \\
*(yellow)  1
\end{ytableau}
\,\,
\{2\}
\qquad 
\begin{ytableau}
*(lightgray) \to  &  & *(yellow) 3 \\
  & *(yellow)  1 \\
*(yellow)  2
\end{ytableau}
\,\,
\{1\}
.
\]
The corresponding compositions are $(1,1,1)$, $(2,1)$ and $(1,2)$.
The function $\schurS_{12}(\xvec)$ is identically 0, so
\[
 \LLT_P(\xvec;q) = q^2\schurS_{111}(\xvec) + q\schurS_{21}(\xvec).
\]
\end{exa}

The results in this paper yield a new Schur expansion of vertical-strip LLT polynomials. 

\begin{cor}{schurExpansion}
The Schur expansion of $\LLT_P(\xvec;q)$ is given by
\begin{eq}{schurExpansion}
\LLT_P(\xvec;q) =
\sum_{\mu}
\sum_{\theta \in \ORIENTS(P)}
(q-1)^{\asc(\theta)} K_{\mu',\hrvpp(\theta)}
\;
\schurS_\mu(\xvec),
\end{eq}
where $K_{\mu\lambda}$ is a Kostka coefficient.
\end{cor}

\begin{exa}{schurExpansion}
For $P=\nstep\nstep\dstep\estep\estep$ we have four orientations in ${\ORIENTS(P)}$
with $\hrvpp(\theta)$ given by
\begin{eq*}
\ytableausetup{boxsize=0.8em}
\begin{ytableau}
*(lightgray) \to  & \to & *(yellow) 3 \\
 \to & *(yellow)  2 \\
*(yellow)  1
\end{ytableau}
\,\,
3
\qquad 
\begin{ytableau}
*(lightgray) \to  & \to & *(yellow) 3 \\
 \dn & *(yellow)  2 \\
*(yellow)  1
\end{ytableau}
\,\,
3
\qquad 
\begin{ytableau}
*(lightgray) \to  & \dn & *(yellow) 3 \\
 \to & *(yellow)  2 \\
*(yellow)  1
\end{ytableau}
\,\,
21
\qquad 
\begin{ytableau}
*(lightgray) \to  & \dn & *(yellow) 3 \\
 \dn & *(yellow)  2 \\
*(yellow)  1
\end{ytableau}
\,\,
21
.
\end{eq*}
Since $K_{\mu\lambda}=0$ unless $\mu$ dominates $\lambda$, we have
\begin{eq*}
 \LLT_P(\xvec;q) &= 
  \left( 
 (q-1)^2 K_{3,3} + (q-1)^1 K_{3,3} + (q-1)^1 K_{3,21} + (q-1)^0 K_{3,21} 
 \right)\schurS_{111} \\
&\phantom{=}+ 
 \left( (q-1)^1 K_{21,21} + (q-1)^0 K_{21,21} 
 \right)\schurS_{21}. \\
\end{eq*}
Since $K_{3,3}=K_{3,21}=K_{21,21}=1$, we see that this expression matches the previous example.
\end{exa}
One advantage of \refq{schurExpansion} over the formula
in \refq{ELG_Schur} is that is provides a more explicit formula for the coefficient of a specific Schur function.
Moreover, it might be easier to come up with a sign-reversing involution on the objects in \refq{schurExpansion} which solves the Schur-positivity problem.
Finally we note that the Schur coefficients of hook shapes 
in the unicellular LLT polynomials are determined in \cite{HuhNamYoo2020}.

\subsection{Dual bounce relations}

Given $P=s_1 s_2 \dotsc s_{\ell}\in\SCH{}$ define the \defin{reverse of $P$}, 
denoted by $\mathdefin{\rev(P)}$, as the Schr{\"o}der path $\hat{s}_{\ell} \hat{s}_{\ell-1} \dotsc \hat{s}_{1}$ where $\hat{\estep}=\nstep$, $\hat{\nstep}=\estep$ and $\hat{\dstep}=\dstep$.
A quite surprising fact about vertical-strip LLT polynomials is that 
they are invariant under reversal of the Schr{\"o}der path.

\begin{lem}{transposeInvariant}
Let $P \in \SCH{n}$.
Then
\begin{eq}{transposeInvariant}
 \LLT_P(\xvec;q) = \LLT_{\rev(P)}(\xvec;q).
\end{eq}
\end{lem}
Similar observations have been made before, for example in~\cite[Prop.~3.3]{CarlssonMellit2017}.

\begin{proof}
Since the LLT polynomials $\LLT_P$ and $\LLT_{\rev(P)}$ are homogeneous symmetric functions of degree $n$, 
it is enough to show that 
\begin{eq}{transposeEquiv}
 \LLT_P(x_1,x_2,\dotsc,x_n;q) = \LLT_{\rev(P)}(x_n,\dotsc,x_2,x_1;q).
\end{eq}
Given a coloring $\kappa$ of $P$, define a coloring $\kappa'$ of $\rev(P)$ via $\kappa'(j) \coloneqq n+1 - \kappa(n+1-j)$ for all $j \in [n]$.
Then $\asc(\kappa)=\asc(\kappa')$.
By definition of vertical-strip LLT polynomials this implies \refq{transposeEquiv}.
\end{proof}

As a corollary, vertical-strip LLT polynomials satisfy the \defin{dual bounce relations}.
These are the relations obtained from the bounce relations in \reft{LLT} by reversing all paths.

\begin{cor}{dual:bounce}
The function $F:\SCH{}\to\Lambda$ defined by $P\mapsto\LLT_P\xqvar$ satisfies the relations obtained from the bounce relations in \reft{LLT} by reversing all paths.
Moreover this function is uniquely determined by the conditions in~\refti{recursion}{e}--\refi{recursion:unicellular} together with the dual bounce relations obtained by reversing all paths in \refti{recursion}{bounce} (or, alternatively, in \refpi{equivalentRelations}{B}).
\end{cor}

\begin{cor}{dual:bounce:dyck}
The function $F:\DYCK{}\to\Lambda$ defined by $P\mapsto\LLT_P\xqvar$ is uniquely determined by the conditions in~\refti{dyckRecursion}{e} and \refi{dyckRecursion:mult} together with the dual bounce relations obtained by reversing all paths in \refti{dyckRecursion}{bounce}.
\end{cor}

Note that some sets of relations we have encountered in this paper are \defin{self-dual}, that is, invariant under reversal of all underlying paths.
Examples are the unicellular relation and the second type of bounce relations in \reft{LLT} --- the case $st=\dstep\nstep$.
Moreover the initial conditions in \refti{recursion}{e} and \refti{dyckRecursion}{e} are self-dual.
However \refl{transposeInvariant} is not manifestly evident, neither from the characterization of vertical-strip LLT poylnomials in \reft{recursion}, nor from the $\elementaryE$-expansion in \refc{expansion}.

\begin{prob}{dual_orientations}
Give a combinatorial explanation of the identity
\begin{eq}{dual_orientations}
\sum_{\theta \in \ORIENTS(P)} q^{\asc(\theta)}
\elementaryE_{\hrvpp(\theta)}(\xvec)
=
\sum_{\theta' \in \ORIENTS(\rev(P))} q^{\asc(\theta')}
\elementaryE_{\hrvpp(\theta')}(\xvec).
\end{eq} 
\end{prob}

\begin{prob}{dual_relations}
Deduce the dual bounce relations directly from the bounce relations (in the style of the proof of \refp{equivalentRelations}).
\end{prob}

Moreover it is an interesting question if there are other sets of relations that uniquely determine the vertical-strip LLT polynomials.

\begin{prob}{bounce+dual}
Combining bounce relations and dual bounce relations,
can we find other results in the style of \reft{recursion}?
\end{prob}
An answer to this question can be found in the recent \cite{AbreuNigro2020ax},
where the Dyck relation and the dual is shown to uniquely define the 
chromatic quasisymmetric functions.

\subsection{Hall--Littlewood polynomials}

For a Dyck path $P$ it is noted in \cite{AlexanderssonPanova2016}
that $\LLT_{P}(\xvec;q) = \LLT_{\rev(P)}(\xvec;q)$
and that
\begin{eq}{omegaRel}
\omega \LLT_{P}(\xvec;q) = q^{\area(P)}\LLT_{P}(\xvec;q^{-1}),
\end{eq}
and that a similar relation holds for general Schr{\"o}der paths,
where vertical-strip LLT polynomials are mapped to horizontal-strip LLT polynomials.
Hence, there is a formula analogous to \refq{conjFormula} 
for the horizontal-strip LLT polynomials.
As a special case, we can say something about 
the transformed Hall--Littlewood polynomials.
The \defin{transformed Hall--Littlewood polynomial} $\hallLittlewoodH_\mu(\xvec;q)$
may be defined as follows (see \cite[Eq.~(2)+(11)]{TudoseZabrocki2003} for a reference).
For $\lambda \vdash n$ let
\begin{eq}{hallLittlewoodHDef}
 \mathdefin{\hallLittlewoodH_\lambda(\xvec;q)} \coloneqq \prod_{1\leq i<j \leq n} \frac{1 - R_{ij}}{1 - qR_{ij}} \completeH_{\lambda}(\xvec)
\end{eq}
where the $R_{ij}$ are \defin{raising operators} acting on the 
partitions (or compositions) indexing the complete homogeneous symmetric functions
by
\begin{eq*}
\mathdefin{R_{ij}} \completeH_{(\lambda_1,\dotsc,\lambda_n)}(\xvec) \coloneqq 
\completeH_{(\lambda_1,\dotsc,\lambda_i +1, \dotsc, \lambda_j-1,\dotsc,\lambda_n)}(\xvec)
.
\end{eq*}
The transformed Hall--Littlewood polynomials 
can be expressed (see for example~\cite[Prop.~38]{Alexandersson2019llt}) 
via a vertical-strip LLT polynomial as
\[
 \hallLittlewoodH_{\mu'}(\xvec;q) = q^{-\sum_{i\geq 2} \binom{\mu_i}{2} } \omega \LLT_{P_\mu}(\xvec;q) 
\]
where 
\[
\mathdefin{P_\mu} \coloneqq 
(\nstep^{\mu_1})
(\estep^{\mu_1-\mu_2})
(\dstep^{\mu_2})
(\estep^{\mu_2-\mu_3})
(\dstep^{\mu_3}) 
\dotsm 
(\estep^{\mu_{\ell-1}-\mu_\ell})
(\dstep^{\mu_\ell}) 
(\estep^{\mu_\ell}).
\]
Hence, our main result implies that
\begin{eq}{newHLPolyFormula}
 \hallLittlewoodH_{\mu'}(\xvec;q+1) = (q+1)^{-\sum_{i\geq 2} \binom{\mu_i}{2} }
 \sum_{\theta \in \ORIENTS(P_\mu)} q^{\asc(\theta)} \completeH_{\hrvpp(\theta)}(\xvec).
\end{eq}
It would be interesting to see if 
there is a direct combinatorial proof of this identity connecting 
Kostka--Foulkes polynomials and ascents. 
Moreover, perhaps it is possible to interpret the raising operators 
used to define the transformed Hall--Littlewood polynomials
on the level of orientations.

\subsection{Diagonal harmonics}

We shall briefly sketch some consequences in diagonal harmonics.
For more background on this topic, we refer to \cite{Bergeron2009,qtCatalanBook}.

Let us first recall the definition of the operators $\nabla$ and $\Delta'_f$.
They act on the space $\Lambda(q,t)$ of symmetric functions over $\setQ(q,t)$, and 
the \defin{modified Macdonald polynomials} are eigenvectors.
First set $\mathdefin{B(\mu)} \coloneqq \sum_{(i,j)\in \mu} (q^{i-1} t^{j-1})$,
where the sum ranges over all cells in the Young diagram of $\mu$.
We then define the linear operators $\nabla$ and $\Delta'_f$ by
\begin{eq}{nablaDef}
 \mathdefin{\nabla} \macdonaldH_\mu(\xvec;q,t) &\coloneqq 
 \elementaryE_n[B(\mu)] \cdot  \macdonaldH_\mu(\xvec;q,t), 
 \\
 \mathdefin{\Delta'_f} \macdonaldH_\mu(\xvec;q,t) &\coloneqq 
 f[B(\mu)-1] \cdot  \macdonaldH_\mu(\xvec;q,t), 
\end{eq}
where we use plethystic substitution.
For example, 
\[
\elementaryE_2[B(2,1)] = \elementaryE_2[1+q+t] = \elementaryE_2(1,q,t)=q+t+qt.
\]

Let $\mathdefin{\mathcal{F}_n(\xvec;q,t)} \coloneqq \nabla \elementaryE_n$.
M.~Haiman showed \cite{Haiman1994} that $\mathcal{F}_n(\xvec;q,t)$
is the bigraded Frobenius series of the space of diagonal coinvariants in two sets of $n$ variables where $\symS_n$ act diagonally. 
A great deal of research has been devoted to give combinatorial interpretations of 
$\nabla f$, for various choices of symmetric function $f$.

In \cite{HaglundHaimanLoehrRemmelUlyanov2005}, the authors presented a conjectured
combinatorial formula for $\mathcal{F}_n(\xvec;q,t)$, the \defin{Shuffle Conjecture}.
This conjecture was later refined in \cite{HaglundMorseZabrocki2012},
where the \defin{Compositional Shuffle Conjecture} was stated.
This refined conjecture was proved by E.~Carlsson and A.~Mellit~\cite[Prop.~2.4]{CarlssonMellit2017}.
\begin{thm}{shuffleTheorem}[The Shuffle Theorem]
 We have that
\begin{eq}{shuffleThm}
 \mathcal{F}_n(\xvec;q,t) = \sum_{w \in \mathrm{WPF}(n)} q^{\area(w)} t^{\dinv(w)} \xvec_w,
\end{eq}
where the sum is over all \emph{word parking functions} of size $n$.
\end{thm}
By applying the $\zeta$-map of J.~Haglund \cite{qtCatalanBook}
and interpreting the result, we obtain the equivalent statement 
now using vertical-strip LLT polynomials as they are defined in this paper. We have that
\begin{eq}{shuffleThmLLT}
 \mathcal{F}_n(\xvec;q,t) = \sum_{P \in \DYCK{n}} t^{\bounce(P)} \LLT_{P^*}(\xvec;q),
\end{eq}
where $P^*$ is the Schr{\"o}der path obtained from the Dyck path $P$
by replacing every instance of a pattern $\estep\nstep$ by a diagonal step $\dstep$.
By applying \refc{expansion} we obtain a new combinatorial formula for 
$\mathcal{F}_n(\xvec;q+1,t)$.

\begin{cor}{shuffleCor}
We have that 
\begin{eq*}
\mathcal{F}_n(\xvec;q+1,t)
=
\sum_{P \in \DYCK{n}} t^{\bounce(P)}
\sum_{\theta \in \ORIENTS(P^*)}
q^{\asc(\theta)} \elementaryE_{\hrvpp(\theta)}(\xvec)
.
\end{eq*}
\end{cor}
We also note that vertical-strip LLT polynomials appear in a similar manner
in the rational shuffle conjecture \cite{BergeronGarsiaLevenXin2015}
and its subsequent proof \cite{Mellit2016x}.
Similar observations can be used to show that the conjectured
expressions for $ \Delta'_{\elementaryE_k} \elementaryE_n$
appearing in the \defin{Delta Conjecture} by J.~Haglund, J.~Remmel and A.~Wilson~\cite{HaglundRemmelWilson2018} are also $\elementaryE$-positive after the substitution $q \mapsto q+1$.

We also note that the expression in \cite[Thm. 3.2.6]{BergeronCeballosPilaud2018x},
where we have $r$ sets of variables (compared to the one or two sets of variables in
the main literature of diagonal harmonics). In particular,
an open question, \cite[Conj. 3.2.17]{BergeronCeballosPilaud2018x}, 
seem to indicate a close connection with vertical-strip LLT polynomials,
due to an $\elementaryE$-positivity phenomenon.

Another interesting result is the combinatorial interpretation of 
$(-1)^{n-1} \nabla \powerSum_n$.  The \defin{Square Path Conjecture} 
was formulated by N.~Loehr and G.~Warrington~\cite{LoehrWarrington2007}.
Later E.~Sergel~\cite{Sergel2017} gave a proof of the Square Path Conjecture by using the Compositional Shuffle Theorem.
\begin{thm}{squarePathsThm}[E.~Sergel, (2017)]
We have the expansion into fundamental quasisymmetric functions
\begin{eq*}
(-1)^{n-1} \nabla \powerSum_n
=
\sum_{w \in \mathrm{Pref}(n)} 
t^{\area(w)} q^{\dinv(w)}
\gessel_{\mathrm{IDES}(\sigma(w))}(\xvec)
,
\end{eq*}
where the sum is taken over all \defin{preference functions} of size $n$,
and $\mathdefin{\sigma(w)}$ is the \defin{reading word} of $w$,
and $\mathdefin{\dinv}$ is a statistic closely related to 
the statistic $\asc$ on colorings.
\end{thm}

There is again a close connection with vertical-strip LLT polynomials.
We shall give an example explaining the connection.
\begin{exa}{prefFunc}
A preference function is simply a map $f:[n]\to [n]$.
We can describe such an $f$ by writing $f^{-1}(i)$
in column $i$ in an $n \times n$-diagram, sorted increasingly from the bottom.
Moreover, we demand that the entries $f^{-1}(i)$ appear in rows lower than
$f^{-1}(j)$ for all $i<j$. This uniquely determines the diagram.
The numbers in the digram are called \defin{cars}\footnote{This is due to the close connection with parking functions.}.
There is a unique North-East path from $(0,0)$ to $(n,n)$
where each North step is immediately to the left of a car.
We set $\mathdefin{\alpha_i \coloneqq |f^{-1}(i)|}$.
For example, the preference function $f = (5,2,5,2,2,3)$
give rise to the diagram (a) here:
\begin{eq}{prefFunc}
(a)\;
\begin{tikzpicture}
\fill[lightblue] (1,0)--(1,3)--(2,3)--(2,4)--(4,4)--(4,6)--(6,6)--(6,5)--
(5,5)--(5,4)--(4,4)--(4,3)--(3,3)--(3,2)--(2,2)--(2,1)--(1,1)--(1,0);
\draw[step=1em,gray] (0,0) grid (6,6);
\draw[thickLine] (0,0)--(1,0)--(1,3)--(2,3)--(2,4)--(4,4)--(4,6)--(6,6);
\draw[thinLine,dashed] (2,1)--(6,5);
\draw[entries]
(2,3)node{$5$}
(2,2)node{$4$}
(2,1)node{$2$}
(3,4)node{$6$}
(5,6)node{$3$}
(5,5)node{$1$};
\end{tikzpicture}
\qquad 
\qquad 
(b)\;
\begin{tikzpicture}
\draw[thinLine,black!10] (0,1)--(5,6);
\draw[thinLine,black!10] (0,0)--(6,6);
\draw[thinLine,black!10] (1,0)--(6,5);
\draw[step=1em,gray] (0,0) grid (6,6);
\draw[thickLine] (0,0)--(1,0)--(1,3)--(2,3)--(2,4)--(4,4)--(4,6)--(6,6);
\draw[entries]
(2,3)node{$3$}
(2,2)node{$5$}
(2,1)node{$6$}
(3,4)node{$2$}
(5,6)node{$1$}
(5,5)node{$4$};
\end{tikzpicture}
\qquad 
\qquad 
(c)\;
\ytableausetup{boxsize=0.9em,aligntableaux=center}
\begin{ytableau}
*(lightgray)  &*(lightgray) & *(lightgray) & *(lightgray)  & *(lightgray) \to & *(yellow) 6 \\
*(lightgray)  & *(lightgray)   & *(lightgray) \to  & & *(yellow) 5 \\
*(lightgray) \to & &  & *(yellow) 4 \\
   &   & *(yellow) 3 \\
   & *(yellow)  2 \\
*(yellow)  1
\end{ytableau}
\end{eq}
The path $P$ is given by $\mathtt{ennneneennee}$. Moreover, $\alpha = (0,3,1,0,2,0)$,
and note that $P$ is uniquely determined by $\alpha$.
There are a few other statistics on preference functions 
that only depend on $\alpha$. We let $\mathdefin{\area(\alpha)}$ be the number 
of squares below $P$ and strictly above the lowest 
diagonal of the form $y=x+k$, containing a car.
In our diagram above, $\area(\alpha)=8$.
Moreover, we let $\mathdefin{\mathrm{below}(\alpha)}$
be the number of cars strictly below the $y=x$ line. 
In this case, $\mathrm{below}(\alpha) = 1$, as the car labeled $2$ is below the main diagonal.

Finally, given $\alpha$, we associate a Schr{\"o}der path $\mathdefin{P(\alpha)}$
as follows. We read the cars diagonal by diagonal as in (b), this gives 
the \defin{reading order} of the cars.
Then $\dedge{x}{y}$ is a diagonal steps in the Schr{\"o}der path 
if and only if label $x$ is immediately below $y$ in the reading word order (b).
We note that the map that takes $P$ to the Schr{\"o}der path is also a type of $\zeta$-map,
see \cite[Sec.~5.2]{ArmstrongHanusaJones2014}.
\end{exa}

We can now state the expansion of $(-1)^{n-1} \nabla \powerSum_n$
in terms of vertical-strip LLT polynomials.
It is straightforward to deduce this from \reft{squarePathsThm},
by unraveling the definitions.

\begin{thm}{squarePathsThmLLT}
We have the expansion
\begin{eq*}
(-1)^{n-1} \nabla \powerSum_n
=
\sum_{\alpha}
t^{\area(\alpha)} q^{\mathrm{below}(\alpha)} \LLT_{P(\alpha)}(\xvec;q),
\end{eq*}
where the sum is taken over all 
weak compositions $\alpha$ with $n$ parts, of size $n$,
and $P(\alpha)$ is as in \refe{prefFunc}.
\end{thm}

\begin{table}
\begin{tabular}{ll}
\toprule
$n$ & $(-1)^{n-1} \nabla \powerSum_n$ \\
\midrule
1 & $\schurS_{1}$ \\
2 & $\schurS_{2} + (q + t + q t) \schurS_{11}$ \\
3 & $\schurS_{3} +  (q + q^2 + t + q t + q^2 t + t^2 + q t^2 + q^2 t^2)\schurS_{21} +$ \\
  & $+(q^3 + q t + q^2 t + q^3 t + q t^2 + q^2 t^2 + q^3 t^2 + t^3 + q t^3 +
   q^2 t^3) \schurS_{111}$ \\
\bottomrule
\end{tabular}
\caption{The Schur expansion of $(-1)^{n-1} \nabla \powerSum_n$ for the first few $n$.}
\label{Table:nablaPSchur}
\end{table}

By applying \refc{expansion}
we again obtain a new combinatorial expression for $(-1)^{n-1} \nabla \powerSum_n$
and note that it is $\elementaryE$-positive
if we replace $q$ with $q+1$.
We also mention \cite{Bergeron2013} and \cite[Sec.~4]{Bergeron2017}, 
where related $\elementaryE$-positivity conjectures are made.

\subsection{The Dyck path algebra}
The \defin{Dyck path algebra} $\mathdefin{\DyckPathAlgebra}$ is defined by E.~Carlsson and A.~Mellit in~\cite[Def.~5.1]{CarlssonMellit2017} as the path algebra over $\field$ of the quiver with vertex set $\setN$ with directed edges $d_{+}$ from $k$ to $k+1$, directed edges $d_{-}$ from $k+1$ to $k$, and loops $T_1,\dots,T_{k-1}$ at $k$ for all $k\in\setN$, subject to the relations
\begin{eq*}
(T_i-1)(T_i+q)=0,\quad
T_iT_{i+1}T_i=T_{i+1}T_iT_{i+1},\quad
T_iT_j=T_jT_i\text{ for }\abs{i-j}>1,
\end{eq*}
\begin{eq*}
T_id_{-}=d_{-}T_i,\quad
d_{+}T_i=T_{i+1}d_{+},\quad
T_1d_{+}^2=d_{+}^2,\quad
d_{-}^2T_{k-1}=d_{-}^2,
\end{eq*}
\begin{eq*}
d_{-}\varphi T_{k-1}=
q\varphi d_{-}\text{ for }k>1,\quad
\text{and}\quad
T_1\varphi d_{+}=qd_{+}\varphi
.
\end{eq*}
Here $k$ denotes the vertex on which the respective paths begin, and we denote
$\varphi\coloneqq(d_{-}d_{+}-d_{+}d_{-})/(q-1)$.

The Dyck path algebra can be realised as an algebra of operators acting on (generalized) symmetric functions.
For $k\in\setN$ let $\mathdefin{V_k}\coloneqq\Lambda\otimes\setQ[y_1,\dots,y_k]$, such that $V_0=\Lambda$, and set
\begin{eq*}
\mathdefin{V_{*}}
\coloneqq
\bigoplus_{k\in\setN}V_k
.
\end{eq*}
For $k\in\setN$ and $i\in[k-1]$ define $\mathdefin{T_i}:V_k\to V_k$ by
\begin{eq*}
T_i(f(\xvec,q,\yvec))
\coloneqq
\frac{(q-1)y_{i+1}f(\dots y_i,y_{i+1}\dots)+(y_{i+1}-qy_i)f(\dots y_{i+1},y_i\dots)}{y_{i+1}-y_i}
.
\end{eq*}
For all $k\in\setN$ define a \defin{raising operator} $\mathdefin{d_{+}}:V_k\to V_{k+1}$ and a \defin{lowering operator} $\mathdefin{d_{-}}:V_{k+1}\to V_k$ by
\begin{eq*}
d_{+}(f(\xvec;q;\yvec))
&\coloneqq
T_1T_2\cdots T_k(f[\xvec+(q+1)y_{k+1};q;\yvec])
,\\
d_{-}(f(\xvec;q;\yvec))
&\coloneqq
\Big(-f[\xvec-(q-1)y_{k+1};q;\yvec]
\sum_{i\in\setN}\frac{1}{(-y_{k+1})^i}\elementaryE_i(\xvec)\Big)\Big|_{y_{k+1}^{-1}}
.
\end{eq*}
Here brackets denote plethystic substitution in the $\xvec$-variables, and $g|_{y_k^i}$ denotes the coefficient of $y_k^i$ in $g$.
Finally for all $k\geq 1$ define $\mathdefin{\varphi}:V_k\to V_k$ by
\begin{eq*}
\varphi
\coloneqq
\frac{1}{q-1}(d_{-}d_{+}-d_{+}d_{-}).
\end{eq*}
To each Schr{\"o}der path $P=s_1s_2\cdots s_{\ell}$ we associate an operator
\begin{eq*}
\mathdefin{d_P}
\coloneqq
d_1d_2\cdots d_{\ell}
\qquad
\text{where}
\qquad
d_i
=
\begin{cases}
d_{-}&\quad\text{if }s_i=\nstep,\\
\varphi&\quad\text{if }s_i=\dstep,\\
d_{+}&\quad\text{if }s_i=\estep,
\end{cases}
\end{eq*}
for all $i\in[\ell]$.
One of the fundamental insights of E.~Carlsson and A.~Mellit is that these path operators are creation operators of vertical-strip LLT polynomials.

\begin{thm}{pathCharacteristicFunction}[{{\cite[Thm.~4.2]{CarlssonMellit2017}}}]
Let $P$ be a Schr{\"o}der path. Then $d_P(1)=\LLT_P(\xvec;q)$.
\end{thm}

In \cite[Sec.~4]{CarlssonMellit2017} \reft{pathCharacteristicFunction} is proven by analyzing the effect of the raising and lowering operators on characteristic functions of partial Dyck paths.
This can be compared to our proof of the bounce relations in \refs{colorings}, which has the advantage that we do not need to introduce additional variables.
The following theorem is a consequence of the results of  M.~D'Adderio in~\cite{DAdderio2020}.

\begin{thm}{DAdderio}[{{\cite{DAdderio2020}}}]
Let $F:\SCH{}\to\Lambda$ be defined by $P\mapsto F_P=d_P(1)$.
Then $F$ satisfies the conditions in \reft{recursion}.
\end{thm}

In fact, this is precisely how \reft{recursion} was discovered.
M.~D'Adderio proceeds to use \reft{DAdderio} in conjunction with \reft{pathCharacteristicFunction} to prove the $\elementaryE$-positivity of vertical-strip LLT polynomials.
On the other hand, combining \reft{DAdderio} with Theorems~\ref{Theorem:recursion} and~\ref{Theorem:LLT} one obtains an alternative proof of \reft{pathCharacteristicFunction}.

\subsection{Chromatic quasisymmetric functions}

Chromatic symmetric functions were defined by R.~Stanley as a symmetric function 
generalization of the chromatic polynomial of a graph in \cite{Stanley1995}.
These functions have received much recent attention since J.~Shareshian and M.~Wachs 
introduced a quasisymmetric $q$-analog in \cite{ShareshianWachs2016}.
E.~Carlsson and A.~Mellit observed a plethystic relationship between 
chromatic quasisymmetric functions $\mathdefin{\chrom_P(\xvec;q)}$ of unit-interval 
graphs and unicellular LLT polynomials.
This was implicitly done already in \cite[Sec.~5.1]{HaglundHaimanLoehr2005}.
It is also hidden in \cite[Eq.~(178)]{GuayPaquet2016x}.
The relationship is as follows:

\begin{lem}{plethRelation}[{{\cite[Prop.~3.5]{CarlssonMellit2017}}}]
Let $P$ be a Dyck path of size $n$. Then
\begin{eq}{plethRelation}
\chrom_P(\xvec;q)
=
(q-1)^{-n} \LLT_P[\xvec(q-1);q]
,
\end{eq}
where the bracket denotes a substitution using plethysm.
\end{lem}

Using \refl{plethRelation} and \reft{dyckRecursion} we obtain a new 
characterization of chromatic quasisymmetric functions of unit-interval graphs.

\begin{cor}{chromaticRecursion}
The function that assigns to each Dyck path $P$ the chromatic quasisymmetric function of the unit-interval graph $\uig_P$ is the unique function $F:\DYCK{}\to\Lambda$, $P\mapsto F_P\xqvar$ that satisfies the following conditions:

\begin{enumerate}[(i)]
\myi{chromaticRecursion:e}
For all $k\in\setN$  the initial condition 
$F_{\nstep(\nstep\estep
)^{k}\estep}\xqvar
=
\chrom_{\nstep(\nstep\estep
)^{k}\estep}\xqvar
$ holds, where $\chrom_{\nstep(\nstep\estep
)^{k}\estep}\xqvar$ denotes the chromatic quasisymmetric function of the path on $k+1$ vertices\footnote{See \cite{ShareshianWachs2016} or \cite{Athanasiadis2015} for explicit formulas in terms of elementary symmetric functions, respectively power-sum symmetric functions.}.
\myi{dyckRecursion:mult:chrom}
The function $F$ is multiplicative, that is,
$F_{PQ}\xqvar
=
F_{P}\xqvar
F_{Q}\xqvar
$
for all $P,Q\in\DYCK{}$.

\setcounter{enumi}{5}
\myi{chromaticRecursion:bounce}
Let $P\in\DYCK{}$ be a Dyck path and let $(x,z)\in\setZ^2$ be a point on $P$ with $x+1<z$.
If the bounce decomposition of $P$ at $(x,z)$ is given by $P=U\nstep\bs\nstep V\nstep\estep\bs\estep W$ then
\begin{eq*}
F_{U\nstep\bs\nstep V\nstep \estep\bs\estep W}\xqvar
&= (q+1)F_{U\nstep\bs\nstep V\estep\bs\nstep\estep W}\xqvar
- q F_{U\nstep\bs\nstep V\estep\bs\estep\nstep W}\xqvar.
\end{eq*}
If the bounce decomposition of $P$ at $(x,z)$ is given 
by $P=U\nstep\estep\bs\nstep V\nstep\estep\bs\estep W$ then
\begin{eq*}
&F_{U{\estep\nstep\bs\nstep}V{\estep\nstep\bs\estep}W}\xqvar
+
F_{U{\nstep\estep\bs\nstep}V{\nstep\estep\bs\estep}W}\xqvar
+
F_{U{\nstep\bs\nstep\estep}V{\estep\bs\estep\nstep}W}\xqvar = \\
&F_{U{\estep\nstep\bs\nstep}V{\nstep\estep\bs\estep}W}\xqvar
+
F_{U{\nstep\estep\bs\nstep}V{\estep\bs\estep\nstep}W}\xqvar
+
F_{U{\nstep\bs\nstep\estep}V{\estep\bs\nstep\estep}W}\xqvar.
\end{eq*}
\end{enumerate}
\end{cor}

It is a major open problem to prove that the chromatic quasisymmetric 
functions $\chrom_P(\xvec;q)$ expand positively in the $\elementaryE$-basis.
This is known as the Shareshian--Wachs conjecture~\cite[Conj.~4.9]{ShareshianWachs2012}, 
and for $q=1$ it reduces to the 
Stanley--Stembridge conjecture~\cite{StanleyStembridge1993,Stanley1995}.
Some special cases have been solved, for 
example in~\cite{ChoHuh2019,HaradaPrecup2019,HuhNamYoo2020,FoleyHoangMerkel2019,
DahlbergWilligenburg2018,Dahlberg2018x,DahlbergFoleyWilligenburg2020} using different methods.
\reft{LLTe} together with \refl{plethRelation} implies the following expansion.

\begin{cor}{chrom_e}
For all Dyck paths $P\in\DYCK{}$  we have
\begin{eq}{chrom_e}
\chrom_P(\xvec;q)
=
\sum_{\theta \in \ORIENTS(P)} (q-1)^{\asc(\theta)-n} 
\elementaryE_{\hrvpp(\theta)}[\xvec(q-1)].
\end{eq}
\end{cor}
Why the right-hand side in \refq{chrom_e} is $\elementaryE$-positive remains very much unclear.

We remark that vertical-strip LLT polynomials do not seem to have a natural pendent in the world of chromatic symmetric functions in general.
In principle one could use \refq{plethRelation} also when $P$ is a Schr{\"o}der path, however, the resulting functions are not even monomial positive.

\subsection{Hessenberg varieties}\label{Section:hessenberg}

Hessenberg varieties have been studied since the work of F.~De Mari, C.~Procesi, and M.~Shayman~\cite{MariProcesiShayman1992}.
Following up a conjecture of J.~Shareshian and M.~Wachs, recently unicellular LLT polynomials and chromatic quasisymmetric functions of unit-interval graphs have been connected to the equivariant cohomology rings of regular semisimple Hessenberg varieties by P.~Brosnan and T.~Chow, and independently by M.~Guay-Paquet~\cite{BrosnanChow2018,GuayPaquet2016x}.
Here we follow more or less the notation of \cite{GuayPaquet2016x}.

A \defin{flag} in $\setC^n$ is a sequence $F_{\bullet}
\coloneqq (F_1,F_2,\dotsc,F_n)$ of subspaces
\begin{eq*}
\{0\}\subset F_1\subset F_2\subset\dots\subset F_n=\setC^n
\end{eq*}
such that $\dim F_{i}=i$.
The \defin{full flag variety} $\mathdefin{\Flag(\setC^n)}$ is defined as the set of all flags in $\setC^n$. 
Each Dyck path $P\in\DYCK{n}$ is readily identified with a \defin{Hessenberg function}, that is, a function $h:[n]\to[n]$ such that $h(i)\geq i$ for all $i\in[n]$, and $h(i+1)\geq h(i)$ for all $i\in[n-1]$.
Let $\mathdefin{M}\in\GL(n,\setC)$ be a diagonal matrix with $n$ distinct eigenvalues.
The \defin{regular semisimple Hessenberg variety of type $A$} indexed by $M$ and $P$ is the subvariety of $\Flag(\setC^n)$ given by
\begin{eq*}
\mathdefin{\Hess(M,P)}
\coloneqq
\{F_{\bullet}\in\Flag(\setC^n):MF_i\subseteq F_{h(i)}\text{ for all }i\in[n]\}
.
\end{eq*}
Let $\mathdefin{T}\cong(\setC^*)^n$ be the group of invertible diagonal matrices in $\GL(n,\setC)$.
Since the elements of $T$ commute with $M$ we have an action of $T$ on $\Hess(M,P)$ and can consider the \defin{equivariant cohomology ring} $\mathdefin{H_T^*(M,P)}\coloneqq\mathdefin{H_T^*(\Hess(M,P))}$.
The reader is referred to~\cite{Tymoczko2005,Tymoczko2008b} for excellent accessible expositions of this topic.
The ring $H_T^*(M,P)$ is isomorphic to a subring of the direct product of polynomial rings
\begin{eq*}
\mathdefin{A}\coloneqq
\prod_{w\in\symS_n}\setC[x_1,\dots,x_n]
=\{(f_w)_{w\in\symS_n}:f_w\in\setC[x_1,\dots,x_n]
\text{ for all }w\in\symS_n\}
\end{eq*}
whose elements satisfy divisibility conditions imposed by the edges of the unit-interval graph $\uig_P=(V,E)$.
To this end define the \defin{moment graph} of $P$ as the directed edge-labeled graph with vertex set $\symS_n$ and a directed edge $w\overset{(i,j)}{\longrightarrow}w(i,j)$ if and only if $i<j$, $w_i<w_j$, and $\edge{i}{j}\in E$.
Then
\begin{eq*}
H_T^*(M,P)
\cong
\left\{
(f_{w})_{w\in\symS_n}\in A:
(x_{w_i}-x_{w_j})|(f_{w}-f_{w'})
\text{ whenever }w\overset{(i,j)}{\longrightarrow}w'
\right\}
.
\end{eq*}
See also \cite{Teff2013} for additional details.
The symmetric group $\symS_n$ acts on $H_T^*(M,P)$ by
\begin{eq*}
u\cdot(f_w)_{w\in\symS_n}
\coloneqq
(u\cdot f_{u^{-1}w})_{w\in\symS_n}
.
\end{eq*}
This action respects the grading.
There are two natural ways to embed $\setC[x_1,\dots,x_n]$ into $H_T^*(M,P)$, namely
\begin{eq*}
\mathdefin{L}&\coloneqq
\{(f)_{w\in\symS_n}:f\in\setC[x_1,\dots,x_n]\}
\quad\text{and}
\\
\mathdefin{R}&\coloneqq
\{(w\cdot f)_{w\in\symS_n}:f\in\setC[x_1,\dots,x_n]\}
.
\end{eq*}
One can show that $H_T^*(M,P)$ is a free module over $L$ and over $R$.
This reflects the fact that $H_T^*(M,P)$ is a free module over $H_T^*(\point)$, the equivariant cohomology ring of a point.
Note that the action of $\symS_n$ fixes $L$ as a set, and fixes $R$ pointwise.
Define the \defin{graded Frobenius series} of $H_T^*(M,P)$ over $L$ as
\begin{eq*}
\mathdefin{\Frob(H_T^*(M,P),L;\xvec,q)}
\coloneqq
\sum_{w\in\symS_n}\trace(w,H_T^*(M,P),L;q)\;
\powerSum_{\lambda(w)}(\xvec)
,
\end{eq*}
where $\mathdefin{\lambda(w)}$ denotes the cycle type of $w$, and
\begin{eq*}
\mathdefin{\trace(w,H_T^*(M,P),L;q)}
\coloneqq
\sum_{b\in B}(\text{coefficient of }b\text{ in }w\cdot b)\;q^{\deg(b)}
,
\end{eq*}
where $B$ is a homogeneous basis of $H_T^*(M,P)$ over $L$.
Then the chromatic quasisymmetric function satisfies
\begin{eq}{chromatic_Hessenberg}
\chrom_P(\xvec;q)
=
\omega\Frob(H_T^*(M,P),L;\xvec,q)
.
\end{eq}
This was conjectured in \cite{ShareshianWachs2012} and proven in \cite{BrosnanChow2018}.
An independent proof was obtained by M.~Guay-Paquet.
Moreover it is shown in \cite[Lem.~158]{GuayPaquet2016x} that
\begin{eq}{LLT_Hessenberg}
\LLT_P(\xvec;q)
=
\Frob(H_T^*(M,P),R;\xvec,q)
.
\end{eq}
The results obtained in the present paper therefore imply relations 
among $\symS_n$-modules built from the modules $H_T^*(M,P)$ for various Dyck paths $P$.
For example consider the six-term relation in \refti{dyckRecursion}{bounce}.
It follows that there exists an isomorphism of graded $\symS_n$-modules
\begin{eq*}
&H_T^*(M,U{\estep\nstep\bs\nstep}V{\estep\nstep\bs\estep}W)
\oplus
H_T^*(M,U{\nstep\estep\bs\nstep}V{\nstep\estep\bs\estep}W)
\oplus
H_T^*(M,U{\nstep\bs\nstep\estep}V{\estep\bs\estep\nstep}W) \cong \\
&H_T^*(M,U{\estep\nstep\bs\nstep}V{\nstep\estep\bs\estep}W)
\oplus
H_T^*(M,U{\nstep\estep\bs\nstep}V{\estep\bs\estep\nstep}W)
\oplus
H_T^*(M,U{\nstep\bs\nstep\estep}V{\estep\bs\nstep\estep}W)
.
\end{eq*}
It would be interesting to know whether such relations 
can be understood directly or provide some insight in the geometry of these representations.
Note that the relations considered in \cite{BrosnanChow2018,HaradaPrecup2019}
are not linear in the sense that they involve chromatic quasisymmetric
functions indexed by Dyck paths of different sizes.

\subsection{Circular unit-interval digraphs}

The class of Dyck paths and unit-interval graphs can be extended to so called 
\defin{circular Dyck paths} and \defin{circular unit-interval graphs},
see \cite{AlexanderssonPanova2016,Ellzey2016,Ellzey2017}.
There are corresponding chromatic quasisymmetric functions 
and analogs of vertical-strip LLT polynomials in this setting.
Modifying the notion of highest reachable vertex slightly, 
there is an analog of \refq{conjFormula} in this extended setting.
It remains to prove $\elementaryE$-positivity for circular unit-interval graphs.
The positivity in the power-sum basis\footnote{After applying the involution $\omega$.} 
has been proved for the circular unit-interval graphs,
both for chromatic quasisymmetric functions and the vertical-strip LLT polynomials,
via a uniform method in \cite{AlexanderssonSulzgruber2018}.
Note that the plethystic relationship in \refl{plethRelation} 
no longer holds in the circular setting.

\subsection{Other open problems}
We conclude this paper with three additional open problems.

\begin{prob}{unimodality}
Consider the coefficients $a_\mu(q) \in \setN[q]$ defined by
\begin{eq*}
\LLT(\xvec;q+1) = \sum_{\mu} a_\mu(q) \elementaryE_\mu(q).
\end{eq*}
It seems that $a_\mu(q)$ is unimodal with mode $\lfloor \mu_1 / 2 \rfloor$. 
Moreover, if $a_\mu(q) = b_0+b_1q + \dotsb + b_\ell q^\ell$,
then for all $j\in \{1,2,\dotsc,\ell-1\}$, we seem to have
$
 b_{j-1}b_{j+1} \leq b_j^2
$. In other words, the coefficients of $a_\mu(q)$ seem to be \defin{log-concave}.
\end{prob}
A similar conjecture is stated in \cite[Conj.~25]{AlexanderssonPanova2016}.
Some partial results in this direction are given 
for example in \cite[Cor.~6.2]{Ellzey2017} and \cite[Thm.~4.9]{HuhNamYoo2020}.
See also \cite[Conj.~53]{AlexanderssonPanova2016} and \cite[Conj.~7.13]{AlexanderssonSulzgruber2018} for related conjectures on the $\omega\powerSum$-unimodality of certain LLT polynomials.

\begin{prob}{nonSymVersion}
In \cite{HaglundWilson2017,TewariWilsonZhang2019x} a non-symmetric extension
of the chromatic quasisymmetric functions is considered. 
It is natural to ask whether our recurrences extend to this setting.
\end{prob}

\begin{prob}{nonCommVersion}
There is a notion of non-commutative unicellular LLT polynomials introduced in \cite{NovelliThibon2019}.
Can one extend the $\elementaryE$-positivity result to the non-commutative setting?
\end{prob}

\subsection*{Acknowledgements}

Decisive parts of the research presented in this paper were conducted during a stay of the second named author at Institut
Mittag-Leffler in Djursholm, Sweden in the course of the program on Algebraic and Enumerative Combinatorics in Spring 2020.
We thankfully acknowledge the support of the Swedish Research
Council under grant no.~2016-06596, and thank Institut~Mittag-Leffler for its hospitality.
The authors would also like to thank J.~Haglund and A.~Morales for helpful discussions.

\section*{Appendix: Tables}
This appendix contains two tables that list all possible types of orientations that appear in the proofs of \reft{orientations:bounceI} and \reft{orientations:bounceIII}, and show how they are matched.
The following example illustrates how to read these tables.

\begin{exa}{use_table}
Consider the following orientation of the Schr{\"o}der path $P=\nstep\nstep\estep\nstep\nstep\nstep\estep\nstep\dstep\estep\nstep\estep\estep\estep\estep$ and let $(x,z)=(3,7)$.
\begin{eq*}
\theta
=
\begin{tikzpicture}
\fill[lightgray](0,2)--(1,2)--(1,5)--(2,5)--(2,6)--(3,6)--(3,7)--(4,7)--(4,8)--(0,8)--cycle;
\fill[yellow]
(0,0)rectangle(1,1)rectangle(2,2)rectangle(3,3)rectangle(4,4)rectangle(5,5)rectangle(6,6)rectangle(7,7)rectangle(8,8);
\draw
(0,1)--(1,1)--(1,8)
(0,2)--(2,2)--(2,8)
(0,3)--(3,3)--(3,8)
(0,4)--(4,4)--(4,8)
(0,5)--(5,5)--(5,8)
(0,6)--(6,6)--(6,8)
(0,7)--(7,7)--(7,8)
(0,0)--(1,0)--(1,1)--(2,1)--(2,2)--(3,2)--(3,3)--(4,3)--(4,4)--(5,4)--(5,5)--(6,5)--(6,6)--(7,6)--(7,7)--(8,7)--(8,8)--(0,8)--cycle;
\draw[red,->,thickLine](3,7)--(3,3)--(1,3);
\draw[entries]
(1,1)node{$1$}
(2,2)node{$2$}
(3,3)node{$3$}
(4,4)node{$4$}
(5,5)node{$5$}
(6,6)node{$6$}
(7,7)node{$7$}
(8,8)node{$8$}
(3,7)node{$\to$}
(2,3)node{$\to$}
(2,5)node{$\to$}
(3,4)node{$\to$}
(3,5)node{$\to$}
(4,6)node{$\to$}
(5,8)node{$\to$}
(6,7)node{$\to$}
;
\end{tikzpicture}
\qquad\mapsto\qquad\theta'=
\begin{tikzpicture}
\fill[lightgray](0,2)--(1,2)--(1,5)--(2,5)--(2,6)--(4,6)--(4,8)--(0,8)--cycle;
\fill[yellow]
(0,0)rectangle(1,1)rectangle(2,2)rectangle(3,3)rectangle(4,4)rectangle(5,5)rectangle(6,6)rectangle(7,7)rectangle(8,8);
\draw
(0,1)--(1,1)--(1,8)
(0,2)--(2,2)--(2,8)
(0,3)--(3,3)--(3,8)
(0,4)--(4,4)--(4,8)
(0,5)--(5,5)--(5,8)
(0,6)--(6,6)--(6,8)
(0,7)--(7,7)--(7,8)
(0,0)--(1,0)--(1,1)--(2,1)--(2,2)--(3,2)--(3,3)--(4,3)--(4,4)--(5,4)--(5,5)--(6,5)--(6,6)--(7,6)--(7,7)--(8,7)--(8,8)--(0,8)--cycle;
\draw[entries]
(1,1)node{$1$}
(2,2)node{$2$}
(3,3)node{$3$}
(4,4)node{$4$}
(5,5)node{$5$}
(6,6)node{$6$}
(7,7)node{$7$}
(8,8)node{$8$}
(4,7)node{$\to$}
(2,4)node{$\to$}
(2,5)node{$\to$}
(4,5)node{$\to$}
(3,6)node{$\to$}
(5,8)node{$\to$}
(6,7)node{$\to$}
;
\end{tikzpicture}
\end{eq*}
The bounce decomposition of $P$ at $(x,z)$ is $P=U{\nstep\bs\nstep}V{\dstep\bs\estep}W$ 
where $U=\nstep\nstep\estep$, $V=\nstep\estep\nstep$ and $W=\nstep\estep\estep\estep\estep$.
Therefore we need to consult Table~\ref{Table:orientations:bounceI} below.
We have $v=1$, $w=2$, $y=4$, $Q=U{\nstep\bs\nstep}V{\estep\bs\dstep}W$, 
and importantly $\hrv{\theta}{x}=8$, $\hrv{\theta}{y}=7$, and $\hrv{\theta}{z}=7$.
Thus we need to consider the weak standardisation of $877$ which is $211$.
Table~\ref{Table:orientations:bounceI} shows that there 
are four subcases in the row indexed by $\sigma=211$.
That is, all four orientations of the edges $\edge{x}{y}$ and $\edge{y}{z}$ could lead to this word.
In our case $\dedge{x}{y},\dedge{z}{y}\in\theta$.
Thus we are in the second subcase, which belongs to \starMap{5}.
This means we should apply the map $\psi$ of \refl{II:psi} and 
end up in case \starMap{5} in the column corresponding to $Q$.
In accordance with what the table predicts, the resulting orientation $\theta'$ 
contains $\dedge{y}{x}$, and the weak 
standardization of $\hrv{\theta'}{x},\hrv{\theta'}{y}$, $\hrv{\theta'}{z}$ is $121$.
Since $\theta$ has one more descent than $\theta'$ we should regard $\theta'$ as 
the second subcase of \starMap{5} (in column $Q$), which is multiplied by $q$.
The reader can check that if $\psi$ is applied to an orientation 
from the first subcase of \starMap{5} (in column $P$), then the number of ascents is preserved.
\end{exa}

\FloatBarrier

\renewcommand{\arraystretch}{0.6}

\begin{table}
\small
\begin{tabular}[ht]{c ccccc c ccccc }
\toprule
 &\multicolumn{5}{c}{$P$} & \phantom{space} & \multicolumn{5}{c}{$Q$}\\
 \cmidrule{2-6} \cmidrule{8-12} 
$\sigma$&$xy$&$yz$& & &Map &  &$q^*$&$xy$& & & Map \\
\midrule
$111$&$\dn$&$\dn$&$ $&$ $&\starMap{1}&\phantom{ }&1&$\dn$&$ $&$ $&\starMap{1}\\
&$\to$&$\dn$&$ $&$ $&&\phantom{ }&$q$&$\dn$&$ $&$ $&\\
&$\dn$&$\to$&$ $&$ $&\diamondMap{2}&\phantom{ }&1&$\to$&$ $&$ $&\diamondMap{2}\\
&$\to$&$\to$&$ $&$ $&&\phantom{ }&$q$&$\to$&$ $&$ $&\\

\midrule
$122$&&&&&&\phantom{ }&$1$&$\dn$&$ $&$ $&\starMap{3}\\
&&&&&&\phantom{ }&$q$&$\dn$&$ $&$ $&\\

\noalign{\vskip 3mm}
$212$&$\dn$&$\dn$&$ $&$ $&\starMap{3}\\
&$\dn$&$\to$&$ $&$ $\\

\midrule
$211$&$\dn$&$\dn$&$ $&$ $&\starMap{5}&\phantom{ }&$1$&$\dn$&$ $&$ $&\starMap{6}\\
&$\to$&$\dn$&$ $&$ $&&\phantom{ }&$q$&$\dn$&$ $&$ $&\\
&$\dn$&$\to$&$ $&$ $&\diamondMap{4}&\phantom{ }&$1$&$\to$&$ $&$ $&\diamondMap{4}\\
&$\to$&$\to$&$ $&$ $&&\phantom{ }&$q$&$\to$&$ $&$ $&\\
\noalign{\vskip 3mm}
$121$&$\dn$&$\dn$&$ $&$ $&\starMap{6}&\phantom{ }&$1$&$\dn$&$ $&$ $&\starMap{5}\\
&$\to$&$\dn$&$ $&$ $&&\phantom{ }&$q$&$\dn$&$ $&$ $&\\
\midrule
$221$&$\dn$&$\dn$&$ $&$ $&\beadMap{7}&\phantom{ }&$1$&$\dn$&$ $&$ $&\beadMap{7}\\
&$\to$&$\dn$&$ $&$ $&&\phantom{ }&$q$&$\dn$&$ $&$ $&\\
&$\dn$&$\to$&$ $&$ $&&\phantom{ }&$1$&$\to$&$ $&$ $&\\
&$\to$&$\to$&$ $&$ $&&\phantom{ }&$q$&$\to$&$ $&$ $&\\

\midrule
$132$&&&&&&\phantom{ }&$1$&$\dn$&$ $&$ $&\starMap{8}\\
&&&&&&\phantom{ }&$q$&$\dn$&$ $&$ $&\\

\noalign{\vskip 3mm}
$312$&$\dn$&$\dn$&$ $&$ $&\starMap{8}\\
&$\dn$&$\to$&$ $&$ $\\
\midrule
$231$&$\dn$&$\dn$&$ $&$ $&\beadMap{9}&\phantom{ }&$1$&$\dn$&$ $&$ $&\beadMap{9}\\
&$\to$&$\dn$&$ $&$ $&&\phantom{ }&$q$&$\dn$&$ $&$ $&\\
\midrule
$321$&$\dn$&$\dn$&$ $&$ $&\beadMap{10}&\phantom{ }&$1$&$\dn$&$ $&$ $&\beadMap{10}\\
&$\to$&$\dn$&$ $&$ $&&\phantom{ }&$q$&$\dn$&$ $&$ $&\\
&$\dn$&$\to$&$ $&$ $&&\phantom{ }&$1$&$\to$&$ $&$ $&\\
&$\to$&$\to$&$ $&$ $&&\phantom{ }&$q$&$\to$&$ $&$ $&\\
\bottomrule
\end{tabular}
\caption{Subcases for orientations, first bounce relation.
The missing $\sigma \in \{112,123,213\}$ are not possible, 
since the edge $\dedge{x}{z}$ is present in $P$ and $\dedge{y}{z}$ is present in $Q$.
The maps marked with $\starMap{i}$ are covered by \refl{II:switch},
while $\diamondMap{i}$ and $\beadMap{i}$ are covered by \refl{II:changeb}
and \refl{II:id}, respectively.
}\label{Table:orientations:bounceI}
\end{table}

\begin{table}
\small
\begin{tabular}[ht]{c ccccc c ccccc }
\toprule
 &\multicolumn{5}{c}{$P$} & \phantom{space} & \multicolumn{5}{c}{$Q$}\\
 \cmidrule{2-6} \cmidrule{8-12} 
$\sigma$&$xy$&$yz$& & &Map &  &$xy$&$wx$& & &Map \\

\midrule
$111$&$\dn$&$\dn$&$ $&$ $&\starMap{1}& \phantom{-}  &$\dn$&$\dn$& $ $&$ $&\starMap{1}\\
&$\dn$&$\to$&$ $&$ $&&\phantom{-}&$\dn$&$\to$&$ $&$ $&\\
&$\to$&$\dn$&$ $&$ $&&\phantom{-}&$\to$&$\dn$&$ $&$ $&\\
&$\to$&$\to$&$ $&$ $&&\phantom{-}&$\to$&$\to$&$ $&$ $&\\
\midrule
$122$&&&&&&\phantom{-}&$\dn$&$\dn$&$ $&$ $& \starMap{2}\\
&&&&&&\phantom{-}&$\dn$&$\to$&$ $&$ $\\
\noalign{\vskip 3mm}
$212$&$\dn$&$\dn$&$ $&$ $&\starMap{2}\\
&$\to$&$\dn$&$ $&$ $\\

\midrule
$211$&$\dn$&$\dn$&$ $&$ $&\starMap{3}&\phantom{-}&$\dn$&$\dn$&$ $&$ $&\starMap{5}\\
&$\to$&$\dn$&$ $&$ $& &\phantom{-}&$\to$&$\dn$&$ $&$ $& \\
&$\dn$&$\to$&$ $&$ $&\beadMap{4}&\phantom{-}&$\dn$&$\to$&$ $&$ $&\beadMap{4}\\
&$\to$&$\to$&$ $&$ $& &\phantom{-}&$\to$&$\to$&$ $&$ $& \\
\noalign{\vskip 3mm}
$121$&$\dn$&$\dn$&$ $&$ $&\starMap{5}&\phantom{-}&$\dn$&$\dn$&$ $&$ $&\starMap{3}\\
&$\dn$&$\to$&$ $&$ $&&\phantom{-}&$\dn$&$\to$&$ $&$ $\\
\midrule
$221$&$\dn$&$\dn$&$ $&$ $&\beadMap{6}&\phantom{-}&$\dn$&$\dn$&$ $&$ $&\beadMap{6}\\
&$\to$&$\dn$&$ $&$ $&&\phantom{-}&$\to$&$\dn$&$ $&$ $\\
&$\dn$&$\to$&$ $&$ $&&\phantom{-}&$\dn$&$\to$&$ $&$ $\\
&$\to$&$\to$&$ $&$ $&&\phantom{-}&$\to$&$\to$&$ $&$ $\\

\midrule
$132$&&&&&&\phantom{-}&$\dn$&$\dn$&$ $&$ $&\starMap{7}\\
&&&&&&\phantom{-}&$\dn$&$\to$&$ $&$ $\\
%
%
\noalign{\vskip 3mm}
$312$&$\dn$&$\dn$&$ $&$ $&\starMap{7}\\
&$\to$&$\dn$&$ $&$ $\\
\midrule
$231$&$\dn$&$\dn$&$ $&$ $&\starMap{8}&\phantom{-}&$\dn$&$\dn$&$ $&$ $&\starMap{9}\\
&$\dn$&$\to$&$ $&$ $&&\phantom{-}&$\dn$&$\to$&$ $&$ $\\
\noalign{\vskip 3mm}
$321$&$\dn$&$\dn$&$ $&$ $&\starMap{9}&\phantom{-}&$\dn$&$\dn$&$ $&$ $&\starMap{8}\\
&$\to$&$\dn$&$ $&$ $& &\phantom{-}&$\to$&$\dn$&$ $&$ $& \\
&$\dn$&$\to$&$ $&$ $&\beadMap{10}&\phantom{-}&$\dn$&$\to$&$ $&$ $&\beadMap{10}\\
&$\to$&$\to$&$ $&$ $& &\phantom{-}&$\to$&$\to$&$ $&$ $& \\
\bottomrule
\end{tabular}
\caption{Subcases for orientations, second bounce relation.
The maps marked with $\beadMap{i}$ correspond to \refl{IV:id} and the maps 
marked with $\starMap{i}$ are covered by \refl{IV:switch}.
The missing $\sigma \in \{112, 123, 213\}$ are not possible,
due to the edges $\dedge{x}{z}$ and $\dedge{y}{z}$ being present in $P$ and $Q$,
respectively.
}\label{Table:orientations:bounceII}
\end{table}

\FloatBarrier

\bibliographystyle{alphaurl}
\bibliography{bibliography}

\end{document}